\documentclass[smallextended]{svjour3}      
\usepackage{geometry}
\usepackage{appendix}
\usepackage{subfigure}
\usepackage{amsmath}
\usepackage{amssymb}
\usepackage{mathtools}
\usepackage{multirow, makecell}
\usepackage{float}
\usepackage{color}
\usepackage{tikz}
\usepackage{xypic}
\usepackage{diagbox}

\usetikzlibrary{shapes.geometric, arrows}
\tikzstyle{startstop} = [rectangle, rounded corners, minimum width = 0.9cm, minimum height=0.9cm,text centered, draw = black, 	align=center, fill = red!40]
\tikzstyle{io} = [trapezium, trapezium left angle=70, trapezium right angle=110, minimum width=0.9cm, minimum height=0.9cm, text centered, draw=black, fill = blue!40]
\tikzstyle{process} = [rectangle, minimum width=1.8cm, minimum height=0.9cm, text centered, draw=black, 	align=center, fill = yellow!50]
\tikzstyle{process2} = [rectangle, minimum width=1.8cm, minimum height=0.9cm, text centered, draw=black, fill = green!50]
\tikzstyle{arrow} = [->,>=stealth]
\newtheorem{wldefinition}[theorem]{Definition} 
\newtheorem{wllemma}[theorem]{Lemma}

\newtheorem{wlremark}[theorem]{Remark}
\newtheorem{wlproperty}[theorem]{Property}

\usepackage{thmtools}

\begin{document}
\title{Regular simplex tensors: optimization landscape, conjecture proof, and beyond}
\author{Lei  Wang}

\institute{
	This work was  supported by National Natural Science Foundation   of China (62401088).
	\\
	Lei  Wang
	\at
	the School of Microelectronics and Communication Engineering, Chongqing University, Chongqing 400044, China.
	\and
	Lei  Wang
	\at
	\email{wanglei179@mails.ucas.ac.cn  \\  Corresponding author  }
}

\date{Received: date / Accepted: date}

\maketitle

\begin{abstract}
The  concept of   tensor  eigenpairs    has attracted increasing research attention   in  the  past  decades.
Recent works have focused on a special class termed   regular simplex tensors, which are constructed from an equiangular tight frame of  $n$ vectors in  $(n-1)$-dimensional space 	 for  $n \ge 3$ and  order   $m \ge 3$.
Existing  works focus on   analyzing  the  robustness of eigenpairs obtained by the  tensor power method.
 At the end of that work, a conjecture was made  that   if  they  exist,  the only robust eigenvectors of a regular simplex tensor,    up to sign equivalence,  are the vectors in the    regular simplex   frame.
A subsequent  study 
 theoretically proved that this  holds for  the simplest  triangle  case  where $n=3$.
However,  
 for    cases with higher $n$,  the process becomes  complicated   in both  checking   all eigenpairs    and  determining the explicit formula   for   the  robustness criterion.
In this paper,	to deal with this issue,	a connection  between  robust and locally optimal eigenpairs is built, recognizing the latter as another pivotal concept in the field of optimization.
Then, we turn to  checking   the   local optimality of   all  eigenpairs,
for which  we have  developed   an efficient  model   with
a favorable structure that facilitates the enumeration of all eigenpairs and delineates the optimization landscape for the model.
Then, 	integrating the   two advances   enables us to
narrow the  scope of  the  robust  eigenpairs   to 
 those
locally maximized ones, which are exactly the vectors in the  regular simplex  frame.
Finally, the proof of the conjecture reduces to only  examining  
the   vectors in the  frame, whose robustness can be easily checked for  any higher $n$  and $m$. 
	This work shows that, up to sign equivalence,  
	 excluding 
 the exceptional cases $(m,n)=(3,3)/(3,4)/(4,3)$ where  no robust eigenpairs exist, 
	 the only robust eigenvectors of a regular simplex tensor are 
	 	 the  vectors in the regular simplex frame. 
	We   also   discuss the limitation of regular simplex tensors
	and  explore   a  promising  generalization by  fusing  
	 the Dirichlet distribution on the underlying  simplex  support set.
 A
 provable guarantee and   potential application  deserve
	future study.



\end{abstract}

\section{Introduction}
Tensor   analysis and   applications  have been 
investigated
in recent years \cite{kolda,2007Numerical,RobustTensorCompletion,TensorPCA}.
Among them,
the concept of
tensor eigenpairs
serves as an important  subject,
which naturally extends the concept of matrix eigenpairs to the tensor setting, and
was originally proposed
in \cite{lim,qi}.
This concept
has been widely exploited  in  theory \cite{SHOPM,NCM,OTD,Cuicf,hm}  and 
has also found numerous applications
in  many disciplines, such as latent  variable models  \cite{Hsu}, hyperspectral  remote  sensing  image processing \cite{NPSA,MSDP},
signal processing \cite{tensor_bss1},  to name  a few. 

Different  from  the matrix  case,    there are several  definitions for eigenpairs  of 
a  symmetric tensor, such as   D-eigenpairs \cite{Deigen}, H-eigenpairs  \cite{qi},  E/Z-eigenpairs  \cite{qi}.
In this paper,  we mainly  focus  on  one of them ---
Z-eigenpairs
whose eigenvector  and eigenvalue are  over  the  real field and the vector  is of
unit length, which can be viewed as a natural generalization of the matrix eigenpair definition.
Without causing ambiguity, the term  eigenpairs will always refer to Z-eigenpairs by default. 
Similar to the matrix case, there exists an equivalence 
relation for tensor eigenpairs \cite{hm}. Briefly speaking, $(\lambda, \mathbf{v})$ is a Z-eigenpair of tensor  $\mathcal{S}$ if and only if $((-1)^{m-2}\lambda, -\mathbf{v})$ is also a Z-eigenpair of $\mathcal{S}$, where $m \ge 3 $ is the order of the tensor. Conventionally, only one representative from each equivalence class is selected for analysis and for counting the number of eigenpairs—a practice widely adopted in previous works \cite{hm,RobustEigen,Cuicf,upperbound}. We follow this convention in our discussion,
and  all eigenpairs and their total number
are presented   up to sign equivalence. 
Please refer to the  references for further details.
A tensor eigenpair problem
can  be treated  as a generic system of polynomial equations,
and all sorts of multivariate root-finding algorithms and  computer algebra techniques can be used to obtain its solutions,
to name a few \cite{NAClab,NCM,SDPrank1}.
However,  it  has  also  been   shown  that  most of 
the tensor  problems  are 
NP-hard  \cite{NP-hard},  including computing all  eigenpairs of  a  tensor.
To the best of our knowledge,
so far,  there are  only  two   algorithms that  can  obtain  all  eigenpairs of  a  tensor \cite{Cuicf,hm},
which can also fall into the scope of polynomial optimization.
Due to the NP-hardness, 
most   previous works only  aim
to obtain 
one maximized or minimized  eigenpair,
which is a 
relatively  easy task.
Various optimization algorithms have been developed.
Among them,
one of the
classical algorithms
is called   the  tensor power method (TPM) \cite{HOPM,hopm2002},  which is based on a  fixed-point scheme.
Shifted, adaptive and 
accelerated
versions   of
this method  were  also proposed  later  in  \cite{SHOPM,ASHOPM,AcceleratingHOPM}.
The  
fixed  points (which may be more than one)
of   TPM
are the eigenvectors of   a  tensor.
An important  issue concerning  TPM
is to
investigate
whether
the   obtained  eigenvector  is  robust or not,  which,  for the real solutions, can be determined by checking the spectral radius of the corresponding
Jacobian matrix (See Subsection \ref{robustsection} for    the detailed  definition).
Concerning  this problem,
previous works have
mainly  paid 
attention
to some   special  classes  of  symmetric  tensors,
and  one   widely researched   type
is termed
orthogonally  decomposable  (odeco) tensors \cite{odst,OTD,Hsu,sr1,cunmu,GloballyConvergent},
which
is
a
natural generalization of 
orthogonally decomposable
matrix decomposition.

However,
unfortunately,
most  symmetric tensors cannot be
orthogonally  decomposable.
In  this sense,
recently,
some researchers further  extended
the case of  odeco tensors  to
a  more generalized one, where the  symmetric  tensor  is
constructed     from
the   equiangular  set (ES) or equiangular tight frame (ETF), which contains
$r$ vectors in
$(n-1)$-dimensional space \cite{RobustEigen}.
For  example,
the case of  $r=n-1$  serves as a  special  case  of   the  ETF,  which
forms a  standard orthonormal  basis and  corresponds to
the  odeco    tensor.
When $r=n$,  the
frame is termed the regular simplex  frame,
and the  generated tensor is thus called
the regular simplex tensor, the core concept we tend to discuss in this work. 
In  \cite{RobustEigen},
they  discussed that under what   conditions
the eigenvectors  are robust,
and
studied
the  special  tensor  types     generated by ES or ETF with  
theoretical  proof.

Furthermore,
at  the end of that work \cite{RobustEigen},  a  conjecture  was     made  that the only  robust eigenvectors of a regular simplex tensor  are
the
vectors in the ETF.
See Conjecture 4.8 in  \cite{RobustEigen} for details.
We provide a more precise presentation  of
this conjecture
by  discussing its existence and uniqueness, 
which is restated as a formal    theorem   \ref{conjecturesimplex}   in Subsection \ref{conjecture}.
Later,
several researchers
further
focused  on  this conjecture \cite{teneigenstructure}.
However,
they only provided the robustness proof for the simplest  case of  $n=3$,
and the experimental   justification for  $n=4$.
For  larger
$n$,  it
is a tough task to prove,
due to the  increasing number of   all    eigenpairs  to be checked,
and   the   difficulty of obtaining explicit formula for the Jacobian matrix 
for  most  eigenpairs.

In this paper,  building upon  the only two  existing   works \cite{RobustEigen,teneigenstructure}
and to deal with the  difficulties,
we proceed further and provide a complete proof   concerning this  conjecture
and
the main contributions of this paper are
concluded as  follows:

\begin{itemize}
	\item
	The connection between robust
	and locally  optimal  eigenpairs
	is   investigated,
	which  enables us to simplify the 
	proof strategy adopted in \cite{teneigenstructure}.
	In detail, one does not need to check
	the robustness of  all eigenpairs, and
	it suffices   to only   check whether   the  locally maximized eigenpairs
	are robust or not.
	
	\item
	Though identifying   the  local optimality of all eigenpairs
	is also tough,
	we develop an efficient and equivalent
	model
	where  local optimality
	can be  discussed with a benign structure.
	We provide  the optimization landscape
	concerning the model 
	of  regular simplex tensor eigenpairs,
	and  show
	that only 
	the vectors in the  regular simplex  frame 
	are  locally maximized eigenpairs.
	\item
	With the above conclusions  at hand,
	we only need to check whether 
	all   vectors in the  regular simplex  frame
	are robust or not,
	whose    Jacobian matrix  can be  easily    deduced for   any higher $n$
	and the robustness can be easily  identified by its spectral radius.
	Finally,   the proof for the unsolved conjecture  is  provided.
\end{itemize}

We also discuss the  limitations of the regular simplex tensor,
and  provide a promising generalization, which is worth studying and is  discussed in Section
\ref{futurework}.
In the end, we  clarify a slight  difference
from the previous
works
\cite{RobustEigen,teneigenstructure}  in    the notation  adopted   in this paper.
The existing ones 
consider   $n+1$ vertices in $n $-dimensional space, while  here we
considered  $n$ vertices in $(n-1)$-dimensional space.
Such a  modification is adopted
for convenience in formula derivations
in the later  section (see Section  \ref{reformulated} for details), which is  trivial, 
and we wish that this minor adjustment  does
not cause any ambiguity  compared to previous works.

The  rest of the paper
is organized as  follows.
In Section  \ref{pre},
some preliminaries related to the subject are provided,  including
the  definition   of
tensor  eigenpairs,    the  focused regular simplex tensor, and the robustness criterion.
In  Section  \ref{Relationonrobust},
the connection between robust and 
locally  optimal eigenpairs
is investigated, providing an auxiliary  criterion for robustness checking.
Section  \ref{optland}
present   the optimization landscape of 
the model related to  eigenpairs of   the  regular simplex tensor, and
its detailed proof
is deferred to Section \ref{sec.proof.opt}.
Section \ref{sec.Conjecture}
completes the proof    of   the focused conjecture.
Limitations   
and 
promising  directions
for future work are discussed in  Section  \ref{futurework}.

\section{Preliminaries}
\label{pre}
We    introduce  some  necessary   notations, definitions,   and  lemmas   used  in   this  article.
In  this  work,    high-order tensors are denoted
by 
boldface Euler script letters, e.g.,
$\mathcal  A $.
Matrices are denoted  by    boldface capital letters, e.g., $\mathbf  A $; vectors are denoted in
boldface lowercase letters, e.g.,  $\mathbf  a $.
Sets and subsets are denoted by  blackboard bold  capital letters, e.g.,  $\mathbb  A $.

An
$m$th-order tensor is denoted
$\mathcal A \in \mathbb {R}^{I_1 \times I_2  \times \dots \times I_{m} }$,
where
$m$
is the order   of
$\mathcal A $, and
$ I_j $ ($  j \in \{ 1,2,\dots,m \}$)  is  the  dimension  of
the $j$th-mode.
The 
elements
of $\mathcal A$,    which  are  indexed  by integer tuples $(i_1,i_2,\dots,i_m) $,  are  denoted
by
$a_{i_1,i_2,\dots,i_m}$
where 
$	1 \le i_1 \le I_1,
\dots,
1 \le i_m \le I_m
$.
A   tensor is called    symmetric if its elements remain invariant under any permutation of
the  indices \cite{kolda}.
Let    $  T^{m}(\mathbb R^{p}) $ 
denote   the  space  of  all  such  real  symmetric    tensors
of order $m$ and dimension $p$.
Given an  $m$th-order  $p$-dimensional symmetric  tensor  $\mathcal S $ and  a  vector $ \mathbf v  \in \mathbb {R}^{p \times 1}$,
we simply denote \cite{Cuicf}
\begin{equation*}
\mathcal S \mathbf  v^{m} :=
\sum\limits_{i_1,i_2,\dots,i_m=1}^{p}
s_{i_1,i_2,\dots,i_m}  v_{i_1} \dots   v_{i_m},
\end{equation*}
and  $   \mathcal S \mathbf v^{m-1}   $   denotes    a
$p$-dimensional
column
vector,  whose  $j$th  element   is
\begin{equation*}
(\mathcal S \mathbf v^{m-1})_{j} =
\sum\limits_{i_2,\dots,i_m =1}^{p}
s_{j,i_2,\dots,i_m}  v_{i_2} \dots   v_{i_m}.
\end{equation*}
Furthermore,   $   \mathcal S \mathbf v^{m-2}   $   is   a
$p  \times  p $  matrix,    whose  $(i, j )$th  element   is
\begin{equation*}
(\mathcal S \mathbf v^{m-2})_{i, j} =
\sum\limits_{i_3,\dots,i_m =1}^{p}
s_{i,j,i_3,\dots,i_m}  v_{i_3} \dots   v_{i_m}.
\end{equation*}

$\triangledown$ is  the  gradient  operator.
$ \operatorname{null}  ( \mathbf A) $
denotes the null space  of $\mathbf A$.
Given   a  matrix
$ \mathbf A  \in \mathbb {R}^{p_1 \times p_2} $   and     assuming  that
$ \mathbf A $  is   of  full   column  rank ($  p_2  \le  p_1$),
the orthogonal  complement   projection matrix of  $ \mathbf A $
is  denoted
$ \mathbf {P}^{\bot}_{  \mathbf A }
=
\mathbf I_{p_1}-\mathbf A (\mathbf A^{\mathrm T} \mathbf A)^{-1} \mathbf A^{\mathrm T}$,
where $ \mathbf I_{p_1} $ is  a  $ {p_1}  \times  {p_1} $ identity matrix.
$\mathbf  1_{p}$   denotes   a  $ p  \times  1$ column vector.
$\mathbf  J_{p_1 \times p_2}$    and  $\mathbf  J_{p}$     are
a   $ p_1  \times  p_2 $  matrix  and   a   $ p  \times  p$    square    matrix,     with all  elements equal to 1,  respectively.
It holds   that  \cite{zhang2017matrix}
\begin{equation}\label{Jproperty}
\mathbf  J_{p_1 \times p_2} = \mathbf  1_{p_1}  \mathbf  1_{p_2}^{\mathrm {T}}.
\end{equation}
Similarly,  $\mathbf  0_{p}$,  $\mathbf  O_{p_1 \times p_1}$    and  $\mathbf  O_{p}$   are   matrices   with all  elements equal to 0.
$\operatorname{diag}(\mathbf{u})$   is an operator which maps the vector  $  \mathbf u  \in  \mathbb R^{n \times 1} $   to
an  $ n \times n $
diagonal matrix 
whose diagonal elements are those of    $  \mathbf u$.
$\circledast$ 
denotes  the  Hadamard product, an  element-wise   multiplication  operation where $   (\mathbf  {A} \circledast  \mathbf  {B})_{ij} =
\mathbf  A_{ij}\mathbf  B_{ij}$.
For simplicity,  we use
$ \mathbf A ^{\circledast^{p}}$
and
$ \mathbf a ^{\circledast^{p}}$
to denote the $p$ -fold   element-wise   multiplication  of the matrix
$ \mathbf A  $ and  of  the vector $ \mathbf a $,  respectively.

\begin{wldefinition}[Spectral  radius]\label{radiusdefinition}
	The  spectral  radius  of    a  square  matrix  $\mathbf B \in \mathbb {R}^{p \times p}$  is  the   maximum  value among  
	the  absolute   values
	of  all    eigenvalues  of
	$\mathbf B$,  denoted  by
	$ \rho (\mathbf B) = \max_i  \vert  \sigma_{i}(\mathbf B) \vert$,
	where
	$ \sigma_{i}(\mathbf B)  ,  i=1,2,\dots, p $
	are   $p$ eigenvalues of
	$\mathbf B$.
\end{wldefinition}

\begin{wldefinition}[The  outer product]
	\label{outerprod}
	Given   $m$  vectors
	$ \mathbf a^{(i) } \in \mathbb {R}^{I_i \times 1}$
	($i=1,2, \dots, m$),
	their   outer  product
	$ \mathbf a^{(1) }
	\circ
	\mathbf a^{(2) }
	\circ  \dots
	\circ
	\mathbf a^{(m) }
	$
	is   an    $ m$th-order  tensor  denoted $   \mathcal A$,   of size
	$ I_1 \times I_2  \times \dots \times I_m  $.
	When
	$  \mathbf a^{(1) }
	=   \dots
	=
	\mathbf a^{(m) }
	=\mathbf a $,  we denote
	$ \mathcal A =  \mathbf a^{\circ m}$  for  simplicity.
\end{wldefinition}

\subsection{Optimization  theories of  Tensor eigenpairs}
In this part,  we briefly  introduce the  optimization theories   related to the tensor  eigenpairs problem.
Given a  general   $m$th-order  $p$-dimensional symmetric  tensor  $\mathcal S $,
consider  the  following   constrained  optimization  model:
\begin{equation}\label{opti_ori}
\begin{cases}
\max\limits_{\mathbf v} \quad \mathcal S \mathbf v^{m} \\
\text{s.t.}  \quad \mathbf v^{\mathrm {T}}\mathbf v=1
\end{cases}.
\end{equation}
The Lagrangian function
for   (\ref{opti_ori})  is  defined as:
\begin{equation}\label{Lagrangian_function}
L(\mathbf v, \lambda) =
\frac {1}  { m}
\mathcal S \mathbf v^{m}
+
\frac { \lambda} {2} (1-  \mathbf v^{\mathrm {T}}\mathbf v).
\end{equation}
When  
$\triangledown_{\mathbf v}  L(\mathbf v, \lambda) =\mathbf 0$, i.e.,   the gradient of
$  L(\mathbf v, \lambda) $  with respect to
$ \mathbf  v $  is  $ \mathbf 0$,
the concept of  the eigenpair of a  symmetric  tensor
can  be  deduced,  which  was  independently    defined  by   Lim  and  Qi  in  2005:

\begin{wldefinition} [Eigenpairs of   a   symmetric tensor \cite{qi,lim}] 
	Given a  tensor $\mathcal S   \in    T^{m}(\mathbb R^{p}) $,
	a pair
	$(\lambda ,\mathbf v )$
	is  an  eigenpair  of
	$\mathcal S  $
	if
	\begin{equation}\label{definition}
	\mathcal S \mathbf v^{m-1}=\lambda \mathbf v,
	\end{equation}
	where
	$ \lambda  \in  \mathbb C $
	is  the  eigenvalue and
	$ \mathbf v  \in   \mathbb C^{p \times  1} $
	is the  corresponding   eigenvector  satisfying
	$\mathbf v^{\mathrm {T}}\mathbf v=1 $.
\end{wldefinition}

\begin{wlremark}\label{rem.signold}
	In  the  related reference \cite{qi}, the solution
	that satisfies the  above definition over  $ \mathbb C $  is  termed an E-eigenpair.
	If both  eigenvector and eigenvalue  are  real, it is also termed  a    Z-eigenpair,
	which  corresponds to  a    Karush-Kuhn-Tucker (KKT) point of  (\ref{opti_ori}) over   $ \mathbb R$.
	Throughout this document,
	we focus on the Z-eigenpair  type.
\end{wlremark}

To simplify the sign  convention    for   eigenpairs analyzed later,
the definition of 
equivalence   class of eigenpairs, borrowed from \cite{hm},  is  presented as follows:
\begin{wldefinition}[Equivalence class of    eigenpairs]	\label{def.Equivalent}
	Let $ \mathcal{S} \in T^{m}(\mathbb{R}^{p}) $ be a real symmetric tensor of order $m$ and dimension $p$.
	The   equivalence  class   of     $ (\lambda,  \mathbf  v)  $   is  denoted
	\begin{equation}\label{equclass}
	[
	(\lambda,  \mathbf  v) ] :=
	\{
	(\lambda',  \mathbf  v')  \vert
	\lambda' =  t^{m-2}\lambda  ,
	\mathbf  v' =  t \mathbf  v ,
	t \in  \mathbb  C \backslash \{0 \}
	\}.
	\end{equation}
\end{wldefinition}

Then, for the Z-eigenpairs we discussed in this work, 
regarding  the  sign of an  eigenpair 	$(\lambda,  \mathbf  v)$,
we can conclude the following lemma:
\begin{wllemma}\label{lemma.sign}   
	Let $\mathcal{S} \in T^{m}(\mathbb{R}^{p})$ be a real symmetric tensor.
	Then, 	the following statements hold:	
	\\
	(1):   $(\lambda, \mathbf{v})$ is 
	a  Z-eigenpair of $\mathcal{S}$ if and only if $((-1)^{m-2}\lambda ,-\mathbf v )$ is 
	a    Z-eigenpair of $\mathcal{S}$. 
	\\
	(2):
	When 
	$m$  is odd, the eigenpair  \( (\lambda ,\mathbf v) \) of   \( \mathcal S\)  can always be chosen  such that  
	\( \lambda \geq 0 \); 
	when $m$ is even, and  $ \mathcal S$ is 
	generated by a finite number of  vectors,
	any Z-eigenpair  must satisfy \( \lambda \geq 0 \).	
\end{wllemma}

\begin{proof}	
	The first claim can be easily verified 	based on  
	the equivalence class of  eigenpairs defined in (\ref{equclass}).
	When restricting to Z-eigenpairs,	$ t $ reduces to  $t \in  \mathbb  R \backslash \{0 \}$
	with $ t^2 =1$ by the unit  length condition, which yields the claim.
	In this sense, we usually select one of them for discussion.
	When  $  m$ is odd, for
	$(\lambda,\mathbf v)$   and   $(-\lambda,-\mathbf v)$,  without  loss   of 	generality,
	we    always   take   eigenpairs   with  $ \lambda \ge 0$;
	when  $  m$ is even,    Z-eigenpairs   appear  in  $\pm$ pairs  with  the same  eigenvalue. 
	For $ \mathcal S  \in T^{m}(\mathbb{R}^{p})$
	generated by a finite number of vectors 
	of  the form 
	$	\mathcal{S}:=\sum_{i=1}^{N}\mathbf{r}_{i}^{\circ m} $ which is a  generalized  form of  (\ref{simplextensor}), where  $ \mathbf{r}_{i}   \in \mathbb{R}^{p \times 1 }$ 
	is the generated  vector,  and $N$ is the number of vectors, 
	it also  holds  that 
	$ \lambda  = \mathcal S  \mathbf v^{m}    =
	\sum_{i=1}^{N} ( \mathbf  {r}_{i}^{\mathrm T} \mathbf v)^{ m} \ge 0$
	for even $m$.
	This proves the second claim. 
	The proof is complete.     $\blacksquare$
\end{proof}

\begin{wlremark}\label{rem.signnew}
	Up to sign equivalence, 
	we usually only select one of them as the representative for  analysis, 
	which is  conventionally adopted in previous works \cite{hm,RobustEigen,Cuicf,upperbound}.
	For convenience,
	for the tensor
	generated by a finite number of vectors,  which is a  generalized  form of  the regular simplex tensor discussed in this work, 
	we always select the pair  $(\lambda,\mathbf v)$  with $\lambda \ge 0$   for discussion. 
	For the focused regular simplex tensor to be discussed later, 
	we also provide a detailed 
	parameter   selection to guarantee such a  setting  in Lemma  \ref{rem.sign.alpha}.
\end{wlremark}

To further identify the local optimality of an  eigenpair,
the knowledge of  optimization theory
tells us that   it is necessary to utilize
the    second-order derivation
of
$  L(\mathbf v, \lambda) $
with respect to  $ \mathbf v $,
termed  the Hessian matrix of
(\ref{Lagrangian_function}).
It  is   denoted
by
\begin{equation}\label{hessian_matrix}
\mathbf H(\mathbf v) :=  \triangledown_{\mathbf v\mathbf v}^{2} L(\mathbf v, \lambda) = (m-1)\mathcal S \mathbf v^{m-2} - \lambda \mathbf I_{p} ,
\end{equation}
where
$ \mathbf I_{p} $
is  a   $ p  \times  p $  identity matrix.
Since   (\ref{opti_ori}) is a constrained model,
besides the Hessian term,
the constraint should be further considered.
Then,
the
locally    optimal   solutions of   (\ref{opti_ori})
can  be  identified  by
checking the  negative
semi-definiteness
of the  following  matrix:
\begin{align}\label{Mhess}
\mathbf{K }  : = \mathbf {K}(\mathbf v)
& =  \mathbf  P_{\mathbf  v}^{\bot} \mathbf H (\mathbf v)   \mathbf  P_{\mathbf  v}^{\bot}
\end{align}
where
$ \mathbf  P_{\mathbf  v}^{\bot}
=
\mathbf  I_{p}  -
\mathbf v\mathbf  v^{\mathrm T}$
is  the orthogonal  complement   projection matrix of  $ \mathbf v $.
We conclude the following lemma:
\begin{wllemma}\label{lem.LOCALMAX}
	For	 a   Z-eigenpair
	$(\lambda ,\mathbf v )$   of    $\mathcal S   \in    T^{m}(\mathbb R^{p}) $,
	if it holds that
	\begin{equation}
	\triangledown_{\mathbf v}  L(\mathbf v, \lambda) = \mathbf 0,
	\quad     
	\mathbf K   \preceq 0,
	\end{equation}
	and  $\mathbf K$    has    only one  zero eigenvalue, 
	it  is a locally maximized
	eigenpair  of   $\mathcal S$.
\end{wllemma}
The  detailed  explanation for  the
derivation of
(\ref{Mhess}) and Lemma  \ref{lem.LOCALMAX}
can  be found in    the Appendix part.

\subsection{Robust  eigenpairs of  symmetric tensors}\label{robustsection}

One of  the  widely used  algorithms  to  obtain  tensor   eigenpairs  is  called  the  tensor  power   method, which
is  based on  the  following  mapping  
\begin{equation}
\phi(\mathbf{v}) :=\frac{\mathcal S  \mathbf{v}^{m-1}}{\left\|\mathcal S  \mathbf{v}^{m-1}\right\|},
\end{equation}
and  performs  the    iterative   scheme:
\begin{equation}\label{tenpower}
\mathbf  {v}_{k} \mapsto \frac{\mathcal{S}  \mathbf{v}_{k-1}^{m-1}}{\left\|\mathcal{S}  \mathbf{v}_{k-1}^{m-1}\right\|}
,
\end{equation}
and the  initial    $\mathbf  {v}_{0}$
is a randomly generated unit-length vector.

Regarding the mapping  $\phi(\mathbf{v}) $,
an interesting and 
widely researched   issue is to 
check  whether  the solution is robust or not.
A robust eigenvector of
$ \mathcal{S} $    is an eigenvector $\mathbf{v}$ that is an attracting fixed point of the tensor power method \cite{RobustEigen}.
Based on  Lemmas 3.2 and 3.3  of \cite{RobustEigen}, the robustness can be attributed to   calculating  the spectral radius of the Jacobian
matrix   $ \mathbf{J}$
of  $\phi(\mathbf  v) $,  which  is  
given  by
\begin{equation}\label{Jacobianmatirx}
\mathbf{J} := \mathbf{J}(\mathbf{v}, \lambda)=\frac{m-1}{\lambda}\left(\mathcal{S}
\mathbf{v}^{m-2}- \lambda  \mathbf{v} \mathbf v^{ \mathrm T} 
\right).
\end{equation}

Two details are emphasized here for   the   robustness issue.  
1): (\ref{Jacobianmatirx})	implies that eigenpairs with  $\lambda = 0$ 
are excluded from the definition of robustness,  because the Jacobian matrix for $\lambda = 0$   is undefined. 
Such eigenpairs lie outside the scope of the robustness analysis.
2):  the  convergence   set of  (\ref{tenpower})  differs for odd and even $m$.
For even  $m$, both 	$(\lambda,\mathbf v)$   and   $(\lambda,-\mathbf v)$
are the fixed points; 
for odd $m$, 
only 	$(\lambda,\mathbf v)$  is  a fixed point. However,
the sign-equivalent representative  $(-\lambda,-\mathbf v)$  is generally not a fixed point of the same map.
To avoid discussing odd and even $m$ separately in  the  later proof,  
our discussion on  the  robustness issue is also up to sign equivalence
and  we  only reserve  the pair  $(\lambda,\mathbf v)$   for discussion, similar to the selection in Remark  \ref{rem.signnew}. 
When the tensor is generated by a finite number of vectors, 
which serves as a generalized version of the regular simplex tensor discussed in this work,
Claim 2 in Remark \ref{lemma.sign}  implies that,  up to sign equivalence,  we can always select eigenpairs with non‑negative eigenvalues $\lambda \ge 0$ for both odd and even $m$.
Another justification for this choice    is that  when $m$ is even, 
if 	$(\lambda,\mathbf v)$ is a robust one, so is  $(\lambda,-\mathbf v)$.	
It holds  that
$
\mathbf{J}(
-\mathbf{v},\lambda)=\frac{m-1}{\lambda}\left(\mathcal{S}  (-\mathbf  v)^{m-2}- \lambda  (-\mathbf  v)  (-\mathbf  v)^{{\mathrm T }}
\right)
=
\mathbf{J}(
\mathbf{v},\lambda).
$
This implies that their robustness  are   synchronous,
and checking one of them is  sufficient.
Combining those two claims, we can always select the eigenpairs with $\lambda >0$ for the robustness issue.
Then, we formally provide the following definition of  the robustness issue:

\begin{wldefinition}[Robust eigenpair]	\label{def.robust}
	Let $\mathcal{S} \in T^{m}(\mathbb{R}^{p})$ be a real symmetric tensor of order $m$ and dimension $p$ generated by a finite number of vectors. For the real Z-eigenpair $(\lambda, \mathbf{v})$ of $\mathcal{S}$ up to sign equivalence, with positive eigenvalue $\lambda > 0$, consider the tensor power method defined in (\ref{tenpower}),
	it    is called
	robust    if $\mathbf{v}$ is an attracting fixed point of (\ref{tenpower}).
	Equivalently, let   $ \mathbf{J}$
	be the Jacobian matrix of the mapping $ \phi(\mathbf{v}) $ at $ \mathbf{v}$ defined in (\ref{Jacobianmatirx}). Then $(\lambda,\mathbf{v})$ is robust if and only if
	\[
	\rho(\mathbf{J}  )< 1,
	\]
	where $\rho(\cdot)$ denotes the spectral radius.
\end{wldefinition}

\begin{wlremark}
	We emphasize  that the robustness  of   eigenpairs
	is strongly dependent on the adopted numerical algorithm.
	For example,  a  similar definition also appears for the Newton-type schemes in \cite{NCM}.
	In this work, 
	a robust eigenpair by default refers to that obtained from the tensor power iteration.
\end{wlremark}

\subsection{Regular  simplex  frame  and  tensor}
In this part,  we   introduce  a   special class of   tensors  termed
regular  simplex tensors.
For  convenience of  notation
in later parts,
we set the dimension of   the   regular  simplex  tensor to be $p=n-1$.
First,  the    definition of the  generalized  equiangular set  is  introduced as  follows,
which is borrowed from Definition 4.1 and 4.3  of   \cite{RobustEigen}:
\begin{wldefinition}[Equiangular set and equiangular tight frame]
	An equiangular set (ES) is a collection of vectors $\mathbf{w}_{1}, \ldots, \mathbf{w}_{r} \in \mathbb{R}^{n-1}$ with $r \geq n-1$ if there exists $\theta \in \mathbb{R}$ such that
	$$
	\theta=
	\left|
	\mathbf{w}_{i}^{\mathrm T} \mathbf{w}_{j}
	\right|,   1 \le i, j \le r, i \neq j   \quad \text { and } \quad\left\|\mathbf{w}_{i}\right\|=1,
	i=1,2, \dots, r.	$$
	Furthermore, an $\mathrm{ES}$ is  called  an equiangular tight frame (ETF)  if
	\begin{equation}
	\mathbf{W} \mathbf{W}^{\mathrm T}
	=\frac {r} { n-1} \mathbf{I}
	:=a \mathbf{I}, \quad
	\mathbf{W}:=\left(\mathbf{w}_{1},  \cdots,  \mathbf{w}_{r}\right) \in \mathbb{R}^{(n-1) \times r},
	\end{equation}
	where 	$\mathbf{w}_{i}, i=1,2 \dots, r$   are  called  the  vectors  in  the  equiangular tight frame.
\end{wldefinition}

Two  examples that satisfy the above definition are  given  as follows:
\begin{itemize}
	\item
	When   $r=n-1$ and $a=1$, the  orthonormal   basis
	$\mathbf{w}_{1}, \ldots, \mathbf{w}_{(n-1)}$
	forms
	an  ETF,  where $\theta =0$.
	Its  corresponding deduced tensor is termed  an  odeco tensor.
	\item
	When    $r=n$ and $a=\frac {n}{n-1}$,   $\left\{\mathbf{w}_{1}, \ldots, \mathbf{w}_{n}\right\} \in  \mathbb{R}^{(n-1) \times  n}$,
	if a stronger condition
	$\mathbf{w}_{i}^{\mathrm T} \mathbf{w}_{j} = -\frac{1}{n-1}$
	holds for any $i \neq j$,
	it  is   termed
	the   regular  simplex frame,
	where
	$ \theta =  \frac{1}{n-1}$.
	In   two-dimensional   space,
	the  regular simplex  frame    reduces to  a  regular  triangle.
\end{itemize}

The regular  simplex  tensor  is the core of this work.
In the following,  for notational convenience,
$r$ is fixed  as  $r=n$,
$\mathbf W $  is set    to be of size    $(n-1) \times  n$ and  always refers to the regular simplex frame.
Then,  	 the formal definition of
the 	regular  simplex  tensor    generated   by    its    frame $\mathbf W $
is given as follows:
\begin{wldefinition}[Regular simplex tensor]	\label{def.rst}
	Let $n   \ge 3$ and  	$\mathbf W  \in  \mathbb{R}^{(n-1) \times  n} $ be the
	regular simplex frame.
	For a given order $m \ge 3$, the  regular simplex tensor is defined as
	\begin{equation}\label{simplextensor}
	\mathcal{S}_{\mathbf W }:=\sum_{i=1}^{n} \mathbf{w}_{i}^{\circ m} ,
	\end{equation}
	where $\circ$  is the outer product   in Definition \ref{outerprod},
	and  	$\mathcal{S}_{\mathbf W }
	$ is  a   symmetric  tensor of order
	$ m $
	and dimension
	$ n-1$.
\end{wldefinition}

\begin{wlremark}\label{RemarkScope}
	Throughout this paper,    we   are only concerned with
	$n \ge 3, m \ge 3$  for (\ref{simplextensor}).
	The matrix case (corresponding to 	$ m = 2$)  is not included.
	In addition,
	the case  of 		$n = 3, m = 4$   turns   out   to be  a very  special combination,
	whose corresponding   objective  value, i.e.,
	$
	\mathcal S_{\mathbf W } \mathbf v^{m},
	$   is a constant.
	Such a  conclusion has been   discussed   in Theorem 4.7 of \cite{RobustEigen}
	and Theorem 12 of \cite{teneigenstructure}.
	Please  refer to them   for  details.
	In this special  case,
	any unit-length vector is an eigenpair of $
	\mathcal S_{\mathbf W }$ with the same eigenvalue  $\frac 9 8$. 
	In this sense, any vector  is both  locally maximized and minimized. 
	For the robustness issue,  based on the conclusion in Theorem 16 of  \cite{teneigenstructure},  no  eigenpair is  robust
	for this case. 
	Therefore, in the later analysis, 
	this special case will always be discussed separately
	in both the  robustness and local optimality discussions.
	
\end{wlremark}

\subsection{Main focus of this work}\label{conjecture}

In this paper, we mainly analyze robust eigenpairs of 
the  regular simplex tensor and aim to  resolve the  Conjecture 4.8  
originated  in   \cite{RobustEigen},
which states that 
the only robust eigenvectors of  the   regular simplex tensor
are the vectors in the  regular simplex   frame.
Note that the above statement is also  understood   up to sign equivalence of eigenpairs 
as we emphasize in the introduction  and  Remark  \ref{rem.signnew}. 
In addition, this  statement may be less rigorous, 
since for some lower-dimensional settings, there are no robust eigenpairs.
For a more precise presentation of the above conjecture, and since it will be proved later in this work, we state it as the main theorem as follows:

\begin{theorem} \label{conjecturesimplex}
	Let $\mathcal{S}_{\mathbf{W}}\in T^{m}(\mathbb{R}^{n-1})$ be a regular simplex tensor with $n\geq 3$, $m\geq 3$. Then,  the following statements hold:	
	\\
	(1):  If 	 $(m,n) = (3,3)$, $(3,4)$, or $(4,3)$, there are no robust eigenpairs, either among the vectors in the regular simplex frame or elsewhere.
	\\
	(2):  For all other $(m,n)$ combinations, the only robust eigenvectors, considered up to sign  equivalence, 
	are   	 the  vectors $\mathbf{w}_1,\ldots,\mathbf{w}_n$ in the  regular simplex  frame $\mathbf W$.
\end{theorem}

In the following, we would like to
discuss the existing works
and  	 the deficiency  of  	 the  current strategy.
In addition, the implication of resolving this   theorem
will also be discussed in Section
\ref{futurework}.

(1)
Existing works:
In \cite{RobustEigen}, the authors only showed   that
the vectors in the  regular simplex  frame    are indeed  robust 
for the tensor of  dimension $n-1$ and order $m$, with $n \ge 3$, $m \ge 3$, and $n+m \ge 8$.
(See Theorem 4.6 of \cite{RobustEigen} for details.)
However,
the uniqueness
was not justified, which is then presented  as the   aforementioned  Conjecture 4.8. 
For this purpose,
one must check  the robustness of  all  eigenpairs.
Such a strategy was adopted in a later work  \cite{teneigenstructure}.
However,
they only provided the  proof for  $n=3$,
and the experimental    justification for   $n=4$.
For    larger  $n$, it is a tough task to prove.

(2) Difficulties of   the  strategy in  \cite{teneigenstructure}:
there are mainly  two  difficulties. First,
enumerating  all eigenpairs
is computationally  expensive   as $n$ and $m$ become larger, since the number
increases  exponentially.
In addition,
for most  eigenpairs $
\mathbf v$, it will be difficult to
determine the result of
$
\mathbf v^{\mathrm T} \mathbf{w}_{i}$
in (\ref{swj})
and  that of $\lambda$
in (\ref{lmdwj}),
thus
making it difficult to derive   an explicit form
for 
$ \mathbf J (\mathbf  v)$.

To the best of our knowledge,  so far,
the  references \cite{RobustEigen,teneigenstructure} are the only two works
that  focus on  identifying  the  robust  eigenpairs of  the   regular  simplex  tensor.
In this paper,  
building upon the two works,
we proceed further  on  the above    theorem
and
finally complete the proof.
The details are as follows.

\section{Relation  between   robust and locally maximized eigenpairs}\label{Relationonrobust}

In this part,  we  would  like  to  provide  an  auxiliary   criterion   for  robustness  checking.
It should be emphasized that the following conclusions     are not limited only to   the 
regular simplex tensor $ \mathcal{S}_{\mathbf W }$,  
but can be  established   for any type of symmetric tensor 	generated by a finite number of samples.
Therefore, we set the dimension of  the   tensor to be general $p$ in this part.   
Based on the definition  \ref{def.robust}  for the robustness issue,
the assumption $\lambda > 0$ will be adopted 
in the following two lemmas.

It  can  be  observed   by   comparing    (\ref{Jacobianmatirx})  with   (\ref{hessian_matrix})
that  both  of  them   contain     terms  $\mathcal{S}  \mathbf{v}^{m-2}$  and  $\mathbf v\mathbf v^{\mathrm T}$,  and
by  further   investigating  their  relationship,  the  following    lemma  can  be   established.

\begin{wllemma}\label{RobustLocalpre}
	Given  an  eigenpair
	$(\lambda ,\mathbf v )$  of   a  tensor $\mathcal S   \in    T^{m}(\mathbb R^{p}) $
	where $  \lambda > 0$,
	concerning
	the 	matrix  $\mathbf K$  in  (\ref{Mhess})
	and   the Jacobian
	matrix  $\mathbf J$  in  (\ref{Jacobianmatirx}),
	the following relationship holds:

	(1):
	$
	\lambda	\mathbf J =   	{\mathbf K }
	+\lambda  \mathbf  P_{\mathbf  v}^{\bot}$,
	where
	$ \mathbf  P_{\mathbf  v}^{\bot}
	$
	is  the orthogonal  complement   projection matrix of  $ \mathbf v $.

	(2):
	Let  	$ \mathbf  v_{1}^{\bot}, \dots,  \mathbf  v_{p-1}^{\bot} $
	be $p-1$ unit-length  vectors that are orthogonal to each other,
	and also to  	$	\mathbf v$,
	i.e.,
	$	\mathbf v^{\mathrm {T}}
	\mathbf  v_{i}^{\bot} =0,
	(	\mathbf  v_{i}^{\bot})^{\mathrm {T}}
	\mathbf  v_{j}^{\bot} =0$,
	for   $i,j=1, 2, \dots, p-1, i\neq j$. 	Denote
	$\mathbf V
	=
	[\mathbf v,   	 \mathbf  v_{1}^{\bot}, \dots,  \mathbf  v_{p-1}^{\bot}] \in
	\mathbb R^{p  \times  p} $
	the orthonormal matrix
	where  	$\mathbf V^{\mathrm T}  \mathbf V = \mathbf V\mathbf V^{\mathrm T}
	= \mathbf I_p$,
	$\mathbf V$   serves  as   the common  eigenvector  matrix  for
	$	\mathbf J, \mathbf K $ and  	$	 \mathbf  P_{\mathbf  v}^{\bot}$.

	(3):
	$(0,\mathbf v)$ is always one eigenpair for 	$	\mathbf J, \mathbf K $ and  	$	 \mathbf  P_{\mathbf  v}^{\bot}$.
	For the  remaining $p-1$  eigenpairs,
	$(\mu_{i}, \mathbf  v_{i}^{\bot})$ is an eigenpair of $\mathbf J$ if and only if
	$(
	\lambda(\mu_{i}-1), \mathbf  v_{i}^{\bot})$  is an eigenpair of $\mathbf K$,
	where $i=1, 2, \dots, p-1$.
\end{wllemma}

\begin{proof}
	For (1),
	$ 	\mathbf {K} $   in  (\ref{Mhess})  is  rewritten  as  the  following  form:
	\begin{align}\label{equ.K.sim}
	\mathbf {K}
	& \nonumber
	=  \mathbf  P_{\mathbf  v}^{\bot} \mathbf H (\mathbf v)
	\mathbf  P_{\mathbf  v}^{\bot}
	=
	(\mathbf  I_{p}  -
	\mathbf v\mathbf  v^{\mathrm T} )
	[ (m-1)\mathcal S \mathbf v^{m-2} - \lambda \mathbf I_{p}]
	(\mathbf  I_{p}  -
	\mathbf v  \mathbf v^{\mathrm T} )
	\\  	\nonumber
	& =
	(\mathbf  I_{p}  -
	\mathbf v\mathbf v^{\mathrm T} )
	[
	(m-1)\mathcal S \mathbf v^{m-2}
	-
	\lambda \mathbf I_{p}
	-
	(m-2)   \lambda \mathbf v\mathbf v^{\mathrm T}
	]
	\\    	 	\nonumber
	&
	=
	(m-1)\mathcal S \mathbf v^{m-2}
	-
	\lambda \mathbf I_{p}
	-
	(m-1)\lambda \mathbf v \mathbf v^{\mathrm T}
	+
	\lambda	\mathbf v\mathbf v^{\mathrm T}
	\\
	&
	=
	(m-1)\mathcal S \mathbf v^{m-2}
	-
	\lambda \mathbf I_{p}
	-
	(m-2)\lambda \mathbf v \mathbf v^{\mathrm T}  ,
	\end{align}
	where in the last  equation, we utilize
	$ \mathcal S \mathbf v^{m-2}  	\mathbf v
	= \mathcal S \mathbf v^{m-1}
	=
	\lambda		\mathbf v.
	$
	By   comparing  with   (\ref{Jacobianmatirx}), it  holds  that
	\begin{align}\label{lmdJMI}
	\lambda	\mathbf J =   	{\mathbf K }
	+\lambda  (\mathbf  I_{p}  -
	\mathbf v \mathbf v^{\mathrm T} )
	=
	\mathbf K
	+\lambda    \mathbf  P_{\mathbf  v}^{\bot}.
	\end{align}

	For 	(2),
	it can be checked that
	\begin{equation}\label{Jacobianmatirxvec}
	\mathbf{J} \mathbf{v}
	=
	\frac{m-1}{\lambda}
	(
	\mathcal{S}  \mathbf{v}^{m-2}\mathbf{v}
	- \lambda  \mathbf{v} \mathbf{v}^{\mathrm T} \mathbf{v})
	=
	\frac{m-1}{\lambda}
	(
	\lambda  \mathbf{v}
	- \lambda  \mathbf{v})
	=
	0 \cdot
	\mathbf{v},
	\end{equation}
	where   we  reuse
	$ \mathcal S \mathbf v^{m-2}  	\mathbf v
	= \mathcal S \mathbf v^{m-1}
	=
	\lambda		\mathbf v$.
	In  a  	similar way,
	the  vector  $\mathbf v$
	will  simultaneously
	be the  eigenvector  	with eigenvalue 0 for   both
	$ \mathbf K$ and $  \mathbf  P_{\mathbf  v}^{\bot}$,
	which  holds 	that
	\begin{equation}\label{Kmatirxeig}
	\mathbf{K} \mathbf{v} =
	0 \cdot
	\mathbf{v},
	\quad
	\mathbf  P_{\mathbf  v}^{\bot}
	\mathbf{v} =
	0 \cdot
	\mathbf{v}.
	\end{equation}
	The above   results 	 mean
	that
	$  	\mathbf{v}$
	is a common  eigenvector of
	$ \mathbf{J},  \mathbf K,  \mathbf  P_{\mathbf  v}^{\bot} $ with eigenvalue 0.
	Therefore,  this further implies that the other eigenvectors lie  in the
	null  space of 	$  	\mathbf{v}$.
	By the formulation of
	$\mathbf V$ defined above,
	$	\mathbf J, \mathbf K $ and  	$	 \mathbf  P_{\mathbf  v}^{\bot}$ can be  diagonalized
	by 	$\mathbf V$  as    follows:
	\begin{align}\label{lmdJMIVdiag}
	\lambda  \mathbf V^{\mathrm T}
	\mathbf J
	\mathbf V=
	\mathbf V^{\mathrm T}
	[ 	{\mathbf K }
	+\lambda  \mathbf  P_{\mathbf  v}^{\bot}
	]
	\mathbf V.
	\end{align}

	For 	(3),
	$(0,\mathbf v)$ is always one eigenpair  for the three matrices, as   analyzed before.
	Note  that
	$
	\mathbf  P_{\mathbf  v}^{\bot}
	\in
	\mathbb R^{ p  \times  p}
	$
	is  a  projection  matrix  with  rank  $p-1$, and  its  eigenvalues are given by
	$ 1, 1,   \dots,  1,  0    $,  where
	the 	multiplicity of   eigenvalue   $ 1 $  is  $ p-1 $.
	If 	$(\mu_{i}, \mathbf  v_{i}^{\bot})$ is an eigenvalue of $\mathbf J$,
	(\ref{lmdJMIVdiag})
	implies that
	\begin{align}\label{lmdrelationequ}
	\lambda  \mu_{i}\mathbf  v_{i}^{\bot}
	=
	\lambda	  \mathbf J  \mathbf  v_{i}^{\bot}
	=
	(	\mathbf K
	+\lambda    \mathbf  P_{\mathbf  v}^{\bot})
	\mathbf  v_{i}^{\bot}
	=
	\mathbf K  	\mathbf  v_{i}^{\bot} + 	\lambda		\mathbf  v_{i}^{\bot},
	\end{align}
	meaning that	$(\lambda(\mu_{i}-1), \mathbf  v_{i}^{\bot})$  is an  eigenpair  of $\mathbf K$.
	Conversely,
	if
	$(\lambda(\mu_{i}-1), \mathbf  v_{i}^{\bot})$  is an  eigenpair   of $\mathbf K$,
	since $ \lambda\neq 0$, it also holds that $(\mu_{i}, \mathbf  v_{i}^{\bot})$ is an  eigenpair of $\mathbf J$.
	The proof is complete.
	$\blacksquare$
\end{proof}

With the eigenvalues relationship between
$\mathbf J$ and $\mathbf K$ at hand, we can further deduce the
relation on   robust and locally optimal   eigenpairs,
by discussing
as follows:
\begin{wllemma}\label{RobustLocal}
	Given  an  eigenpair
	$(\lambda ,\mathbf v )$  of   a  tensor $\mathcal S   \in    T^{m}(\mathbb R^{p}) $
	where $  \lambda > 0$,
	concerning
	the 	matrix  $\mathbf K$  in  (\ref{Mhess})
	and   the Jacobian
	matrix  $\mathbf J$  in  (\ref{Jacobianmatirx}),
	the following statements hold:

	(1): If  $   \mathbf K \succeq 0$,    it holds that  $ 		\rho (\mathbf J)	\ge 1 $.

	(2): If  $   \mathbf K $ is  indefinite,  it holds that  $	\rho (\mathbf J)	\ge 1 $.

	(3):
	If   $   \mathbf K \preceq 0$,
	and  two extra conditions hold: i) there is only  one zero  eigenvalue; ii) 
	the remaining  non-zero 	eigenvalues of  $   \mathbf K $ 
	lie  in the  open  interval
	$(-2 \lambda,0)$, 
	it holds that  $ 	\rho (\mathbf J) <  1 $.
	
	(4): Conversely to (3),
	if
	$  	\rho (\mathbf J) <  1 $,
	it holds that
	$ 	\mathbf K \preceq 0 $.
\end{wllemma}

\begin{proof}
	(1):
	When $ \mathbf K \succeq 0$,
	it holds that
	for any $i=1, 2, \dots, p-1$,
	$	\lambda(\mu_{i}-1) \ge  0$.
	The assumption $ \lambda >0$
	implies that
	$\mu_{i} \ge 1$,
	and
	\begin{align}	\rho (\mathbf J)
	=
	\max  \vert  \mu_{i}\vert
	\ge
	1 ,
	\end{align}
	which indicates
	that the corresponding  eigenpair  	 is not  robust.

	(2):
	When $ \mathbf K$ is indefinite,
	note that  the  number of  	 zero   eigenvalue is at least one, denoted $t \ge 1$.
	Furthermore,    assume that there are  $k$ positive  eigenvalues
	and  $p-t-k$ negative 	 eigenvalues
	where $k \ge 1, p-t-k \ge 1$.
	Denote the index set for   the  $k$ positive  eigenvalues  to be $\mathbb A$,
	and  $p-t-k$ negative  eigevalues   to be $\mathbb B$.
	For any $i \in \mathbb A$,
	$	\lambda(\mu_{i}-1) \ge  0$.
	The assumption $ \lambda >0$
	implies that
	$\mu_{i} \ge 1, i \in \mathbb A$.
	Similarly,
	$\mu_{i} \le 1, i \in \mathbb B$.
	Then
	\begin{align}	\rho (\mathbf J)
	=
	\max  \vert  \mu_{i}\vert
	=
	\max  (\vert  \mu_{i\in \mathbb A}\vert,   \vert  \mu_{i\in \mathbb B}\vert )
	\ge
	1 ,
	\end{align}
	which indicates
	that the corresponding  eigenpair   is not    robust.

	(3):
	When $ \mathbf K \preceq 0$,
	and there is only one  zero  eigenvalue,
	it holds that 
	$\lambda(\mu_i - 1) < 0$  for 	  the remaining $p-1$ eigenvalues of  $ \mathbf K$ where  $i = 1, 2, \dots, p-1$.
	The assumption $ \lambda >0$
	implies that
	$\mu_{i} <  1$.
	The condition  that  	 the remaining $p-1$ nonzero  eigenvalues of  $   \mathbf K $ 
	lie  in the  open  interval
	$(-2 \lambda,0)$  further  implies that
	\begin{equation}\label{add}
	\begin{cases}
	\max\limits_{i}  \lambda(\mu_{i}-1) < 0   
	\\
	\min\limits_{i}  \lambda(\mu_{i}-1) > -2 \lambda  
	\end{cases} , i=1, 2, \dots, p-1.
	\end{equation}
	This   implies that  $\max\limits_{i} 	\mu_{i}< 1,  	\min\limits_{i} \mu_i  >-1 $, and 
	\begin{align}	\rho (\mathbf J)
	=
	\max  \vert  \mu_{i}\vert
	<
	1 ,
	\end{align}
	which indicates
	that the corresponding  eigenpair  is     robust.

	(4):
	If  an eigenpair is  robust, it holds that 
	$
	\rho  (\mathbf J ) 	=
	\max  \vert  \mu_{i}\vert
	<
	1
	$.
	Since there is always  one  eigenpair with zero  eigenvalue for $\mathbf J$ and $\mathbf K$,
	this 	implies  that for the remaining $ p-1$ eigenvalues of  $\mathbf J$,
	it holds that  
	$  \mu_{i}  \in  (-1, 1), i=1, 2, \dots, p-1 $.
	By    claim  3)  in Lemma  \ref{RobustLocalpre},
	it necessarily holds for the remaining $ p-1$  eigenvalues of  $	\mathbf K$
	that
	$
	-2 \lambda < 	\lambda  \mu_{i} -	\lambda
	< 0 $, meaning that
	$\mathbf K$ is negative semi-definite, since $\mathbf K$ also has   one zero  eigenvalue.
	The proof is complete.  $\blacksquare$
\end{proof}

\begin{figure}
	\centering
	\begin{tikzpicture}[node distance=1.5cm]
	\node (kkt) [process]
	{All Eigenpairs};
	
	\node (max) [process, below of=kkt, xshift = -4cm]
	{Locally maximized ones};
	
	\node (sad) [process, below of=kkt]
	{Saddle ones};
	
	\node (min) [process, below of=kkt, xshift = 4cm]
	{Locally minimized ones};
	\node (rob) [process, below of=max]
	{Robust eigenpairs};
	
	\node (nonrob) [process, below of=sad, xshift = 2cm]
	{Non-robust eigenpairs };

	\draw [arrow] (kkt) -- (max);
	\draw [arrow] (kkt) -- (sad);
	\draw [arrow] (kkt) -- (min);
	
	\draw [arrow]   (max) -- node [right] {	(\ref{add})} (rob) ;
	\draw [arrow] (sad) -- (nonrob);
	\draw [arrow]  (min) -- (nonrob);
	
	\end{tikzpicture}
	\caption{	The classification of all eigenpairs,
		and   the  relationship between the locally optimal  and robust eigenpairs
		based on Lemma \ref{RobustLocal}.
	}
	\label{flow.robust.local}
\end{figure}
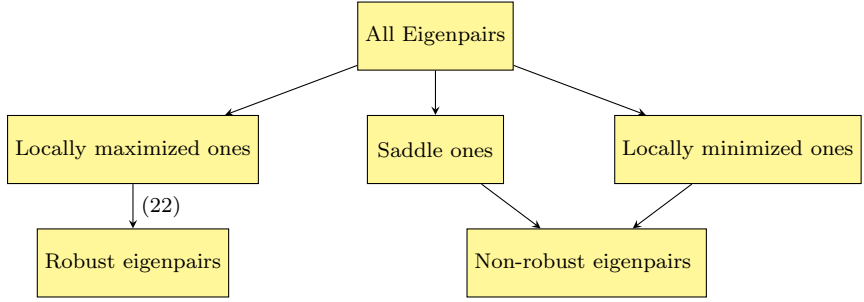

\begin{wlremark}[Implication for robustness checking]\label{rem.imp}
	By the above conclusions in the lemma,
	if		$(\lambda ,\mathbf v )$  is    a   locally  maximized    solution  of  model (\ref{opti_ori}) where  $	\mathbf K \preceq 0$
	and 	 $	\mathbf K $  has only one zero eigenvalue,
	with additional  condition defined  in (\ref{add})  for the remaining $p-1$ eigenvalues,
	it is  robust.
	The locally minimized and saddle solutions are always non-robust.
	Thus, given a  symmetric tensor (which is not limited to the special regular simplex  tensor),
	we do not need to analyze and determine the spectral radius
	of $\mathbf J$    for   all eigenpairs,   as  was done  in  \cite{teneigenstructure}.
	In contrast,
	it suffices to
	focus only on  the locally maximized
	eigenpairs.
	In other words,
	only the  locally maximized eigenpairs of a tensor  can   be  \textbf{potential
		candidates for robust eigenpairs}.
	Local maximality  is necessary but not sufficient for robustness.
	Thus, robustness is a strictly stronger property than local maximality. Consequently, to identify robust eigenpairs, it suffices to first isolate locally maximized eigenpairs and then check the additional spectral condition. Figure 1 summarizes the logical relationships.
\end{wlremark}

\section{Optimization Landscape}\label{optland}
Therefore,
the following aim   is to identify which solutions are locally maximized, for 
the  regular simplex tensor
$ \mathcal{S}_{\mathbf W }    \in    T^{m}(\mathbb R^{n-1})$,
where $n \ge 3, m \ge 3$.
Note that such an issue  has been widely studied recently, and
is  generally termed
"optimization landscape".
For example,  previous works have
investigated this subject for (fourth-order) tensor
eigenpair  decomposition \cite{optlandtensor}, and Tucker  decomposition \cite{optlandtucker}.
In this work, we pay attention to  the  regular simplex tensor.
For a better  structural   arrangement,
we  summarize   it as the following theorem.

\begin{theorem}\label{Theorem.local_eigenpair}
	Let $\mathcal{S}_{\mathbf{W}}\in T^{m}(\mathbb{R}^{n-1})$ be a regular simplex tensor with $n\geq 3$, $m\geq 3$. Then,
	for all  $(m,n)$ combinations
	except for $(m, n) =  (4, 3)$, the only locally maximized 
	eigenpairs, considered up to sign  equivalence, 
	are   the  vectors $\mathbf{w}_1,\ldots,\mathbf{w}_n$ in the  regular simplex  frame $\mathbf W$.	
\end{theorem}

As the detailed proof extends over multiple pages, it is postponed to Section
\ref{sec.proof.opt}.
Combining  Theorem
\ref{Theorem.local_eigenpair} and  Lemma \ref{RobustLocal}
will
resolve   two
difficulties
concerning the identification of robust eigenpairs   for  Theorem  \ref{conjecturesimplex}:

(1):  The number of
eigenpairs to be  checked
can be  greatly  reduced,  because 
the number of    locally maximized
eigenpairs is  exactly  equal to $n$, and    is   always no larger  than that of
all eigenpairs,  especially when $n$  and $m$ become   larger.

(2):
Naturally,
calculating the Jacobian matrix   for   most  eigenpairs
can be avoided.
Here we only focus on the  locally maximized
eigenpairs (the vectors in the  regular simplex  frame
$\mathbf v =\mathbf w_{j}, j=1, 2, \dots, n$),
whose Jacobian matrices are  easy   to  calculate, as can be seen   in
(\ref{swj})
and
(\ref{lmdwj}) in
the later part.

Thus,
the optimization landscape
presented  in Theorem  \ref{Theorem.local_eigenpair}
constitutes   the core of this work.
Once it is proved, the
proof
for 
Theorem
\ref{conjecturesimplex}
can be
easily finished.

\begin{wlremark}\label{remark.comparsion.robust}
	To this end,
	one may doubt
	the reason  for 
	analyzing local optimality
	for exploring   the robustness issue.
	Similarly,
	one  also needs to check all eigenpairs to  prove  Theorem
	\ref{Theorem.local_eigenpair}.
	However,
	as will be seen in the  later  proof,
	an equivalent  model  with  a  benign structure can be deduced
	to analyze  	 local   optimality,
	which cannot be applied if one directly deals with the robustness problem. To explain this intuitively with an explicit formula,
	we defer it to
	Remark \ref{remark.robust.formula} after we have
	finished the  proof.
	In this sense,
	even though  it is roundabout,
	efforts  in   answering local optimality
	will be
	helpful for  the 	robustness problem.
	In this process,  the  optimization landscape
	of the model
	also simultaneously   becomes   clear.
	An intuitive comparison  between local optimality and the robustness problem
	will be clearly  summarized in Table \ref{table.comparison.local.robust} later.
\end{wlremark}

\section{Proof for  Theorem \ref{Theorem.local_eigenpair} }\label{sec.proof.opt}
The  organization of the proof is   outlined  as  follows:
\begin{itemize}
	\item  in   Section  \ref{reformulated},
	we first reformulate an equivalent model for    the  regular simplex  tensor eigenpairs,
	which   has a  benign   structure
	and then will be  our focus in the  subsequent analysis.
	Their equivalence   is   theoretically  guaranteed by Lemma  \ref{lemma.coincidence.KKT}  and  \ref{lemma.wmw.k};
	\item
	in  Section   \ref{station}, we   deduce  the  structure  of  all KKT   points
	for the new model. 	
	Since the  combination
	of  $(m, n) =  (4, 3)$  serves as a special case of 
	Theorem \ref{Theorem.local_eigenpair}, it  is discussed separately in Lemma \ref{Theorem_structureofall.43} via Theorem 12 of \cite{teneigenstructure} and the developed Lemma \ref{lemma.coincidence.KKT}.
	Then,  we mainly focus on   
	the remaining combinations where $n \ge 3, m\ge3$, and $(m,n) \neq (4,3)$.
	The  structure  of  all KKT   points  for those combinations  is summarized in 	Lemma \ref{Theorem_structureofall}   
	by only utilizing    the first-order  KKT condition. Following that, 
	Lemma	\ref{vectorregularsimplexframe}
	further  specifies  the corresponding  form of  the vectors in the regular simplex  frame.
	\item
	in  Section   \ref{locallymaximized},
	to further check local optimality of each   KKT point in Lemma \ref{Theorem_structureofall},
	the  second-order    sufficient   condition
	is
	utilized  to proceed  with   the discussion.
	Several  cases that correspond to  different     structures in odd and even $m$  are separately
	discussed and arranged in three subsections,
	whose conclusions are 	presented in Lemma
	\ref{Theorem_structureoflocal},  \ref{Theorem_structureoflocaleven1}, 
	and  \ref{Theorem_structureoflocaleven2}, respectively.
	\item 	in  Section   \ref{proofmain},  by combining the above
	conclusions, 
	it is proved that the only locally maximized eigenpairs 
	are the vectors in the regular simplex frame when  $n \ge 3, m\ge3$, and $(m,n) \neq (4,3)$.
	With an   extra discussion on the special case of $(m,n) = (4,3)$
	and by  combining all cases,
	we finish the proof for    Theorem \ref{Theorem.local_eigenpair}.
\end{itemize}

For an intuitive overview, the       flowchart 
and  main theoretical results 
are   provided in Figure \ref{flow.proof_theorem}.
In addition, 
combining  Theorem
\ref{Theorem.local_eigenpair} and  Lemma \ref{RobustLocal}
to
resolve  
Theorem  \ref{conjecturesimplex}
is   also  shown intuitively.

\begin{figure}
	\centering
	\begin{tikzpicture}[node distance=1.5cm]
	\node (ori) [startstop]
	{Model (\ref{optmodelori}) \\
		($n \ge 3, m \ge3$)};
	
	\node (new) [process, below of = ori]
	{ Model	(\ref{optmodel}) \\ ($n \ge 3, m \ge3$)};

	\node (kkt43) [process, below of=new, xshift = -6cm]
	{
		Case 1:   
		$(m,n)=(4,3)$\\
		Lemma  \ref{Theorem_structureofall.43} };
	
	\node (kkt) [process, below of=new]
	{Case 2:   $(m,n) \neq (4,3)$ \\
		Lemma	\ref{Theorem_structureofall} 
	};
	
	\node (frame) [process, below of=kkt, xshift = 6cm]
	{Lemma	\ref{vectorregularsimplexframe}};
	
	\node (oddmkkt) [process, below of=kkt, xshift = -3cm]
	{ Odd $m$  \\ (\ref{u_classifyodd})};
	
	\node (evenmkkt1) [process, below of=kkt]
	{ Case 1 of  even  $m$
		\\  (\ref{u_classifyeven1})};
	
	\node (evenmkkt2) [process, below of=kkt, xshift = 3cm]
	{Case 2 of  even  $m$
		\\(\ref{u_classifyeven2})};
	
	\node (locodd) [process, below of=oddmkkt]
	{Lemma	\ref{Theorem_structureoflocal}};
	
	\node (loceven1) [process, below of=evenmkkt1]
	{Lemma \ref{Theorem_structureoflocaleven1} };
	
	\node (loceven2) [process, below of=evenmkkt2]
	{Lemma \ref{Theorem_structureoflocaleven2}};
	
	\node (stop) [startstop, below of=loceven1]
	{Proof  of  Theorem \ref{Theorem.local_eigenpair}};
	
	\node (robust) [startstop, below of=stop,xshift = -3cm]
	{Lemma  \ref{RobustLocal}};
	
	\node (Conjecture) [startstop, below of=stop]
	{Proof  of  Theorem \ref{conjecturesimplex}};
	
	\draw [arrow] (stop) --  (Conjecture);
	\draw [arrow] (robust) -- (Conjecture);
	
	\draw [arrow]   (ori) -- node [right] {		 Lemma  \ref{lemma.coincidence.KKT}/\ref{lemma.wmw.k} }  (new) ;
	\draw [arrow] (new) -- (ori);
	\draw [arrow] (new) -- (kkt);
	\draw [arrow] (new)  -- node [left=0.6cm] {Lemma  \ref{lemma.coincidence.KKT} }  (kkt43) ;
	\draw [arrow] (kkt) -- (oddmkkt);
	\draw [arrow] (kkt) -- (evenmkkt1);
	\draw [arrow] (kkt) -- (evenmkkt2);
	
	\draw [arrow]   (oddmkkt) --(locodd) ;
	\draw [arrow] (evenmkkt1) -- (loceven1);
	\draw [arrow]  (evenmkkt2) -- (loceven2);
	
	\draw [arrow]  (locodd) --(stop) ;
	\draw [arrow] (loceven1) -- (stop);
	\draw [arrow]  (loceven2) -- (stop);
	
	\draw [arrow] (kkt) -| (frame);  
	\draw [arrow]   (frame) |- (stop);
	\draw [arrow] (kkt43) |- (stop);
	\end{tikzpicture}
	\caption{
		An intuitive  
		flowchart for the proof of  
		Theorem \ref{Theorem.local_eigenpair}.
		Starting from  the  original optimization model  (\ref{optmodelori})  for $ n\ge3, m\ge3$,
	a  new  model   (\ref{optmodel})  with a  benign structure  is reformulated
	in    Section  \ref{reformulated}.
	Their equivalence can  be guaranteed by Lemma  \ref{lemma.coincidence.KKT}  and  \ref{lemma.wmw.k};
	the special combination $(m,n) = (4,3)$  is solely analyzed 
	and its conclusion is  presented in Lemma  \ref{Theorem_structureofall.43}; 
	then,  for the other combinations  with  $n \ge 3, m \ge3, (m,n) \neq (4,3)$,
	in    Section  \ref{station},	
	Lemma \ref{Theorem_structureofall}  concludes the structure of all KKT points of (\ref{optmodel}),
	including three cases listed in (\ref{u_classifyodd}), (\ref{u_classifyeven1}) and (\ref{u_classifyeven2}).
	Lemma	\ref{vectorregularsimplexframe}
	specifies  the corresponding  form of 
	the vectors in the  regular simplex  frame, 
	and is  naturally  followed  by  	Lemma \ref{Theorem_structureofall}.
	Lemma
	\ref{Theorem_structureoflocal},  \ref{Theorem_structureoflocaleven1}
	and  \ref{Theorem_structureoflocaleven2}      	 	in  Section   \ref{locallymaximized}
	correspondingly  address the local optimality of the   solutions in
	the  three cases in Lemma \ref{Theorem_structureofall}, 
	thereby  enumerating all KKT points.
	Combining those developed lemmas
	with a further discussion for the special combination of $(m,n)=(4,3)$
	completes the proof of the main Theorem \ref{Theorem.local_eigenpair}
	in  Section  \ref{proofmain}.
	Combining  Theorem
	\ref{Theorem.local_eigenpair} and  Lemma \ref{RobustLocal}
	will
	then  resolve   Theorem  \ref{conjecturesimplex}
	in  Section \ref{sec.Conjecture}.
}
	\label{flow.proof_theorem}
\end{figure}
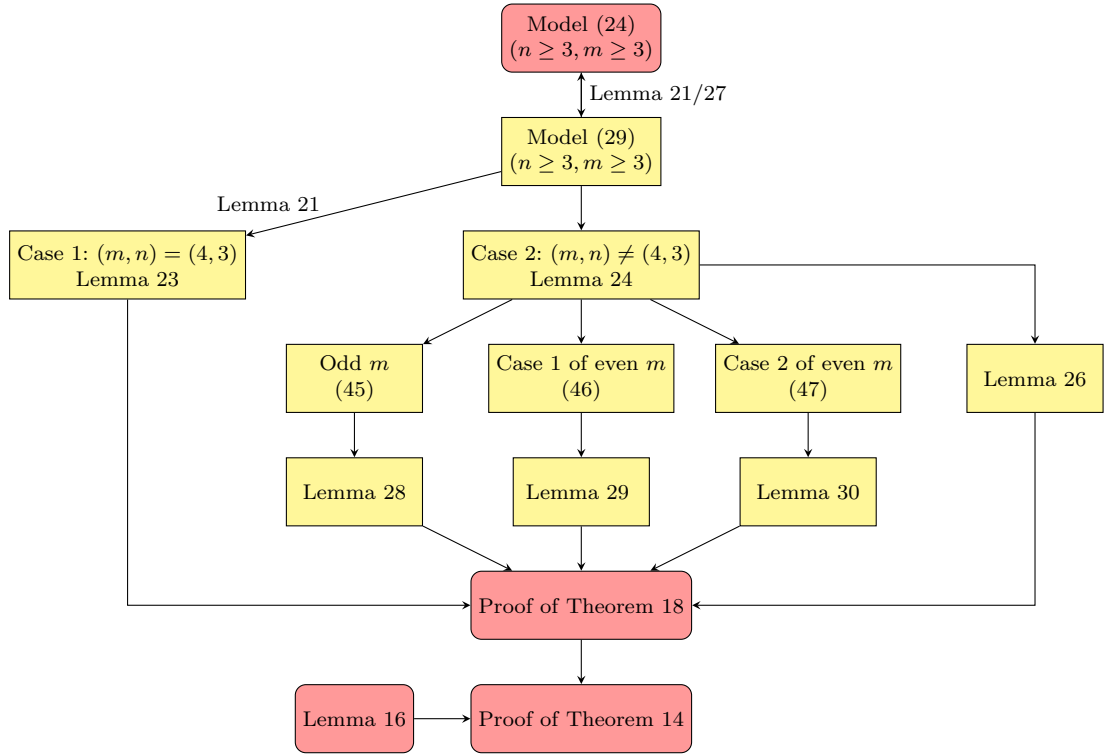

\subsection{Model reformulation   with   a  benign  structure}
\label{reformulated}

In  this  part, we   first   reformulate the   optimization  model    (\ref{opti_ori})  for   the   regular  simplex  tensor  
into  an equivalent  one.
The  details are as  follows.

Let $ n\ge3, m\ge3$.
Model   (\ref{opti_ori})  for   the  regular simplex tensor,   where  the  tensor  is  given  by  (\ref{simplextensor}),
has   the following form
\begin{align}\label{optmodelori}
\begin{cases}
\max\limits_{\mathbf v} \quad \mathcal S_{\mathbf W} \mathbf v^{m}
=
(\sum\limits_{i=1}^{n} \mathbf{w}_{i}^{\circ m}) \mathbf v^{m} \\
\text{s.t.}  \quad \mathbf v^{\mathrm {T}}\mathbf v=1
\end{cases} .
\end{align}
Directly  focusing  on
(\ref{optmodelori})
may be  a  complicated   task.
An equivalent   optimization  model  is  reformulated  to   deal  with  the  issue.
First,
the  objective function  in model   (\ref{optmodelori})   can be  rewritten as:
\begin{equation}
\mathcal{S}_{\mathbf W}   \mathbf v^{m}
=
(\sum_{i=1}^{n} \mathbf{w}_{i}^{\circ m} )\mathbf v^{m}
=
\sum_{i=1}^{n} ( \mathbf v^{\mathrm T} \mathbf{w}_{i})^{ m} .
\end{equation}

Construct a vector   $  \mathbf   u =  [u_1,u_2,\dots, u_{n}]^{\mathrm T}  \in  \mathbb {R}^{n \times 1} $  as follows:
\begin{align}\label{udenote}
u_i :=
\sqrt {
	\frac  {n-1}{ n}
}
\mathbf  v^{\mathrm T}  \mathbf w_i  ,
\quad
\mathbf   u :=
\sqrt {
	\frac  {n-1}{ n}
}
{\mathbf W}^{\mathrm T}  {\mathbf v}
.
\end{align}
It holds that
\begin{equation}\label{eq.obj.pro}
\mathcal{S}_{\mathbf W}  \mathbf v^{m}
=
(
\sqrt {
	\frac  { n}{n-1}
} )^{m}
\sum\limits_{i=1}^{n}  u_{i}^{m}
\varpropto
\sum\limits_{i=1}^{n}  u_{i}^{m}.
\end{equation}

For   the   regular  simplex frame,  the  following property  holds:
\begin{wlproperty}\label{regularproperty}
	For  the  regular  simplex   frame,  
	the null space of $\mathbf{W}^{\mathrm{T}}\mathbf{W} \in  \mathbb{R}^{n \times  n} $ and  that of
	$\mathbf{W} \in  \mathbb{R}^{(n-1) \times  n} $ 
	are both  spanned by
	$ \mathbf{1}_{n} $,
	i.e.,
	$
	\mathbf{W}^{\mathrm T} \mathbf{W}  \mathbf{1}_{n}
	=
	\mathbf{W}  \mathbf{1}_{n}
	=
	\mathbf{0}_{n-1}
	$.
\end{wlproperty}

Therefore, it can be derived
that
\begin{align}\label{utu}
\mathbf   u ^{\mathrm T} \mathbf   u
=
\frac  {n-1}{  n  }
{\mathbf v} ^{\mathrm T} {\mathbf W} {\mathbf W}^{\mathrm T}  {\mathbf v}
=
{\mathbf v} ^{\mathrm T}  {\mathbf v}
=1 ,
\quad
\mathbf   u ^{\mathrm T}
\mathbf 1_{n}
=
\sqrt {
	\frac  {n-1}{ n}
}
{\mathbf v} ^{\mathrm T} {\mathbf W}   {\mathbf 1_{n}}
=
0  ,
\end{align}
where we utilize
$     {\mathbf W} {\mathbf W}^{\mathrm T} =\frac  {  n  }{n-1}  \mathbf I_{n-1}$
and
${\mathbf W}   {\mathbf 1_{n}}
=
\mathbf 0_{n-1} $.
Then,  model     (\ref{optmodelori})
can  be  equivalently    transformed into
\begin{align}\label{optmodel}
\begin{cases}
\max\limits_{\mathbf u}  \quad  \sum\limits_{i=1}^{n}  u_{i}^{m}
\\
\text { s.t. }   \quad
\mathbf   u^{\mathrm T}   \mathbf   u =1 ,
\quad
\mathbf   u^{\mathrm T}   \mathbf   1_{n} =0
\end{cases} .
\end{align}

For   simplicity,  in the  subsequent   analysis,  we denote
$ f (\mathbf u) = \sum_{i=1}^{n}  u_{i}^{m} $,
$ g_{1}(\mathbf u) =  \mathbf u^{\mathrm T}\mathbf u -1 = 0 $,
$ g_{2}(\mathbf u) = \mathbf u^{\mathrm T}\mathbf 1_{n} = 0 $.
The  reformulated  model  can  be
understood as    a  transformation   that  transfers    the  original model  in an  $(n-1)$-dimensional  space  into
an  $n$-dimensional one  for  analysis.
The  new  constraint
$  \mathbf   u^{\mathrm T}   \mathbf   1_{n} =0 $
is  naturally  related  to   Property  \ref{regularproperty}.
An intuitive   schematic diagram
for the case of $n=3$  is  plotted  in  Fig~\ref{simplextrans}.

\begin{figure}
	\centering
	\begin{tikzpicture}[scale=1.5]
	\draw[->] (-1.2,0)  -- (1.2,0);
	\draw[->]  (0,-0.8)  -- (0,1.2);
	\fill  (0:0)  circle(0.5pt);
	
	\node (P)  at (0:1.3) {$x$} ;
	\node (P)  at (90: 1.3) {$y$} ;
	
	\fill[color=red]  (0,1)  circle(0.8pt);
	\fill[color=red]  ( 1.7321/2,  -0.5)  circle(0.8pt);
	\fill[color=red]  (-1.7321/2,  -0.5)    circle(0.8pt);

	\draw[color=black,thick] (0,1) -- ( 1.7321/2,  -0.5) ;
	\draw[color=black,thick] (0,1)  -- ( -1.7321/2,  -0.5)  ;
	\draw[color=black,thick] ( 1.7321/2,  -0.5)  --( -1.7321/2,  -0.5)  ;
	
	\end{tikzpicture}
	\begin{tikzpicture}[scale=1.5]
	\draw[->] (1,0)  -- (1.2,0);
	\draw[->]  (0,1)  -- (0,1.2);
	\draw[->]  (-0.7071   ,-0.7071  )  -- (-0.95  ,-0.95 );
	
	\draw[color=black,dashed] (0,0)  -- (1,0);
	\draw[color=black,dashed]  (0,0)  -- (0,1);
	\draw[color=black,dashed]  (0,0)  -- (-0.7071   ,-0.7071  );
	
	\fill  (0:0)  circle(0.5pt);
	
	\node (P)  at (0:1.3) {$x$} ;
	\node (P)  at (90: 1.3) {$y$} ;
	\node (P)  at (220: 1.06) {$z$} ;
	
	\fill[color=red]  (0,1)  circle(0.8pt);
	\fill[color=red]  ( 1,  0)  circle(0.8pt);
	\fill[color=red]  (-0.7071 ,  -0.7071 )    circle(0.8pt);

	\draw[color=black,thick] (0,1) -- ( 1,  0) ;
	\draw[color=black,thick] (0,1)  -- (-0.7071 ,  -0.7071 ) ;
	\draw[color=black,thick] (-0.7071 ,  -0.7071 )    -- ( 1,  0)  ;
	
	\end{tikzpicture}
	\caption{
		An intuitive   schematic diagram
		concerning the reformulation from model (\ref{optmodelori}) to (\ref{optmodel}) for the case of $n=3$: transforming the regular triangle from 2D to 3D space.
	}
	\label{simplextrans}
\end{figure}
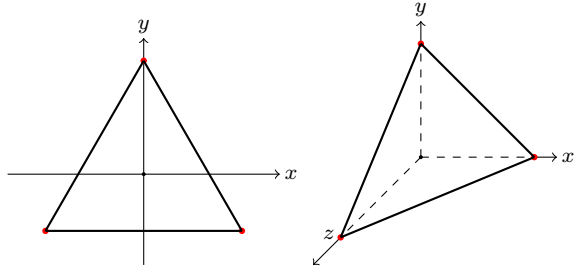

Compared to the original  model (\ref{optmodelori}),
(\ref{optmodel})
has  the  following advantage:
the objective function $f (\mathbf u)$  is
separable  with respect to    its   $n$ variables  $u_{i}$,
which
has a   benign
structure  for  analyzing and   deriving 
the  solutions.
Therefore, in the  following   sections  \ref{station} and \ref{locallymaximized},
the structure  of   all  KKT   and  locally optimal  solutions  $\mathbf u$ of
(\ref{optmodel})
is investigated  to   carry out the  subsequent analysis.
Correspondingly,
in the process,
the correspondence between the KKT and locally optimal solutions of (\ref{optmodelori}) and those of (\ref{optmodel})
will be demonstrated.
Details are given in  Lemma  \ref{lemma.coincidence.KKT}  and   \ref{lemma.wmw.k}.

\subsection{Structure  of  all  KKT points of  (\ref{optmodel})}\label{station}

The Lagrangian function  of the  reformulated   model    (\ref{optmodel})    is
defined as:
\begin{equation}\label{Lagrangianfunctionnew}
L :=  L(\mathbf u, \alpha, \beta)=
\frac{1}{m}
\sum\limits_{i=1}^{n}  u_{i}^{m}
+
\frac{\alpha}{2}
(1- \mathbf   u^{\mathrm T}   \mathbf   u  )
+
\beta (0- \mathbf   u^{\mathrm T}   \mathbf   1_{n}),
\end{equation}
where $ \alpha,\beta$ are  the  corresponding    Lagrange   multipliers of  the two  constraints.

To  obtain all KKT  points, one can  check  the first-order  KKT  condition
of model  (\ref{optmodel}).
When
$ \triangledown_{\mathbf u}L = \mathbf 0$,
KKT points of  (\ref{optmodel})
simultaneously
satisfy
the  following  three  equations:
\begin{align}
\label{KKTgradient}
& \mathbf u ^{\circledast^{m-1}}-\alpha \mathbf u-\beta \mathbf  1_{n} = \mathbf 0 ,
\\
\label{KKTcon1}
& \mathbf   u^{\mathrm T}   \mathbf   u =1 ,
\\
\label{KKTcon2}
& \mathbf   u^{\mathrm T}   \mathbf   1_{n} =0.
\end{align}

By multiplying 
both sides of (\ref{KKTgradient}) 
by $\mathbf{u}^{\mathrm{T}}$ (or $\mathbf{1}_n^{\mathrm{T}}$) and using (\ref{KKTcon1}), (\ref{KKTcon2}), it holds that
\begin{align}\label{alphares}
\alpha
:= \alpha (\mathbf u)
=
\sum\limits_{i=1}^{n}  u_{i}^{m}
=
f(\mathbf u),
\quad
\beta
:= \beta (\mathbf u)
=
\frac 1 n  \sum\limits_{i=1}^{n}  u_{i}^{m-1}.
\end{align}
Note that both
$\alpha$ and $\beta$
are  functions of $\mathbf u$.
For simplicity, we omit the variable  $\mathbf u$ in their  notations.

\subsubsection{Equivalence between  the KKT points of     (\ref{optmodelori})  and 
	those of	(\ref{optmodel})}

After presenting the KKT conditions of the reformulated model,
we  clarify the equivalence between  the KKT points of    (\ref{optmodelori})  and  	 those of 	(\ref{optmodel})
via   the following lemma:
\begin{wllemma}\label{lemma.coincidence.KKT}
	Let $\mathcal{S}_{\mathbf{W}}\in T^{m}(\mathbb{R}^{n-1})$ be a regular simplex tensor with $n\geq 3$, $m\geq 3$. 
	There exists 
	a   one-to-one correspondence between the 	 KKT points of     (\ref{optmodelori})  and  
	those of 	(\ref{optmodel}).
\end{wllemma}

\begin{proof}
	For model (\ref{optmodelori})  and 	(\ref{optmodel}),
	by  (\ref{eq.obj.pro}),  their  objective functions are proportional to each other;
	by  (\ref{udenote})  and (\ref{utu}),  their  constraints can be mutually  transformed.
	For a  better and  stricter   discussion  of their KKT conditions,
	it is proved  in  two folds.
	First, 	when
	$ \mathbf v$ is a feasible KKT solution of  (\ref{optmodelori}),
	it holds that
	\begin{align}
	\mathcal S_{\mathbf W} \mathbf v^{m-1}=\lambda \mathbf v,
	\quad
	\mathbf v^{\mathrm T}	\mathbf v=1.
	\end{align}
	It is easy to verify that 	the corresponding $\mathbf u$  calculated by
	(\ref{udenote})
	will satisfy (\ref{KKTcon1})  and
	(\ref{KKTcon2}), which can be seen from (\ref{utu}).
	To further show (\ref{KKTgradient}) holds,
	$\mathcal S_{\mathbf W} \mathbf v^{m-1}$ can be rewritten as
	\begin{align}\label{swvm1}
	\mathcal{S}_{\mathbf W}  \mathbf v^{m-1}
	& =
	(
	\sum_{i=1}^{n} \mathbf{w}_{i}^{\circ m}
	)  \mathbf v^{m-1}
	=
	\sum_{i=1}^{n} (  \mathbf w_{i}^{\mathrm T} \mathbf  v)^{ m-1}
	\mathbf{w}_{i}
	\nonumber \\
	& =
	(
	\sqrt {
		\frac   { n}{n-1}
	} )^{m-1}
	\sum_{i=1}^{n}  u_i ^{ m-1}
	\mathbf{w}_{i}
	=
	(
	\sqrt {
		\frac   { n}{n-1}
	} )^{m-1}
	\mathbf  W
	\mathbf u ^{\circledast^{m-1}}.
	\end{align}
	Combining $  \mathbf   u :=
	\sqrt {
		\frac  {n-1}{ n}
	}
	{\mathbf W}^{\mathrm T}  {\mathbf v} $  in  	(\ref{udenote})
	with
	$     {\mathbf W} {\mathbf W}^{\mathrm T} =\frac  {  n  }{n-1}  \mathbf I_{n-1}$
	yields
	$  \mathbf   v=
	\sqrt {
		\frac  {n-1}{ n}
	}
	{\mathbf W}  {\mathbf u} $.
	(\ref{eq.obj.pro})  and (\ref{alphares})     imply  that $ \lambda = \mathcal{S}  \mathbf v^{m}
	=
	(
	\sqrt {
		\frac  { n}{n-1}
	} )^{m} \alpha$.
	Then, it holds that
	\begin{align}
	\mathcal S_{\mathbf W} \mathbf v^{m-1}-\lambda \mathbf v
	& =
	(
	\sqrt {
		\frac   { n}{n-1}
	} )^{m-1}
	\mathbf  W
	\mathbf u ^{\circledast^{m-1}}
	-(
	\sqrt {
		\frac  { n}{n-1}
	} )^{m} \alpha
	\sqrt {
		\frac  {n-1}{ n}
	}
	{\mathbf W}  {\mathbf u}
	\\ \nonumber
	& =
	(
	\sqrt {
		\frac   { n}{n-1}
	} )^{m-1}
	(
	\mathbf  W
	\mathbf u ^{\circledast^{m-1}}
	- \alpha
	{\mathbf W}  {\mathbf u}
	),
	\end{align}
	which implies that
	$\mathbf  W
	\mathbf u ^{\circledast^{m-1}}
	- \alpha
	{\mathbf W}  {\mathbf u} =  \mathbf 0$.
	Observe that the desired (\ref{KKTgradient}) does not contain $\mathbf W$.
	To further eliminate this term,
	we	multiply   $\mathbf W^{\mathrm T}$ on both sides of $\mathbf  W
	\mathbf u ^{\circledast^{m-1}}
	- \alpha
	{\mathbf W}  {\mathbf u} =  \mathbf 0$.
	By the equiangular property of the matrix $\mathbf W$,
	it holds that
	\begin{align}
	(\mathbf W^{\mathrm T}\mathbf W)_{ij}=
	\begin{cases}
	1 ,  \quad  \quad    i=j             \\
	-\frac{1}{n-1}  ,  \quad  i  \neq  j \\
	\end{cases}  .
	\end{align}
	Therefore,
	$	\mathbf W^{\mathrm T}\mathbf W$ can be decomposed as
	$ 	\mathbf W^{\mathrm T}\mathbf W = -\frac{1}{n-1}  \mathbf 1_{n}\mathbf 1_{n}^{\mathrm T}   +
	\frac{n}{n-1}    \mathbf I_{n}$.
	Then, we have that
	\begin{align}\label{wtwum1}
	\mathbf 0
	& \nonumber=
	\mathbf W^{\mathrm T} \mathbf  W
	\mathbf u ^{\circledast^{m-1}}
	- \alpha
	\mathbf W^{\mathrm T}  \mathbf W  {\mathbf u}
	\\ \nonumber
	& =
	(-\frac{1}{n-1}  \mathbf 1_{n}\mathbf 1_{n}^{\mathrm T}   +
	\frac{n}{n-1}    \mathbf I_{n})   ( \mathbf u ^{\circledast^{m-1}} - \alpha \mathbf u )
	\\
	& =
	-\frac{1}{n-1}   \mathbf 1_{n}\mathbf 1_{n}^{\mathrm T} \mathbf u ^{\circledast^{m-1}}
	+	\frac{n}{n-1}   \mathbf u ^{\circledast^{m-1}}
	- \frac{n}{n-1}  \alpha \mathbf u .
	\end{align}
	Substituting
	$ \mathbf 1_{n}^{\mathrm T}
	\mathbf u ^{\circledast^{m-1}}
	=
	\sum\limits_{i=1}^{n}  u_{i}^{m-1}
	=
	n \beta
	$
	into (\ref{wtwum1})
	yields   the desired  (\ref{KKTgradient}).
	We can conclude that the corresponding $\mathbf u$ is also a KKT point
	of  	(\ref{optmodel}).

	Conversely,
	if $\mathbf u$ is  a KKT point
	of  	(\ref{optmodel}),
	the above  analysis can be  followed   
	from setting  $\mathbf u ^{\circledast^{m-1}}
	-   \alpha \mathbf u - \beta \mathbf 1_n =\mathbf 0$
	to  deriving   $\mathcal S_{\mathbf W} \mathbf v^{m-1}=\lambda \mathbf v$.
	The constraint for $\mathbf v$ can be easily checked by (\ref{utu}).
	The conclusion is proved as claimed. $\blacksquare$
\end{proof}

\begin{figure}[t]
	\centering
	\subfigure[]
	{
		\includegraphics[width=0.33\textwidth]{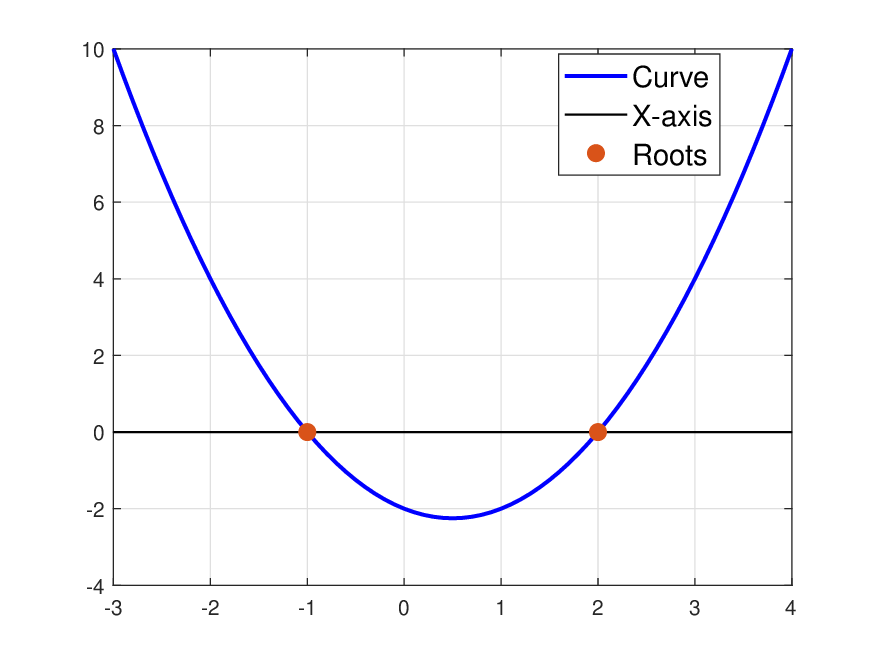}
		\label{odd}
	}
	\subfigure[]
	{
		\includegraphics[width=0.33\textwidth]{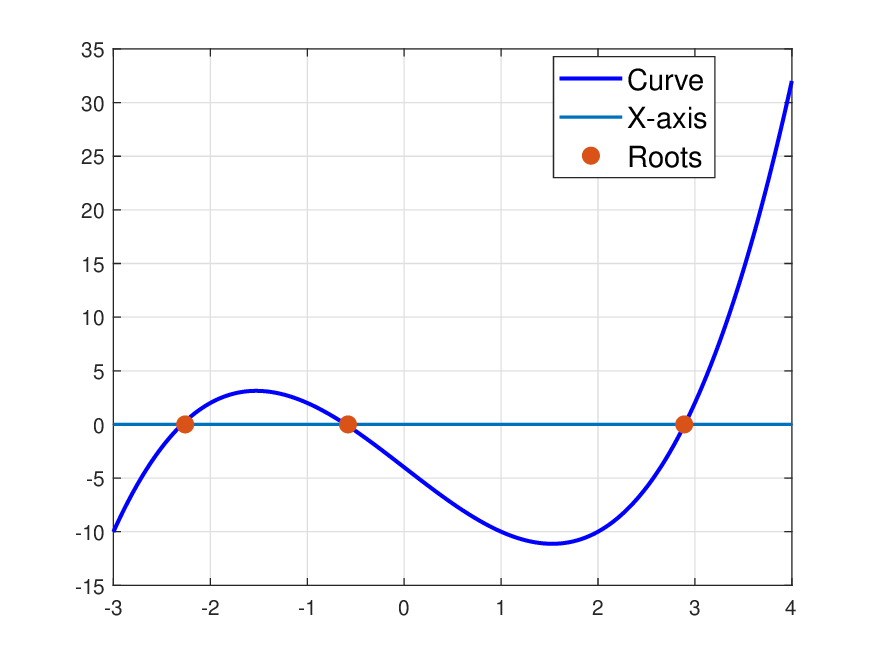}
		\label{even}
	}
	\caption{	The   schematic   figures  
		of   the  function    $  f(x) = x^{m-1} -\alpha x - \beta$ for
		specified  odd and even cases:  a)   the  odd  $m$  case: $m=3, \alpha=1, \beta =2$;   b)   the even   $m$  case:
		$m=4, \alpha=7, \beta =4$.	
	}
		\label{curveplot}
	\end{figure}

\subsubsection{Main results}
 With this equivalence guarantee, we focus on the 
 	 structure of the solutions  of  	(\ref{optmodel}). 
	Assume that  the pair $(\mathbf u, \alpha ,\beta)$   is a feasible solution
	that satisfies (\ref{KKTgradient}),  (\ref{KKTcon1})  and
	(\ref{KKTcon2}).
		 To be consistent with  the selection up to sign equivalence in Lemma  \ref{lemma.sign} and  Remark  \ref{rem.signnew},
	it is necessary to first
	discuss the signs  of $\alpha $ and  $\beta $  for  different orders $m$.
	Analogous to  Lemma  \ref{lemma.sign}, it is stated as the following lemma:

\begin{wllemma}\label{rem.sign.alpha} 	
		Let $\mathcal{S}_{\mathbf{W}}\in T^{m}(\mathbb{R}^{n-1})$ be a regular simplex tensor with $n\geq 3$, $m\geq 3$. 
		Concerning  its equivalent model 	(\ref{optmodel})
		and the feasible  solution pair  $(\mathbf u, \alpha ,\beta)$, 
		then, 	the following statements hold:	
		\\
		(1):   $(\mathbf u, \alpha ,\beta)$   is a solution  of 
		(\ref{optmodel})   if and only if 
		$( -\mathbf u, (-1)^{m} \alpha, (-1)^{m-1}\beta)$  is  a solution  of 
		(\ref{optmodel}).
		\\
		(2):
		When 
		$m$  is odd, the solution pair  $(\mathbf u, \alpha ,\beta)$   can always be chosen  such that  
		\(  \alpha  \geq 0 \), and it always holds 	\(  \beta  > 0 \); 
		when $m$ is even, 
		it  must satisfy \( \alpha  > 0 \), 
		and  the solution pair  $(\mathbf u, \alpha ,\beta)$   can always be chosen  such that  
		\(  \beta  \geq 0 \). 
\end{wllemma}

\begin{proof}	 
		For the first claim, it is straightforward to verify that
		if $(\mathbf u, \alpha ,\beta)$  is a feasible solution
		satisfying (\ref{KKTgradient}),  (\ref{KKTcon1})  and
		(\ref{KKTcon2}),
		then for  $ -\mathbf u$,
		we have 
		\begin{align}
		\alpha (-\mathbf u)  =  (-1)^{m} \alpha (\mathbf u),
		\quad 
		\beta (-\mathbf u) =  (-1)^{m-1} \beta (\mathbf u),
		\quad 
		(-\mathbf u )^{\circledast^{m-1}} = (-1)^{m-1} \mathbf u ^{\circledast^{m-1}}.
		\end{align}
		Thus, $(\mathbf u, \alpha ,\beta)$ and  	$( -\mathbf u, (-1)^{m} \alpha, (-1)^{m-1}\beta)$
		can be regarded as equivalent,
		which is analogous to the first claim in  Lemma  \ref{lemma.sign}.

		Concerning the second claim, 
		recall
		the choice  of $\lambda \ge  0$   stated in Lemma  \ref{lemma.sign} and Remark  \ref{rem.signnew},
		by  (\ref{eq.obj.pro})  and (\ref{alphares}), we have  $ \lambda = \mathcal{S}_{\mathbf W}  \mathbf v^{m}
		=
		(
		\sqrt {
			\frac  { n}{n-1}
		} )^{m} \alpha$.
		So the selection   $\alpha \ge 0$ is also natural
		and  consistent  for simplifying the sign equivalence  of solutions. 
		In addition, 
		when $m $  is even, for this special regular simplex tensor, we can obtain a stronger conclusion that  $\alpha \propto  \lambda > 0$.
		Indeed, 
		if 	$ \lambda  = \mathcal S_{ \mathbf W}  \mathbf v^{m}    =
		\sum_{i=1}^{n} ( \mathbf  {w}_{i}^{\mathrm T} \mathbf v)^{ m} = 0$,
		then for even $m$, we must have  $ \mathbf  {w}_{i}^{\mathrm T} \mathbf v =0$ for all $i$.
		However,  due to the highly symmetric structure of $\mathcal{S}_{\mathbf W}$,   there is no  $\mathbf v$
		that is orthogonal to all vectors in $ \mathbf W$.	
		We now discuss the sign of $\beta$ for different orders $m$:
		\begin{itemize}
			\item  When  $m$  is odd,
			we have $u_{i}^{m-1} \ge 0$ for all $i$.
			Since $\beta$ is the sum of  $n$ non-negative  terms,  it  holds that  $\beta \ge 0$.
			However, from the  constraints  in  (\ref{KKTcon1})  and
			(\ref{KKTcon2}), not all $u_{i}$ can be zero simultaneously.
			Therefore, we  conclude that $\beta > 0$.
			\item  When  $m$   is even,
			the sign of each term $u_{i}^{m-1}$
			depends on that of $u_{i}$,  so the sign of $\beta$ is not fixed a priori.
		As stated in	 Lemma  \ref{lemma.sign} and Remark  \ref{rem.signnew}, we conventionally  select  $(\lambda,\mathbf v)$   for analysis.
			Here, $\alpha$ plays a role similar to that of $\lambda$.
			For  $\mathbf u$, we always have  $ \sum_{i} u_{i}=\mathbf   u^{\mathrm T}   \mathbf   1_{n} = 0$.	
			Therefore,  to be consistent with the selection of the pair  $(\lambda,\mathbf v)$, we restrict  the sign of $\beta$ so that the pair  $(\mathbf u, \alpha ,\beta)$  is reserved.
			Then we have   $\beta \ge 0$.
			Here, in contrast to the odd case,  $\beta =0$ can occur.
		\end{itemize}
		The proof is complete. $\blacksquare$
\end{proof}

With the analysis concerning the  signs of $\alpha $ and  $\beta $ in different $m$,
	 we also select only   one representative form  for each  solution up to sign equivalence, to be consistent with 
	tensor eigenpairs convention.
Then,  as can be  observed  from (\ref{KKTgradient}),  each $u_{i}$  is  one of  the roots  of   the   $(m-1)$-order polynomial
$ u_{i}^{m-1} -\alpha u_{i} - \beta =0  $.
Note that  in  this  paper,  we  only  discuss  the  real  solutions
as  clarified  in Remark \ref{rem.signold}, consistent with the previous work \cite{teneigenstructure}.
 In addition, 
	recall that the combination $(m,n) = (4,3)$ serves as a special case of the original model (\ref{optmodelori}), where the objective value 
	$
	\mathcal{S}_{\mathbf{W}} \mathbf{v}^{m}
	$ is a constant, as commented in Remark \ref{RemarkScope}. 
	Naturally, this also holds for the reformulated model (\ref{optmodel}). 
	So in the following, we deal with this exceptional case first. 

\begin{wllemma}\label{Theorem_structureofall.43} 
		Let $(m,n) = (4,3)$. Consider the real KKT points of model (\ref{optmodel})
		satisfying (\ref{KKTgradient}), (\ref{KKTcon1}), and
		(\ref{KKTcon2}), denoted by $\mathbf{u} = [u_1, u_2, \dots, u_n]^{\mathrm{T}} \in \mathbb{R}^{n \times 1}$, with the sign equivalence analyzed in Lemma \ref{rem.sign.alpha}.
		Then any feasible $\mathbf{u}$  with 
		$\alpha > 0 $ and $\beta \ge  0 $
		is a solution of model (\ref{optmodel}), and the objective value $f(\mathbf{u})$   is a constant  $\frac{1}{2}$.
\end{wllemma}

\begin{proof}
	 
		We utilize the conclusion in Theorem 12 of \cite{teneigenstructure} for model (\ref{optmodelori}),
		which states that every feasible $\mathbf{v}$, up to sign equivalence, 
		is an eigenvector of $\mathcal{S}_{\mathbf{W}}$ with eigenvalue $\frac{9}{8}$, i.e., 
		$\mathcal{S}_{\mathbf{W}} \mathbf{v}^{m-1} = \frac{9}{8} \mathbf{v}$.
		Based on the equivalence result in Lemma \ref{lemma.coincidence.KKT}, 
		when $\mathbf{v}$ is a feasible KKT solution of (\ref{optmodelori}),
		the corresponding $\mathbf{u}$ calculated by
		(\ref{udenote}) is also a KKT point
		of (\ref{optmodel}). 
		Since this holds for every feasible 
		$\mathbf{v}$, it follows that every feasible 
		$\mathbf{u}$
		 satisfying the constraints  is a KKT point of (\ref{optmodel}).
		 To be consistent with the sign convention established in Lemma  \ref{rem.sign.alpha}
	for even $m$,  we note that for any feasible solution 
$(\mathbf u, \alpha ,\beta)$, its sign-equivalent counterpart 
$(-\mathbf u, \alpha ,-\beta)$ is also a solution. 
Therefore, without loss of generality, we may restrict our attention to the representative with 
$\beta \ge  0 $, while 
$\alpha > 0 $ is automatically guaranteed by Lemma \ref{rem.sign.alpha}. Hence, any feasible $\mathbf u$
 satisfying the sign conditions is indeed a valid KKT point of model (\ref{optmodel}). 
		We then derive the objective value $f(\mathbf{u})$. 
		By (\ref{eq.obj.pro}) and (\ref{alphares}), we have $\mathcal{S}_{\mathbf{W}} \mathbf{v}^{m}
		=
		\left(
		\sqrt{
			\frac{n}{n-1}
		} \right)^{m} \alpha$.
		This means that the objective value of the reformulated model (\ref{optmodel}), i.e., $f(\mathbf{u}) = \alpha$, is also a constant for $(m,n) = (4,3)$.
		Since $\mathcal{S}_{\mathbf{W}} \mathbf{v}^{m} = \frac{9}{8}$ for every feasible $\mathbf{v}$ based on  Theorem 12 of \cite{teneigenstructure},
		we can similarly derive that $\alpha = \frac{1}{2}$ for every feasible $\mathbf{u}$. 
		Such a conclusion can also be derived by directly using the constraints $\mathbf   u^{\mathrm T}   \mathbf   u =1 ,
	$ and 
		$ \mathbf   u^{\mathrm T}   \mathbf   1_{n} =0$. 
		When $ n=3$, for  $ u_1+u_2+u_3=0$ and $ u_1^2+u_2^2+u_3^2=1$,
		it holds that $ (u_1+u_2+u_3)^2= u_1^2+u_2^2+u_3^2 + 2(u_1u_2+u_1u_3+u_2u_3) = 0$.
		Thus, we have that $u_1u_2+u_1u_3+u_2u_3 = -\frac 1 2$.
	Consequently, 
	it yields that 	$\frac 1 4 =  (u_1u_2+u_1u_3+u_2u_3)^2 =u_1^2u_2^2+u_1^2u_3^2+u_2^2u_3^2 + 2 
		u_1u_2u_3(u_1 + u_2+u_3) =u_1^2u_2^2+u_1^2u_3^2+u_2^2u_3^2 $.
	Then, by calculating $ (u_1^2+u_2^2+u_3^2)^2= u_1^4+u_2^4+u_3^4 + 2(u_1^2u_2^2+u_1^2u_3^2+u_2^2u_3^2) = 1$,
	we can derive that  $f(\mathbf u) = u_1^4+u_2^4+u_3^4  = \frac 1 2 $ for $(m,n) = (4,3)$.
		$\blacksquare$
\end{proof}

 This lemma implies that  up to sign equivalence,  every feasible $\mathbf{u}$ for $(m,n) = (4,3)$ 
	is both locally maximized and minimized. There is no need for subsequent local optimality
	identification. 
	Then, we are more interested in the structure of all KKT points where $n \ge 3$, $m \ge 3$, and $(m,n) \neq (4,3)$, which 
	is the main task we aim to resolve.
	
Consider   the  function   $  f(x) = x^{m-1} -\alpha x - \beta$,
its  first-order and second-order 
 derivatives are given by
\begin{align*}
	f^{\prime}(x) = (m-1)x^{m-2} -\alpha,
	\quad
	f^{\prime\prime}(x) = (m-2)(m-1)x^{m-3} .
\end{align*}
Similarly,
the odd and even $m$ cases are separately discussed as follows:
\begin{itemize}
	\item  When  $m$  is odd,
	      since $\alpha \ge  0$,
	      $f^{\prime}(x) =0$
	       yields the only   real root which is given by
	      \begin{align}\label{fx0roots}
		      x_0  =
		      \sqrt[\uproot{3} {m-2}]
		      {
			      \frac{\alpha}{m-1}
		      }    \ge   0.
	      \end{align}
	      The fact that  $f^{\prime\prime}(x) \ge 0 $ for any $x$
	      implies
	      that  $f^{\prime}(x)$ is an increasing function.
	      Then,
	      it holds that
	      $f^{\prime}(x) < 0$ for $x \in   (-\infty,x_0)$,
	      and
	      $f^{\prime}(x) > 0$ for $x \in  [x_0, \infty)$.
	      Therefore,
	      $f(x) $
	      decreases in the range  $x \in  (-\infty,x_0)$,
	      and increases in the range   $x \in  [x_0, \infty)$.
	           Since  $f(0) = -\beta < 0$, it holds that $f(x_0) \le   f(0) < 0$. 
	      Naturally,    there are   always  two  real  roots  for  $ x^{m-1} -\alpha x - \beta =0  $.  
	      	It holds  that
	       one of the roots is positive and the other is negative;
	\item When  $m$  is even,       since $\alpha >   0$,
	      there are two real roots for 	$f^{\prime}(x) =0$,
	      which are given by
	      \begin{align}\label{fx0rootseven}
		      x_1  = -x_0=
		      - \sqrt[\uproot{3} {m-2}]
		      {
			      \frac{\alpha}{m-1}
		      }   <0 ,
		      \quad
		      x_0  =
		      \sqrt[\uproot{3} {m-2}]
		      {
			      \frac{\alpha}{m-1}
		      }    > 0.
	      \end{align}
	      It holds that
	      $f^{\prime}(x) > 0$ for $ x \in   (- \infty,x_1]$,
	      $f^{\prime}(x) < 0$ for $ x \in  (x_1, x_0)$,
	      and
	      $f^{\prime}(x) > 0$ for $x \in     [x_0, \infty)$.
				Therefore,
			     $f(x) $
			 increases in the range  $x \in  (-\infty,x_1]$,
	      decreases in  the range 	$ x \in  (x_1,x_0)$,
	      and increases in   the range  $x \in  [x_0, \infty)$.
 When $\beta = 0$, the equation  
$
x^{m-1} - \alpha x - \beta = 0
$
has three roots, given by  
\begin{align}\label{equ.threeroot}
0, \quad
\sqrt[\uproot{3} {m-2}]{\alpha}, \quad
-\sqrt[\uproot{3} {m-2}]{\alpha}.
\end{align}
When $\beta > 0$, we have $f(0) = -\beta < 0$, and $f(x_0) < f(0) < 0$.  
We need to examine the sign of  
$
f(x_1) = 
\left( -\sqrt[\uproot{3} {m-2}]{\frac{\alpha}{m-1}} \right)^{m-1}
+ \alpha \sqrt[\uproot{3} {m-2}]{\frac{\alpha}{m-1}} - \beta.
$
If $f(x_1) < 0$, then $f(x)=0$ has only one root.  
Considering this for every $u_i$, this would contradict the constraints on $\mathbf u$.  
If $f(x_1) = 0$, it yields one negative root and one positive root for $f(x)=0$.  
If $f(x_1) > 0$, it yields two negative roots and one positive root, as can be observed from Fig.~\ref{even}.  
Therefore, to summarize these cases and to be consistent with the constraints on $\mathbf u$, we require that $f(x_1) \ge 0$, which leads to  
\begin{align}\label{equ.alpha.beta.sign}
0 \le \beta \le 
\frac{(m-2)\alpha}{m-1}
\left( \frac{\alpha}{m-1} \right)^{\frac{1}{m-2}}.
\end{align}
\end{itemize}

 For an intuitive illustration,
the  function  curves
 of   the  form
$  f(x) = x^{m-1} -\alpha x - \beta$
for  two specified odd  and  even    cases
are   plotted in  Fig 	\ref{curveplot},
whose results are  consistent with our above analysis.
Therefore,  in the following,  
we  will  separately discuss  the odd and  even   $m$ cases.
Note that  such  an  observation  and  classification  rule is  consistent with the  discussion  presented in
\cite{teneigenstructure}.
The  differences   lies in the fact    that
\cite{teneigenstructure}  focuses   on the original model (\ref{optmodelori}), while
this paper analyzes the  reformulated one (\ref{optmodel}).

	Then,  regarding  the structure of all KKT points where $n \ge 3$, $m \ge 3$, and $(m,n) \neq (4,3)$,  
 the following lemma is presented:

\begin{wllemma}\label{Theorem_structureofall}
	 
Let $n \ge 3$, $m \ge 3$, and $(m,n) \neq (4,3)$. Consider the real KKT points of model (\ref{optmodel})
satisfying (\ref{KKTgradient}),  (\ref{KKTcon1})  and
(\ref{KKTcon2}), denoted by $\mathbf{u} = [u_1, u_2, \dots, u_n]^{\mathrm{T}} \in \mathbb{R}^{n \times 1}$, with the sign equivalence analyzed in Lemma \ref{rem.sign.alpha} and permutation equivalence up to a column permutation matrix $\mathbf{P} \in \mathbb{R}^{n \times n}$.

	(1):
	When  $m$  is  odd
	 	with 
 $\alpha \ge 0 $ and $\beta >0 $  in Lemma  \ref{rem.sign.alpha},  			it holds that 
	\begin{equation}\label{u_classifyodd}
		\mathbf u
		=
			[
				a \mathbf 1_{k}^{\mathrm T},
				b \mathbf 1_{n-k}^{\mathrm T}
			]^{\mathrm T},
	\end{equation}
where $k$ is an integer with $1 \le k \le \lfloor n/2 \rfloor$,
and  $\lfloor n/2  \rfloor$  denotes  the
integer
that is less than or equal to $ n/2$,
$ a=   \sqrt{	\frac{n-k} {k n}		} >0 $
and
$  b=-  \sqrt{	\frac{k} {(n-k) n}	}   <0 $ vary with $k$.

	(2):
	When  $m$  is  even	
		 with 
	$\alpha > 0 $ and $\beta \ge  0 $  in Lemma  \ref{rem.sign.alpha} and (\ref{equ.alpha.beta.sign}) is satisfied, 	it holds that  either
	\begin{equation}\label{u_classifyeven1}
		\mathbf u
		=
				[
					a \mathbf 1_{k}^{\mathrm T},
					b \mathbf 1_{n-k}^{\mathrm T}
			]^{\mathrm T},   
	\end{equation}
	 	where the parameter setting is the same as  the above  (\ref{u_classifyodd}), 
	or
	\begin{equation}\label{u_classifyeven2}
		\mathbf u
		=
			[
				c \mathbf 1_{p}^{\mathrm T},
				d \mathbf 1_{q}^{\mathrm T},
				e \mathbf 1_{s}^{\mathrm T}
			]^{\mathrm T} ,
	\end{equation}
	    where $p, q, s \in \mathbb{Z}^+$ and $1 \le p, q, s \le n-1$,     $p+q+s=n$, 	
	$e < d	  \le   0 <c$, 
	and it must satisfy  $ pc+qd+se= 0$ and $ pc^{2}+qd^{2}+se^{2} = 1$  by  (\ref{KKTcon1})  and
	(\ref{KKTcon2}), and 
	  $ n\beta = pc^{m-1}+qd^{m-1}+se^{m-1} \ge 0$ by Lemma  \ref{rem.sign.alpha},
	  and $ \beta <    \frac{(m-2)\alpha}{m-1}
	  (\frac{\alpha}{m-1})^{ \frac{1}{m-2}} $ 
	  to ensure three distinct roots according to  (\ref{equ.alpha.beta.sign}),
	  where $ \alpha= pc^{m}+qd^{m}+se^{m}$. 
	  When $ d =0$, it holds that $\beta =0$, $p=s$, and $c=-e = \sqrt {\frac {1}{2p}}$. 
\end{wllemma}

\begin{proof}
	Equivalently,  for (\ref{KKTgradient}), there   are   $n$  equations    of
	  the  form
	\begin{equation}\label{uim1alpha}
		u_{i}^{m-1} -\alpha u_{i} - \beta =0  \quad  (i=1,2, \dots, n).
	\end{equation}

	In   the case of (1),
	when $m$ is  odd,   under the sign convention established in Lemma  \ref{rem.sign.alpha},
	 we have $\alpha \ge 0$ and 
	$\beta > 0$. Consequently, the  equation  (\ref{uim1alpha})
	admits exactly two distinct real roots:
	$
	a > 0 $ and  $ b < 0$.
	It can be  checked that
	all  $n$
	variables
	$u_{i},i=1,2,\dots,n$   cannot all take the same value $a$ or $b$.
	Otherwise, we  will
	derive    $ u_{1} =u_{2} = \cdots = u_{n} = 0 $, which  
	 contradicts  the  constraints  in  (\ref{optmodel}).
	Therefore, 
	 at  most  $n-1$
	variables
	$u_{i}$  
	can  take the same value. 
	For  example,
	assume  that  the  first  $k$ ($ 1 \le  k  \le  n-1 $)
	variables  are    selected to be the positive one,
	and
	the last   $n-k$  variables take     the  negative one.
	The  form of the  solution  is  given by
	\begin{align}\label{ab_struc}
		 & u_{1}=u_{2} \cdots=u_{k}  =a > 0  ,
		 & u_{k+1}=u_{k+2}  \cdots=u_{n} =b < 0.
	\end{align}

	By considering  the  two constraints  in  (\ref{optmodel}),
	it implies that
	\begin{equation}\label{ab_expreess}
		\begin{cases}
			k a^{2}+(n-k) b^{2}=1 \\
			k a+(n-k) b=0
		\end{cases}.
	\end{equation}
	Then, it can be derived that
	\begin{equation}\label{ab_solu}
		\begin{cases}
			a=   \sqrt{	\frac{n-k} {k n}		} \\
			b=-  \sqrt{	\frac{k} {(n-k) n}	}
		\end{cases},
	\end{equation}
	where
	$ a, b  $ vary with $k$,  and  for simplicity, we omit the subscript $k$ in both 
	 $ a$ and $ b$.
Furthermore, 
	recall the equivalence between  $(\mathbf u, \alpha ,\beta)$ and  	$( -\mathbf u, -\alpha, \beta)$
	for odd $m$ stated in Lemma \ref{rem.sign.alpha}, 
	we need to discuss the 
	 range  of $k$ to ensure the setting of $\alpha \ge 0$.
For  odd $m  \ge 3$, it holds that 
$	\alpha
=
\sum\limits_{i=1}^{n}  u_{i}^{m}
=
ka^m+(n-k)b^m
$.
  It can be observed that when $n$ is even, and $k= n/2$, it holds that $a = -b$ and $\alpha =0$. 
  Using the fact that $ab=   - \frac  1n $  from the  above equation, 
  we can derive that 
  $	\alpha
  =
ka^m-  \frac  {n-k} {n^m a^m}
$.
 Since  $  a^m > 0$, it holds that 
$ ka^{2m}n^m -(n-k) \ge 0$
and 
$ \frac { k(n-k)^{m}n^m} {k^m n^m} -(n-k) \ge 0$, 
which implies that 
$ \frac { (n-k)^{m-1}} {k^{m-1}}  \ge 1$
and $ n-k \ge k$. 
For general $n$ and $k$,  when 
 we 
  constrain   $1 \le  k  \le  \lfloor n/2  \rfloor$,
$	\alpha \ge 0$ can be ensured. 
The solutions
where $ \lfloor n/2  \rfloor  \le  k  \le  n-1$
are  equivalent   up to sign equivalence.  
	Then,  the  structure of  all   solutions
	can be  denoted  by
	permuting the  indices
	of  $k$ 
	  identical   components     by  a  permutation matrix $\mathbf P$, which then
	can be  expressed 
	  in the form of   (\ref{u_classifyodd})  up to  the specified   sign and  permutation equivalences.

  In the case of (2),
	when $m$ is  even,
 	we have $\alpha >  0$ and 
	$\beta \ge  0$   based on  Lemma  \ref{rem.sign.alpha}.
 	When (\ref{equ.alpha.beta.sign}) is satisfied,   
	there  are   at most three real  roots for
	(\ref{uim1alpha}). 
	Similarly,  all $n$  variables  cannot be the same one, indicating  that there  are  at  least  two  different  values  
 among  the components of $\mathbf  u$.
	Then, the  analysis can be 
	 carried out similarly to that in (1).
	When there are only two distinct roots,
	the structure  of
	the solution will be  given by  (\ref{ab_solu}).
 	To be consistent with the setting $\beta \ge 0$  presented  in Lemma \ref{rem.sign.alpha},
	we  can perform a similar calculation  of   $\beta$  as 
	 was done   for   $\alpha$ in the first claim, and the same constraint on $k$ can be derived. 
	  It can be observed that when $n$ is even, and $k= n/2$, it holds that $a = -b$ and $\beta =0$. 
	 The solutions   are  listed in
	(\ref{u_classifyeven1}).
	While
	for
	(\ref{u_classifyeven2}),
 	we require a strict inequality $\beta <\frac{(m-2)\alpha}{m-1}
	\left( \frac{\alpha}{m-1} \right)^{\frac{1}{m-2}} $   based on    (\ref{equ.alpha.beta.sign}),  to ensure that there are three distinct roots
	for $	u_{i}^{m-1} -\alpha u_{i} - \beta =0$, 
		 denoted $e < d <c $ for simplicity. 
		When $\beta =0$,  based on (\ref{equ.threeroot}), it holds that $ d=0$.
		Then, to satisfy $\beta =0$, it can be derived that $ p =s$ and $c=-e = \sqrt {\frac {1}{2p}}$ as presented above. 
		When  $\beta > 0$, 
		(\ref{uim1alpha}) has  one positive root $c$ and two negative roots $d$ and $e$. In both cases, it holds that $e < d \le 0 < c$.
	Since we have  three  unknown  variables $c,d,e$
	but only  two  constraints, we cannot
	obtain the  explicit  expression  for $c,d,e$ in general.
 	We can only restrict the signs  of  $c,d,e$ as stated in the lemma. 
	$ pc+qd+se= 0$ and $ pc^{2}+qd^{2}+se^{2} = 1$ are presented as two constraints, 
	 which naturally follow from  (\ref{KKTcon1})  and
	(\ref{KKTcon2}). 
	In addition, we need to calculate the  value of  $pc^{m-1}+qd^{m-1}+se^{m-1}$ to ensure that $\beta \ge 0$.
 	Similarly, $\alpha  =  pc^{m}+qd^{m}+se^{m}$ is also calculated, and  
	$\beta <  \frac{(m-2)\alpha}{m-1}
	(\frac{\alpha}{m-1})^{ \frac{1}{m-2}}$ is required. 
Note that 	we cannot provide the explicit constraints on $p,q,s$ in general  as can be done for $k$. 
Only when $ \beta=0$, we can derive a explicit form for $\mathbf u$, as concluded in the lemma. 
The proof is complete. 	     $\blacksquare$
\end{proof}

\begin{wlremark}
	Note that in
	(\ref{u_classifyeven1})
	and
	(\ref{u_classifyeven2}),
	we use five notations  from
	$a$ to $e$ to distinguish
	the  two  different   structures   of  the solutions for the even $m$ cases.
	We  hope  that  such a  distinction
	will not cause ambiguity
 regarding	 the fact that there are at most three different  real roots for
	$  u_{i}^{m-1} -\alpha u_{i} - \beta =0  $  for  $  i=1,2, \dots, n$
	when $m$ is even and (\ref{equ.alpha.beta.sign}) is satisfied.
	Note that there is no   explicit expression for the notations  $c,d,e$ 
	in (\ref{u_classifyeven2})  in general.
\end{wlremark}

Note that 
 the vectors in the  regular simplex  frame  (which are  the generators of the corresponding regular simplex tensor)  $\mathbf{w}_{1}, \ldots, \mathbf{w}_{n}$
 are  $n$ eigenpairs of   $\mathcal S_{\mathbf{W}}$ \cite{RobustEigen}.
After   investigating the structure of all  KKT points,
by equivalence between 	model (\ref{optmodelori})  and 	(\ref{optmodel}),
the  following lemma
can be  established to
 show  how
 the vectors in the  regular simplex  frame  $\mathbf{w}_{1}, \ldots, \mathbf{w}_{n}$
can be related to the set of
all KKT  solutions of 	(\ref{optmodel}):
\begin{wllemma}\label{vectorregularsimplexframe}
 Consider the   setting as stated in   Lemma  \ref{Theorem_structureofall}.
	For the  vectors   in the regular simplex  frame, i.e.,
	$ \mathbf{w}_{1}, \ldots, \mathbf{w}_{n}$,
	the set of 	 $\mathbf u$ obtained  by (\ref{udenote})
	corresponds  to
 	solutions  where    $k=1$   in both
	(\ref{u_classifyodd})  and (\ref{u_classifyeven1}).
\end{wllemma}

\begin{proof}
	Concerning  the   vectors   in the regular simplex  frame
	$
		\mathbf W =
		[
		\mathbf{w}_{1}, \ldots, \mathbf{w}_{n}
		] \in
		\mathbb{R}^{(n-1) \times  n}$,
	it holds that
	$ \mathbf{w}_{i}^{\mathrm T} \mathbf{w}_{i} =1,
		\mathbf{w}_{i}^{\mathrm T} \mathbf{w}_{j} = -\theta = -\frac{1}{n-1} $
	for
	$ i \neq j$.
	It can be calculated that  when
	$ \mathbf v = \mathbf w_{j}$,
	by (\ref{udenote}),
	it holds that
	\begin{align}\label{udenotealpha}
		\mathbf   u
		\nonumber =
		\sqrt {
			\frac  {n-1}{ n}
		}
		{\mathbf W}^{\mathrm T}  {\mathbf w_{j}}
			=
			[u_1,u_2,\dots, u_{n}]^{\mathrm T}
		=
		\sqrt {
			\frac  {n-1}{ n}
		}
		[-\theta,  \dots,   -\theta,
			\underbrace{ 1 }_{j}, -\theta,  \dots,  -\theta ]^{\mathrm T}
		,
	\end{align}
	It can be checked that
	$\mathbf u$ with the above form satisfies the two constraints  in (\ref{optmodel}), and 
 the signs of its corresponding $\alpha$ and $\beta$ are   also consistent with  the  setting stated in  Lemma  \ref{rem.sign.alpha}.
	Clearly,
	when $ \mathbf v = \mathbf w_{j}$,
	$\mathbf u$    obtained  by (\ref{udenote})
	corresponds
	to the
	subset where   $k=1$
	among solutions  of
	(\ref{u_classifyodd})  and (\ref{u_classifyeven1}), and
	it holds for both odd  and even $m$  cases.    $\blacksquare$
\end{proof}

\subsection{
	Local optimality of all KKT   points of  (\ref{optmodel})
}
\label{locallymaximized}

In this section,  we are further interested in identifying  the local optimality of each  KKT point  in Lemma  \ref{Theorem_structureofall}.

For local optimality checking,    it is necessary to   introduce the second-order
derivative information
of  the  Lagrangian function.
For  (\ref{Lagrangianfunctionnew}),   the second-order  derivative with respect to    $ \mathbf u $, termed  the Hessian matrix,  is   denoted  by
\begin{equation}\label{Hessianmatrix}
	\mathbf H (\mathbf u) := \triangledown_{\mathbf u\mathbf u}^{2} L
	=(m-1) \operatorname{diag} (\mathbf u ^{\circledast^{m-2}})
	-\alpha \mathbf  I_{n}  .
\end{equation}
Similar to
the  derivation for (\ref{Mhess}),
the
locally    optimal    point of   (\ref{optmodel})
can  be  identified  by
checking the  negative
semi-definiteness of the  following  matrix:
\begin{equation}\label{Mhess2}
	\mathbf {M} :=
	\mathbf {M} (\mathbf u)
	=
	(\mathbf  P_{\mathbf  A } ^{\bot})^{\mathrm T}    \mathbf H (\mathbf u)  \mathbf  P_{\mathbf  A }^{\bot}
	=
	\mathbf  P_{\mathbf  A } ^{\bot}   \mathbf H (\mathbf u)  \mathbf  P_{\mathbf  A }^{\bot}
	,
\end{equation}
where
$ \mathbf A $
is determined by
the
two constraints
$  g_{1}(\mathbf u) $ and  $ g_{2}(\mathbf u) $,
which
 has  the  following form:
\begin{equation}
	\label{Aform}
	\mathbf A :=  [\triangledown  g_{1}(\mathbf u) , \triangledown  g_{2}(\mathbf u) ] =
	[\mathbf u, \mathbf 1_{n}]  \in   \mathbb R^{n  \times 2 } .
\end{equation}
 We need to 
further count the number of zero  eigenvalues  of 	$\mathbf  M$  
 to guarantee locally maximized solutions. 
When there are  two zero eigenvalues,  	 it is guaranteed that the corresponding solution is locally  maximized.
See   the Appendix  for more details.
This criterion will be  checked and discussed in Lemma \ref{Theorem_structureoflocal}
and \ref{Theorem_structureoflocaleven1}.

Due to the form of $\mathbf  A$ as shown in (\ref{Aform}),
$ \mathbf  P_{\mathbf  A }^{\bot} $ can be   further  rewritten as
\begin{align}\label{projecteddenote}
	\mathbf  P_{\mathbf  A } ^{\bot}
	=
	 & \nonumber
	\mathbf  I_{n}  -
	\begin{bmatrix}
		\mathbf u &
		\mathbf 1_{n}
	\end{bmatrix}
	(
	\begin{bmatrix}
		\mathbf u^{\mathrm T} \\
		\mathbf 1_{n}^{\mathrm T}
	\end{bmatrix}
	\begin{bmatrix}
		\mathbf u &
		\mathbf 1_{n}
	\end{bmatrix}
	)
	^{-1}
	\begin{bmatrix}
		\mathbf u^{\mathrm T} \\
		\mathbf 1_{n}^{\mathrm T}
	\end{bmatrix}
	\\ \nonumber
	 & =
	\mathbf  I_{n}  -
	\begin{bmatrix}
		\mathbf u &
		\mathbf 1_{n}
	\end{bmatrix}
	(
	\begin{bmatrix}
		\mathbf u^{\mathrm T}\mathbf u       & \mathbf u^{\mathrm T}\mathbf 1_{n}     \\
		\mathbf 1_{n} ^{\mathrm T} \mathbf u & \mathbf 1_{n}^{\mathrm T}\mathbf 1_{n} \\
	\end{bmatrix}
	)
	^{-1}
	\begin{bmatrix}
		\mathbf u^{\mathrm T} \\
		\mathbf 1_{n}^{\mathrm T}
	\end{bmatrix}
	\\
	 & =
	\mathbf  I_{n}  -
	\begin{bmatrix}
		\mathbf u &
		\mathbf 1_{n}
	\end{bmatrix}
	(
	\begin{bmatrix}
		1 & 0 \\
		0 & n \\
	\end{bmatrix}
	)
	^{-1}
	\begin{bmatrix}
		\mathbf u^{\mathrm T} \\
		\mathbf 1_{n}^{\mathrm T}
	\end{bmatrix}
	=
	\mathbf  I_{n}  -
	(\mathbf u\mathbf u^{\mathrm T}
	+ \frac  1n \mathbf 1_{n}\mathbf 1_{n}^{\mathrm T} ).
\end{align}
where we use the equations
$ \mathbf u^{\mathrm T}\mathbf 1_{n} = \mathbf 1_{n} ^{\mathrm T} \mathbf u =  0 $,
$\mathbf 1_{n} ^{\mathrm T} \mathbf 1_{n} =  n $   and
$ \mathbf u^{\mathrm T}\mathbf u=1$,
all of which are due to the KKT conditions
presented in (\ref{KKTgradient}),  (\ref{KKTcon1})  and
(\ref{KKTcon2}).

We have discussed the coincidence  between  the KKT points of    (\ref{optmodelori})  and  those of   (\ref{optmodel}) in Lemma \ref{lemma.coincidence.KKT}.
Here,
when    taking second-order sufficient condition into account,
we need to further check 
 whether the correspondence bwtween    the local optimality of  the KKT points of  (\ref{optmodelori})  and 
that of   (\ref{optmodel})   can be guaranteed.
In other words,  we ask whether,  if
$ \mathbf v$ is a locally maximized  solution of  (\ref{optmodelori}),
	the corresponding $\mathbf u$   obtained  by
(\ref{udenote})
 is also  a locally maximized one  of (\ref{optmodel}), and 
 vice versa.
Similar claims and questions   can be made for the locally minimized and saddle cases.
We first  provide a positive answer   to this question.

Regarding  the 	matrix  $\mathbf K$  defined  in  (\ref{Mhess}),
when the  tensor  is   the  regular simplex  tensor   given  by  (\ref{simplextensor}),
for distinction, based on (\ref{equ.K.sim}),
it is denoted  by a   new matrix with the following form:
\begin{align}\label{Mhessw}
	\mathbf  K_{\mathbf W}
	= \mathbf  P_{\mathbf  v}^{\bot}
	[
	(m-1)\mathcal S_{\mathbf W} \mathbf v^{m-2} - \lambda \mathbf I_{n-1})
	]
	\mathbf  P_{\mathbf  v}^{\bot}
	=
	(m-1)\mathcal S_{\mathbf W}  \mathbf v^{m-2}
	-
	\lambda \mathbf I_{n-1}
	-
	(m-2)\lambda \mathbf v \mathbf v^{\mathrm T} .
\end{align}

Then, we have the following lemma:
\begin{wllemma}\label{lemma.wmw.k}
 	Let $\mathcal{S}_{\mathbf{W}}\in T^{m}(\mathbb{R}^{n-1})$ be a regular simplex tensor with $n\geq 3$, $m\geq 3$. 	
For models  (\ref{optmodelori})
	and
	(\ref{optmodel}) related to the regular simplex tensor,
	concerning the matrix $\mathbf M$  defined in (\ref{Mhess2})
	and    the 	matrix  $\mathbf K_{\mathbf W}$  in  (\ref{Mhessw}),
	recall the regular simplex frame matrix $\mathbf W$,
	it holds that
	(1):
	$	\mathbf W  \mathbf  M    	\mathbf W^{\mathrm T}  = c_{n,m}	\mathbf K_{\mathbf W}$, where $c_{n,m} > 0$ is a positive 
	 value 
	which varies with $n,m$;
	(2):
	$  	  \mathbf  M     \preceq 0 \Leftrightarrow 	\mathbf W  \mathbf  M    	\mathbf W^{\mathrm T}  \preceq 0 \Leftrightarrow   \mathbf K_{\mathbf W} \preceq 0$.
\end{wllemma}

The detailed proof is deferred to  Section  \ref{proof.lemma.wmw.k}.
The second claim means that the signs of all eigenvalues of  $    \mathbf  M    $, $  	\mathbf W  \mathbf  M    	\mathbf W^{\mathrm T}  $
and  $ \mathbf K_{\mathbf W} $
are the same,
and so  are their positive or negative definiteness properties.
With this conclusion, we can safely focus on the local optimality of  the KKT points of    model 	(\ref{optmodel}),
and  check 
 the  signs of the  eigenvalues of  the matrix  $\mathbf  M $.

Following the conclusions  presented in 	Lemma  \ref{Theorem_structureofall},		similarly,
we will separately discuss the cases of odd and even $m$.
 
	We emphasize again that    the special combination $(m,n) = (4,3)$ is excluded in the following discussion, since every 
	 feasible $\mathbf u$  that satisfies  
	the constraints and the sign settings of $\alpha$ and $\beta$
 is both locally maximized and minimized,
 and there is no need for the local optimality identification.
Then, for the remaining combinations where $n\geq 3$, $m\geq 3$  and  $(m,n) \neq  (4,3)$, as stated in  Lemma \ref{Theorem_structureofall}, 
 there are three structures   as presented
from
(\ref{u_classifyodd})
to
(\ref{u_classifyeven2})
and  their proof strategies  also differ from each other.
We thus conclude their theoretical results
 on the local optimality issue
 as the following three lemmas, and the detailed proof will be provided in  Sections
  \ref{section.prrof.lemma} and \ref{proof.lemma.Theorem_structureoflocaleven2}.

\begin{wllemma}\label{Theorem_structureoflocal}
	For the case where   $m$ is odd  stated  in Lemma \ref{Theorem_structureofall},  concerning  the solutions
	with the form of
	(\ref{u_classifyodd})
	which satisfy $ \mathbf u^{\mathrm T}\mathbf 1_{n}=0$ and $ \mathbf u^{\mathrm T}\mathbf u=1$,
	(1):
	when  $k=1$,
the solution is  a   strict local maximum of    model (\ref{optmodel});
	(2): when  $2 \le  k \le  \lfloor n/2  \rfloor $, it is  a   saddle point  of  model (\ref{optmodel}).
\end{wllemma}

\begin{wllemma}\label{Theorem_structureoflocaleven1}
	For the case where   $m$ is even  stated  in Lemma \ref{Theorem_structureofall},  concerning  the solutions
	with the form of  (\ref{u_classifyeven1})
	which satisfy $ \mathbf u^{\mathrm T}\mathbf 1_{n}=0$ and $ \mathbf u^{\mathrm T}\mathbf u=1$,
	denote
	\begin{align}
		l(m,n,k)
		=
		(m-1)nk^{m-2} - (n-k)^{m-1}- k^{m-1}.
	\end{align}

	(1):
	When  $k=1$,     it holds that  $	l(m,n,k)  <0 $, and
the solution is  a   strict local maximum of    model (\ref{optmodel});
	(2):  	when  $2 \le  k \le  \lfloor n/2  \rfloor $,
if
		$	l(m,n,k)  >0 $, 
		the  corresponding  solution  is a     strict local minimum  of    model (\ref{optmodel});
		if $	l(m,n,k)  <0 $, it is a saddle point  of    model (\ref{optmodel}); 
		if $	l(m,n,k)  = 0 $, it is degenerate but is not a local maximum.
\end{wllemma}

\begin{wllemma}\label{Theorem_structureoflocaleven2}
	For the case where   $m$ is even  stated  in Lemma \ref{Theorem_structureofall},  concerning  the solutions
	with the form of  (\ref{u_classifyeven2}) 	 which satisfy $ \mathbf u^{\mathrm T}\mathbf 1_{n}=0$ and $ \mathbf u^{\mathrm T}\mathbf u=1$,
 none of them can be 
 a  local  maximum   of    model (\ref{optmodel}).
\end{wllemma}

\subsection{Proof  for   Theorem \ref{Theorem.local_eigenpair}  }
\label{proofmain}
We 
  can then  finish the proof for Theorem \ref{Theorem.local_eigenpair} as  follows.

\textbf{[Proof for the main Theorem \ref{Theorem.local_eigenpair}]}
  	We first clarify the special combination $(m,n) = (4,3)$ where
the  	objective  value 
	$
	\mathcal S_{\mathbf W } \mathbf v^{m}
	$   is a constant as  commented in Remark  \ref{RemarkScope}.
	In this case, 
	any unit-length vector is  an eigenpair of $\mathcal S_{\mathbf W}$. 
	Therefore, all of them are 
	locally maximized and minimized.

We turn to the  other   combinations with  $(m,n) \neq (4,3)$ and $n \ge 3,  m \ge 3$.
First,   combining 	Lemma  \ref{lemma.coincidence.KKT} with  	Lemma  \ref{lemma.wmw.k}
ensures the equivalence between  (\ref{optmodelori})
and
(\ref{optmodel}).
Concerning the reformulated model    (\ref{optmodel}),
Lemma
\ref{Theorem_structureoflocal},  \ref{Theorem_structureoflocaleven1}, 
and  \ref{Theorem_structureoflocaleven2}
discuss the local optimality of all KKT points, up to sign and permutation equivalences,   as   presented in Lemma \ref{Theorem_structureofall},
in both odd and even $m$ cases.
Combining their conclusions yields that   the  only  locally maximized points of   	 (\ref{optmodel})
are the solutions where $k=1$,
which equivalently  are 
 the vectors in the  regular simplex  frame,   according to   Lemma  \ref{vectorregularsimplexframe}.
 Combining the above discussion  for both the case $(m,n) = (4,3)$ and the remaining cases with  $n \ge 3, m\ge 3$, 
the proof of  Theorem \ref{Theorem.local_eigenpair} is complete.  $\blacksquare$

\begin{wlremark}\label{remark.robust.formula}
	In this remark,
	we will    address   the question  raised  in Remark \ref{remark.comparsion.robust}.
	 Specifically,
	in this work, we
	perform two
	main transformations   to deal with the robustness issue.
	One is to resort to analyzing local
	optimality of eigenpairs. 	We have explained the connection between local
	optimality and robustness;
	the other is to  derive  an   equivalent model
	(\ref{optmodel}) to further  check local
	optimality.
 One natural question may arise here. Since model (\ref{optmodel}) has been obtained, we could similarly define its corresponding tensor power method and Jacobian matrix. So why not do this directly, rather than focusing on the indirect local optimality problem?
	The reason is explained as follows.
	Indeed, we can define tensor power method for (\ref{optmodel}),
	which   has  the form of
	\begin{equation}
		\phi (\mathbf{u}):=
		\frac
		{\mathbf u ^{\circledast^{m-1}}	-\beta\mathbf 1_n  	}
		{	\left\|  \mathbf u ^{\circledast^{m-1}}	-\beta\mathbf 1_n   \right\|},
	\end{equation}
	where $\beta$ is calculated by (\ref{alphares}).
	Its  corresponding Jacobian matrix can be derived (we omit the detailed derivation
	and  refer to Lemma 3.3 in \cite{RobustEigen}  for a complete process):
	\begin{equation}\label{Jacobianmatirxforu}
		\mathbf{J}(\mathbf{u}):=\frac{m-1}{\alpha}
		\left(
		\operatorname{diag} (\mathbf u ^{\circledast^{m-2}})
		-   \mathbf{u} \mathbf{u}^{\mathrm T}\operatorname{diag} (\mathbf u ^{\circledast^{m-2}})  \right),
	\end{equation}
	which is an analogous form  to (\ref{Jacobianmatirx}), and $\alpha \neq 0$.
	It can be seen that  checking the spectral radius for
	(\ref{Jacobianmatirxforu})  at  a   solution
	$\mathbf u$  is also difficult.
	The result of $ \mathbf{u} \mathbf{u}^{\mathrm T}$ is listed in (\ref{publock}),
 and thus
	the calculation for $ \mathbf{u} \mathbf{u}^{\mathrm T}\operatorname{diag} (\mathbf u ^{\circledast^{m-2}})$
	will  not  have    the block  structure.
	In contrast, in local  optimality  checking, we can obtain
	a better block structure for $\mathbf M$  in (\ref{Mhess2})  by further  utilizing  $ \mathbf u^{\mathrm T}\mathbf 1_{n} = \mathbf 1_{n} ^{\mathrm T} \mathbf u =  0 $, as can  be seen  from
	(\ref{mblock}).

	Table \ref{table.comparison.local.robust}
	intuitively compares   the checking criteria for
	robust and  local optimality  issues.
	The above illustration
	implies
	that among  four  formulas, namely 
	(\ref{Jacobianmatirx}),     (\ref{Jacobianmatirxforu}), (\ref{Mhess})
	and   (\ref{Mhess2}),
	checking local optimality by (\ref{Mhess2})  will be more convenient and
	effective than the other options.
	This justifies the necessity of our two main transformations, as
	mentioned in this remark.
	In addition,
	local optimality
	checking  provides a more precise way to   group the eigenpairs  into
	three classes
	and provides   an  optimization landscape for the nonconvex and NP-hard model.
	Compared to
	binary classification for   the   robustness
	issue, as can be seen from Fig. \ref{flow.robust.local},
	we can  simultaneously answer  both  the   local optimality and  robustness   issues
	when adopting the above strategy.
\end{wlremark}

\begin{table*}[t]
	\normalsize
	\centering
	\caption{
	An intuitive comparison  between the
	 local optimality and robustness  criteria
	  on two equivalent models (\ref{optmodelori}) and (\ref{optmodel}) related to the  regular simplex tensor eigenpairs. 
	}
	\label{cc1}
	\renewcommand\arraystretch{1.2}
	\begin{tabular}{  |c | c | c |   }
		\hline
		Optimization model
		 &
		$
			\begin{cases}
				\max\limits_{\mathbf v} \quad \mathcal S \mathbf v^{m}
				=
				(\sum\limits_{i=1}^{n} \mathbf{w}_{i}^{\circ m}) \mathbf v^{m} \\
				\text{s.t.}  \quad \mathbf v^{\mathrm {T}}\mathbf v=1
			\end{cases}
		$
		 &
		$
			\begin{cases}
				\max\limits_{\mathbf u}  \quad  \sum\limits_{i=1}^{n}  u_{i}^{m}
				\\
				\text { s.t. }   \quad
				\mathbf   u^{\mathrm T}   \mathbf   u =1 ,
				\quad
				\mathbf   u^{\mathrm T}   \mathbf   1_{n} =0
			\end{cases}
		$                                                                                                                              \\
		\hline
		Robustness criterion
		 & (\ref{Jacobianmatirx})
		 & (\ref{Jacobianmatirxforu})
		\\
		\hline
		Local optimality  criterion
		 & $\mathbf  P_{\mathbf  v } ^{\bot}   \mathbf H (\mathbf v)  \mathbf  P_{\mathbf  v }^{\bot}$  in (\ref{Mhess})
		 & $\mathbf  M = \mathbf  P_{\mathbf  A } ^{\bot}   \mathbf H (\mathbf u)  \mathbf  P_{\mathbf  A }^{\bot}$  in (\ref{Mhess2})
		\\
		\hline
	\end{tabular}
	\label{table.comparison.local.robust}
\end{table*}

\subsection{Proof of the  developed lemmas}
\label{section.prrof.lemma}

In this part, we will provide the detailed proof for   	 Lemmas	  \ref{lemma.wmw.k},
\ref{Theorem_structureoflocal},   and  \ref{Theorem_structureoflocaleven1}
presented  above.

\subsubsection{Proof of Lemma \ref{lemma.wmw.k}}
\label{proof.lemma.wmw.k}
\begin{proof}
We first prove  Claim  (1).
	We begin our proof  by   simplifying the form of $\mathbf W  \mathbf  M    	\mathbf W^{\mathrm T}$.
	By $	\mathbf{W}  \mathbf{1}_{n} =
		\mathbf{0}_{n-1}
	$ stated in Property \ref{regularproperty},
	it holds that
	\begin{align}\label{projecteddenotew}
		\mathbf W  \mathbf  P_{\mathbf  A } ^{\bot}
		=
		\mathbf W   (	\mathbf  I_{n}  -
		\mathbf u\mathbf u^{\mathrm T}
		- \frac  1n \mathbf 1_{n}\mathbf 1_{n}^{\mathrm T}  )
		=
		\mathbf W  -  \mathbf W	\mathbf u\mathbf u^{\mathrm T}  .
	\end{align}
	Then,  $\mathbf W  \mathbf  M    	\mathbf W^{\mathrm T}$
	is further rewritten as
	\begin{align}\label{equ.wmw}
		\mathbf W  \mathbf  M    	\mathbf W^{\mathrm T}
		 & =
		\mathbf W  \mathbf  P_{\mathbf  A } ^{\bot}   \mathbf H (\mathbf u)   \mathbf  P_{\mathbf  A } ^{\bot} \mathbf W^{\mathrm T}
		=
		( \mathbf W  -  \mathbf W	\mathbf u\mathbf u^{\mathrm T} )
		\mathbf H (\mathbf u)
		( \mathbf W  -  \mathbf W	\mathbf u\mathbf u^{\mathrm T} )^{\mathrm T}
		\nonumber \\
		 & =
		\mathbf W  \mathbf H (\mathbf u)    \mathbf W^{\mathrm T}
		-
		\mathbf W	\mathbf u\mathbf u^{\mathrm T}    \mathbf H (\mathbf u)  \mathbf W^{\mathrm T}
		-
		\mathbf W    \mathbf H (\mathbf u)	\mathbf u\mathbf u^{\mathrm T}   \mathbf W^{\mathrm T}
		+
		\mathbf W	\mathbf u\mathbf u^{\mathrm T}    \mathbf H (\mathbf u)  \mathbf u	\mathbf u^{\mathrm T}   \mathbf W^{\mathrm T}.
	\end{align}

	We separately analyze the above four terms in the following.
	The first term is  dealt with first.
	Recall that our final aim is to build a connection with $\mathbf K_{\mathbf W}$  in  (\ref{Mhessw}).
	 First, it can be derived that
	\begin{align}\label{swvm2}
		\mathcal{S}_{\mathbf W}  \mathbf v^{m-2}
		 & =
		(
		\sum_{i=1}^{n} \mathbf{w}_{i}^{\circ m}
		)  \mathbf v^{m-2}
		=
		\sum_{i=1}^{n} (  \mathbf w_{i}^{\mathrm T} \mathbf  v)^{ m-2}
		\mathbf{w}_{i}
		\mathbf{w}_{i}^{\mathrm T}
		\nonumber \\
		 & =
		(
		\sqrt {
			\frac   { n}{n-1}
		} )^{m-2}
		\sum_{i=1}^{n}  u_i ^{ m-2}
		\mathbf{w}_{i}
		\mathbf{w}_{i}^{\mathrm T}
		=
		(
		\sqrt {
			\frac   { n}{n-1}
		} )^{m-2}
		\mathbf  W
		\operatorname{diag} (\mathbf u ^{\circledast^{m-2}})
		\mathbf W^{\mathrm T},
	\end{align}
	where in the third equation,
	we use that
	$u_i :=
		\sqrt {
			\frac  {n-1}{ n}
		}
		\mathbf  v^{\mathrm T}  \mathbf w_i
	$ by 	(\ref{udenote}).
	Using this and the fact that  	 $     {\mathbf W} {\mathbf W}^{\mathrm T} =\frac  {  n  }{n-1}  \mathbf I_{n-1}$,
	it holds that
	\begin{align}\label{wHessianmatrixwt}
		\mathbf  W
		\mathbf H (\mathbf u)
		\mathbf W^{\mathrm T}
		 & =(m-1) \mathbf  W \operatorname{diag} (\mathbf u ^{\circledast^{m-2}})  \mathbf W^{\mathrm T}
		-\alpha \mathbf  W  \mathbf W^{\mathrm T}
		\nonumber                                                                                        \\
		 & =
		(
		\sqrt {
			\frac  {n-1}  { n}
		} )^{m-2}
		(m-1)
		\mathcal{S}_{\mathbf W}  \mathbf v^{m-2}
		-
		(
		\sqrt {
			\frac   {n-1} { n}
		} )^{m}
		\lambda
		\cdot
		\frac   { n}{n-1}
		\mathbf I_{n-1}
		\nonumber                                                                                        \\
		 & =
		(
		\sqrt {
			\frac  {n-1}  { n}
		} )^{m-2}
			[
				(m-1)
				\mathcal{S}_{\mathbf W}  \mathbf v^{m-2}
				-
				\lambda
				\mathbf I_{n-1}
			]	,
	\end{align}
	where in the second equation,  we use
	that
	$ \lambda = \mathcal{S}  \mathbf v^{m}
		=
		(
		\sqrt {
			\frac  { n}{n-1}
		} )^{m} \alpha$
	by  (\ref{eq.obj.pro})  and (\ref{alphares}).

	We deal with the second and third  terms in (\ref{equ.wmw}),  which are equivalent.
	By definition, it holds
	that
	\begin{align}\label{HU}
		\mathbf H (\mathbf u)
		\mathbf u
		 & =
		((m-1) \operatorname{diag} (\mathbf u ^{\circledast^{m-2}})
		-\alpha \mathbf  I_{n}) \mathbf u
		=
		(m-1)  \mathbf u ^{\circledast^{m-1}}-\alpha \mathbf u
		\nonumber \\
		 & =
		(m-1) \alpha  \mathbf u + (m-1) \beta \mathbf  1_{n}-\alpha \mathbf u
		=
		(m-2) \alpha  \mathbf u + (m-1) \beta \mathbf  1_{n},
	\end{align}
	where  (\ref{KKTgradient}) is utilized in the third  equation.
	By $	\mathbf{W}  \mathbf{1}_{n} =
		\mathbf{0}_{n-1}
	$ stated in Property \ref{regularproperty} again
	and
	$  \mathbf   u :=
		\sqrt {
			\frac  {n-1}{ n}
		}
		{\mathbf W}^{\mathrm T}  {\mathbf v} $  in  	(\ref{udenote}),
	it holds that
	\begin{align}\label{wHessianmatrixwt2}
		\mathbf W	\mathbf u\mathbf u^{\mathrm T}    \mathbf H (\mathbf u)  \mathbf W^{\mathrm T}
		 & =
		\mathbf W    \mathbf H (\mathbf u)	\mathbf u\mathbf u^{\mathrm T}   \mathbf W^{\mathrm T}
		=
		\mathbf W [ (m-2) \alpha  \mathbf u + (m-1) \beta \mathbf  1_{n}  ]  \mathbf u^{\mathrm T}   \mathbf W^{\mathrm T}
		=
		(m-2) \alpha  \mathbf W \mathbf u  \mathbf u^{\mathrm T}   \mathbf W^{\mathrm T}
		\nonumber \\
		 & =
		(m-2)(
		\sqrt {
			\frac   {n-1} { n}
		} )^{m}
		\lambda
		\cdot
		\frac   {n} { n-1}    \mathbf v  \mathbf v^{\mathrm T}
		\nonumber \\
		 & =
		(
		\sqrt {
			\frac  {n-1}  { n}
		} )^{m-2}
		(m-2)
		\lambda
		\mathbf v  \mathbf v^{\mathrm T} ,
	\end{align}

	The last term in  (\ref{equ.wmw})
	can be analyzed in a similar way, and it holds that
	\begin{align}\label{wHessianmatrixwt3}
		\mathbf W	\mathbf u\mathbf u^{\mathrm T}    \mathbf H (\mathbf u)  \mathbf u 	\mathbf u^{\mathrm T} \mathbf W^{\mathrm T}
		 & =
		\mathbf W	\mathbf u  \mathbf u^{\mathrm T}   [ (m-2) \alpha  \mathbf u + (m-1) \beta \mathbf  1_{n}  ]  \mathbf u \mathbf W^{\mathrm T}
		=
		(m-2) \alpha  \mathbf W \mathbf u  \mathbf u^{\mathrm T}   \mathbf W^{\mathrm T}
		\nonumber \\
		 & =
		(
		\sqrt {
			\frac  {n-1}  { n}
		} )^{m-2}
		(m-2)
		\lambda
		\mathbf v  \mathbf v^{\mathrm T} .
	\end{align}

	Substituting  (\ref{wHessianmatrixwt}),  (\ref{wHessianmatrixwt2}) and  (\ref{wHessianmatrixwt3})
	into (\ref{equ.wmw}), with some simplification,
	 yields   that
	\begin{align}\label{equ.wmw.2}
		\mathbf W  \mathbf  M    	\mathbf W^{\mathrm T}
		=
		(
		\sqrt {
			\frac  {n-1}  { n}
		} )^{m-2}
		\mathbf  {K }_{\mathbf W},
	\end{align}
	and
	  the positive coefficient  denoted  $c_{n,m}$  is thus determined. 
	The first claim is proved.

	(2): To prove  Claim  (2), we need to analyze  
	 the signs of  the eigenvalues  of two matrices.
	The  conclusion  in  Lemma
	\ref{AB_BA_eig}
	implies that the  non-zero  eigenvalues  of  $\mathbf W  \mathbf  M    	\mathbf W^{\mathrm T}$
	are the same as those  of  a new matrix  $ \mathbf  M    	\mathbf W^{\mathrm T}\mathbf W $.	 
	$	\mathbf W^{\mathrm T}\mathbf W$ can be decomposed as
	$ 	\mathbf W^{\mathrm T}\mathbf W = -\frac{1}{n-1}  \mathbf 1_{n}\mathbf 1_{n}^{\mathrm T}   +
		\frac{n}{n-1}    \mathbf I_{n}$.
	 Then,
	$ \mathbf  M    	\mathbf W^{\mathrm T}\mathbf W $
	can be simplified as
	\begin{align}\label{equ.mwwt}
		\mathbf  M    	\mathbf W^{\mathrm T}\mathbf W
		 & \nonumber  =
		\mathbf  P_{\mathbf  A } ^{\bot}   \mathbf H (\mathbf u)  \mathbf  P_{\mathbf  A }^{\bot}     ( -\frac{1}{n-1}  \mathbf 1_{n}\mathbf 1_{n}^{\mathrm T}   +
		\frac{n}{n-1}   \mathbf I_{n})
		=
		\mathbf  P_{\mathbf  A } ^{\bot}   \mathbf H (\mathbf u)  \mathbf  P_{\mathbf  A }^{\bot}     ( -\frac{1}{n-1}  \mathbf 1_{n}\mathbf 1_{n}^{\mathrm T}   +
		\frac{n}{n-1}\mathbf I_{n})
		\\
		 & =
		-\frac{1}{n-1}   \mathbf  P_{\mathbf  A } ^{\bot}   \mathbf H (\mathbf u)  \mathbf  P_{\mathbf  A }^{\bot}      \mathbf 1_{n}\mathbf 1_{n}^{\mathrm T}   +
		\frac{n}{n-1} \mathbf  P_{\mathbf  A } ^{\bot}   \mathbf H (\mathbf u)  \mathbf  P_{\mathbf  A }^{\bot}
		=
		\frac{n}{n-1}  \mathbf  M  ,
	\end{align}
	where in the last equation, we use  the fact that
	$  \mathbf  P_{\mathbf  A }^{\bot}  \mathbf 1_{n}\mathbf 1_{n}^{\mathrm T}   = \mathbf O$.
Combining (\ref{equ.wmw.2}) with (\ref{equ.mwwt}) implies that 
 the signs of the nonzero eigenvalues
 of $\mathbf{M}$ and $\mathbf{K}_{\mathbf{W}}$ are consistent, and their positive 
  semi-definiteness properties coincide.
	The proof is complete.      $\blacksquare$
\end{proof}

\subsubsection{Proof of  Lemma   \ref{Theorem_structureoflocal}}\label{proof.lemma.Theorem_structureoflocal}

\begin{proof}
	By  considering  the  structure of   $\mathbf u$ as  shown in  (\ref{u_classifyodd})   in Lemma \ref{Theorem_structureofall},
	$ \mathbf u\mathbf u^{\mathrm T} $  can  be  presented in
	  a  $2 \times 2$  block  form,  which  follows:
	\begin{align}\label{publock}
		\mathbf u\mathbf u^{\mathrm T}
		=
		\begin{bmatrix}
			a^{2} \mathbf J_{k}            & ab  \mathbf J_{k \times (n-k)} \\
			ab  \mathbf J_{(n-k) \times k} & b^{2} \mathbf J_{(n-k)}
		\end{bmatrix}
		=
		\begin{bmatrix}
			a^{2} \mathbf J_{k}                   & -\frac 1n  \mathbf J_{k \times (n-k)} \\
			-\frac 1n  \mathbf J_{(n-k) \times k} & b^{2} \mathbf J_{(n-k)}
		\end{bmatrix},
	\end{align}
	where,   based on (\ref{ab_solu}), it holds that
	$ab=   - \frac  1n $
	.
	Similarly,
	$\frac  1n \mathbf 1_{n}\mathbf 1_{n}^{\mathrm T} $   can also be expressed as
	\begin{align}
		\label{p1nblock}
		\frac  1n \mathbf 1_{n}\mathbf 1_{n}^{\mathrm T}
		=
		\begin{bmatrix}
			\frac 1n \mathbf J_{k}               & \frac 1n  \mathbf J_{k \times (n-k)} \\
			\frac 1n  \mathbf J_{(n-k) \times k} & \frac 1n \mathbf J_{(n-k)}
		\end{bmatrix}.
	\end{align}

	Summing (\ref{publock}) and (\ref{p1nblock})  yields
	\begin{align}
		\label{u1nblocksum}
		\mathbf u\mathbf u^{\mathrm T} + \frac  1n \mathbf 1_{n}\mathbf 1_{n}^{\mathrm T}
		=
		\begin{bmatrix}
			(a^{2} + \frac 1n) \mathbf J_{k} & \mathbf O_{k \times (n-k)}          \\
			\mathbf O_{ (n-k) \times k }     & (b^{2} + \frac 1n)\mathbf J_{(n-k)}
		\end{bmatrix}
		=
		\begin{bmatrix}
			\frac 1k \mathbf J_{k}       & \mathbf O_{k \times (n-k)}     \\
			\mathbf O_{ (n-k) \times k } & \frac{1}{n-k}\mathbf J_{(n-k)}
		\end{bmatrix},
	\end{align}
	where  for the second  equation,   we     use 
	\begin{align}\label{equ.a2b}
		a^{2} + \frac 1n =
		(\sqrt{\frac{n-k}{k n}} )^{2} +  \frac 1n
		=
		\frac{n-k}{k n} + \frac 1n
		=
		\frac 1k,
		b^{2} + \frac 1n
		=
		\frac{1}{n-k}.
	\end{align}

	Substituting  (\ref{u1nblocksum})  into (\ref{projecteddenote})
  yields  that
	\begin{align}\label{pup1n}
		\mathbf  P_{\mathbf  A }^{\bot}
		 & =
		\begin{bmatrix}
			\mathbf I_{k} -\frac 1k \mathbf J_{k} & \mathbf O_{k \times (n-k)}                           \\
			\mathbf O_{(n-k) \times k }           & \mathbf I_{(n-k)} - \frac {1}{n-k} \mathbf J_{(n-k)}
		\end{bmatrix},
	\end{align}
	which is a block diagonal matrix.

	In a similar way,
	by  considering  the  structure of   $\mathbf u$ as  shown in  (\ref{u_classifyodd})   in Lemma \ref{Theorem_structureofall},
	$\mathbf  H(\mathbf u)$  can   also be  presented in  the  block  form,  which  follows:
	\begin{align}\label{Hblock}
		\mathbf  H(\mathbf u)
		=
		\begin{bmatrix}
			\sigma_{a} 	\mathbf I_{k}    & \mathbf O_{k \times (n-k)}   \\
			\mathbf O_{(n-k) \times k } & \sigma_{b} \mathbf I_{(n-k)}
		\end{bmatrix},
	\end{align}
	where, 
	based on (\ref{Hessianmatrix}),
	we denote
	$
		\sigma_{a} =  (m-1) a^{m-2}-\alpha,
		\sigma_{b}  =(m-1) b^{m-2}-\alpha $
	for  simplicity.
	Combining  (\ref{Mhess2}), (\ref{pup1n}) and  (\ref{Hblock})  yields
	\begin{align}\label{mblock}
		\mathbf {M}
		=
		\begin{bmatrix}
			\sigma_{a} 	(\mathbf I_{k} -\frac 1k \mathbf J_{k}) & \mathbf O_{k \times (n-k)}                                       \\
			\mathbf O_{(n-k) \times k }                        & \sigma_{b}	(\mathbf I_{(n-k)} - \frac {1}{n-k} \mathbf J_{(n-k)})
		\end{bmatrix} .
	\end{align}

	It can  be  verified that  $  \mathbf I_{k} -\frac 1k \mathbf J_{k}
		= \mathbf I_{k} -\frac 1k \mathbf 1_{k}  \mathbf 1_{k}^{\mathrm T} $  is  an   idempotent matrix, and  its $k$ eigenvalues are    0  and  1 ( with multiplicity $k-1$).
  A  similar   result holds  for  $ \mathbf I_{(n-k)} - \frac {1}{n-k} \mathbf J_{(n-k)}$.
	Clearly,  $\mathbf M $  is  the  direct  sum  of  two   matrices,
	and based on the  conclusion in  Lemma  \ref{direct_eig},
	$n$  eigenvalues  of  $\mathbf M$   are    
	\begin{equation}\label{eigensign}
		\underbrace{
			\sigma_{a}  ,  \sigma_{a},  \dots,   \sigma_{a} }_{k-1},
		\underbrace{
			\sigma_{b}  ,  \sigma_{b},  \dots,   \sigma_{b}  }_{n-k-1},
		\underbrace{
			0, 0}_{2}.
	\end{equation}

	Due to the  fact  that
	each element of  $ \mathbf H _{ii} $
	is  actually
	the gradient of
	$ u_{i}^{m-1} -\alpha u_{i} - \beta $,
	$ \sigma_{a}$ and $ \sigma_{b}$
	actually  reflect
  the increase or decrease behavior
	  at the  corresponding  root.
	When $m$ is odd,  as    can be  observed  from  Fig. \ref{odd},
	it holds that
	\begin{align}\label{eq.Heigenclss}
		\mathbf H _{ii}
		=
		[ (m-1) \operatorname{diag} (\mathbf u ^{\circledast^{m-2}})
		-\alpha \mathbf  I_{n}  ] _{ii}
		=
		\begin{cases}
			(m-1) a^{m-2}-\alpha = \sigma_{a} >0  ,  \quad  \quad    u_{i}=a  >0   \\
			(m-1) b^{m-2}-\alpha = \sigma_{b} <0  ,  \quad  \quad    u_{i}=b    <0 \\
		\end{cases}  .
	\end{align}

	(1): If and only if  $k=1$,  the  $n-2$  non-zero  eigenvalues  of  $\mathbf M$  are all     $\sigma_{b} < 0$, and thus  $\mathbf M$   is  negative  semi-definite.
	 In addition, there are exactly two zero eigenvalues,  indicating  that  the corresponding  direction  $\mathbf u $  is  	
	   a strict local maximum
		 based on the presentation 
	after   Theorem \ref{second_order_necessary} in  the 
  Appendix, which proves  
 the first claim.

	(2):
	When  $2 \le  k \le  \lfloor n/2  \rfloor $,  the  $n$  eigenvalues  of  $\mathbf M$
	  contain
	$k-1$   positive  eigenvalues     $\sigma_{a}$
	and
	$n- k-1$  negative    eigenvalues    $\sigma_{b}$,  
	indicating  that  the corresponding    $\mathbf u $  is   a   saddle  one.     $\blacksquare$
\end{proof}

\subsubsection{Proof of  Lemma   \ref{Theorem_structureoflocaleven1}}\label{proof.lemma.Theorem_structureoflocaleven1}
In this section, we first
focus on the tensor with  even   $m$,
whose solutions
are shown in
(\ref{u_classifyeven1}).
Observe that  the   solutions   in  (\ref{u_classifyeven1})  are  the same as those  in  (\ref{u_classifyodd}).
As can be seen from  the proof for    Lemma  \ref{Theorem_structureoflocal},
the same eigenvalues  for  $\mathbf M$ can be deduced for  (\ref{u_classifyeven1}).
However,
  the signs of $\sigma_a$ and $\sigma_b$ in (\ref{eigensign}) differ instead for even $m$ cases.
This is compared and clarified in the following remark:

\begin{wlremark}\label{rem.sign.even}
The  signs of $\sigma_a$ and $\sigma_b$ for different $m$  are   stated as follows.
	\begin{itemize}
		\item	When  $m$  is odd,
		      since $\alpha \ge  0 $  as stated in Lemma  \ref{rem.sign.alpha},
		      $  p(x) = x^{m-1} -\alpha x $ has  two real  roots $ x_1=0$, $ x_2  =
			      \sqrt[\uproot{3} {m-2}]
			      {
				      \alpha
			      }    > 0$ when $\alpha >0$, and has only one zero root when $\alpha =0$.
		      Then, 
		      $  f(x) = x^{m-1} -\alpha x - \beta = 0 $  with  $\beta >  0$
		     has   one positive and negative root,
		      indicating that    $\sigma_{a} > 0 $ and  $\sigma_{b} < 0$, as presented in (\ref{eq.Heigenclss}).
		\item 	For even $m$, 	      since $\alpha >  0$ as stated in Lemma  \ref{rem.sign.alpha},
		      $  p(x) = x^{m-1} -\alpha x = 0$ has   three  roots
		     as  stated in (\ref{equ.threeroot}).
		      When $\beta = 0$, this is the case;
		      when $\beta > 0$,
		      it  yields   two  negative roots and a positive one,
		      which     can be  observed from  Fig~\ref{even}.
		      Thus,
		      in  this case,  it only holds that  $\sigma_{a} > 0$ for $a>0$, while   the sign of  $\sigma_{b}$  for $b< 0$ may
		      vary  with  $k$.
	\end{itemize}
\end{wlremark}

With this remark, the proof is given as  follows:

\begin{proof}
	First,
	$\sigma_{b} $ as  a  function of  $k$  can be   expressed   as
	\begin{align}
		\sigma_{b}
		 & =
		(m-1)b^{m-2} - \alpha
		\nonumber \\
		 &
		=
		(m-1)( -\sqrt{	\frac{k} {(n-k) n}	})^{m-2}  -
		[ k ( \sqrt{	\frac{n-k} { kn}		})^{m} + (n-k) (-  \sqrt{	\frac{k} {(n-k) n}	})^{m} ]
		\nonumber \\
		 & =
		\frac
		{  (m-1)nk^{m-2} - (n-k)^{m-1}- k^{m-1}}
		{(n-k)^{r-1} k^{r-1}  n^{r}},
	\end{align}
	where  $ r = \frac {m} {2} 	\ge 2$.
	The  denominator  will always be positive   for $n \ge 3, m\ge 4,k \ge 1$, i.e.,
	$(n-k)^{r-1} k^{r-1}  n^{r} >0 $.
	We only focus on the sign of the numerator,  denoted  by
	\begin{align}
		l(m,n, k)
		:=
		(m-1)nk^{m-2} - (n-k)^{m-1}- k^{m-1}.
	\end{align}

	(1):  When $k=1$,
	it reduces to    $		l(m,n, 1) = (m-1)n - (n-1)^{m-1}-1 $  where  $n \ge 3$ and $m \ge 4$.
 Both $  (m-1)n $  and  $ (n-1)^{m-1} $ are  increasing   functions  with respect to  $m$ and $n$.
We first fix $n \ge 3$ and consider the difference
$
l(m+1,n,1) - l(m,n,1)
= n - (n-1)^{m-1}(n-2).
$
For $m \ge 4$, it holds that 
$
(n-1)^{m-1}(n-2) \ge (n-1)^3(n-2),
$
and $ (n-1)^3(n-2) > n $ for $n \ge 3$. 
Thus, for all $n \ge 3$ and $m \ge 4$,
$
l(m+1,n,1) - l(m,n,1) < 0.
$
So $l(m,n,1)$ is strictly decreasing in the range  $m \ge 4$. Hence,
$
l(m,n,1) \le l(4,n,1) = 3n - (n-1)^3 - 1.
$
Next, we fix $n \ge 3$ and consider
$
l(4,n+1,1) - l(4,n,1)
= 3 - \bigl(n^3 - (n-1)^3\bigr)
= 3 - (3n^2 - 3n + 1)
= 2 - 3n(n-1) < 0.
$
Therefore, $l(4,n,1)$ is strictly decreasing in the range $n \ge 3$, and so
$
l(4,n,1) \le l(4,3,1) = 3 \cdot 3 - (3-1)^3 - 1 = 9 - 8 - 1 = 0.
$
Combining the above inequalities yields
	\begin{align}
		l(m,n, 1)
	 \le 	l(4,3, 1) =0. 
	\end{align}
 If and only  if  $n = 3$ and $m = 4$,  the  equality holds.
	Since  we have excluded the special case $(m,n) = (4,3)$ 
	 from the scope of discussion (see    Lemma \ref{Theorem_structureofall}  for  details), 
	when $k=1$,
	it always holds that
	$	l(m,n,k)  <0 $,
$\sigma_{b} < 0$, and thus  $\mathbf M$   is  negative  semi-definite.
	 In addition, there are exactly two zero eigenvalues,  indicating  that  the corresponding  direction  $\mathbf u $  is  
	  a strict local maximum.

	(2): When $2 \le  k \le  \lfloor n/2  \rfloor $,
	there  are
	$k-1$ positive eigenvalues $\sigma_{a} $,
	and two $0$ eigenvalues  according to   (\ref{eigensign}),
	so the local
	optimality  is  determined by the sign of
	$l(m,n, k) $,
	which  varies  for   different combinations of  $(m,n)$.
If
$	l(m,n,k)  >0 $, 
	there  are
$n-2$ positive eigenvalues
and two zero   eigenvalues  for $\mathbf M$, which implies that   the  corresponding 
solution   is a strict local minimum, based on the presentation 
after   Theorem \ref{second_order_necessary} in  the 
Appendix; 
if $	l(m,n,k)  <0 $, it is a saddle point; 
 if $	l(m,n,k)  = 0 $, it is degenerate but is not a local maximum unless a separate higher-order analysis
is provided.
	Therefore, in this proof, we do not intend to
	exhaust all cases,  but conclude  as stated in the above lemma.
	The proof is complete.     $\blacksquare$
\end{proof}

\subsection{Proof of  Lemma   \ref{Theorem_structureoflocaleven2}}
\label{proof.lemma.Theorem_structureoflocaleven2}
\label{sec.odd_2}
In  this part,  we   analyze  the  signs  of    eigenvalues   of  $\mathbf M$     for    the  KKT points
in  (\ref{u_classifyeven2}) and aim to prove Lemma   \ref{Theorem_structureoflocaleven2},
which could be the most complex case we deal with.
We restate  the form
\begin{equation}\label{u3struc}
	\mathbf u
	=[
	c \mathbf 1_{p}^{\mathrm T},
	d \mathbf 1_{q}^{\mathrm T},
	e \mathbf 1_{s}^{\mathrm T}
	]^{\mathrm T},
\end{equation}
and  the following relationship holds:
\begin{equation}
\notag
 	e < d \le  0 <c .
\end{equation}
See (\ref{equ.threeroot}) and   Remark \ref{rem.sign.even} for details concerning the signs  of the three values.

With  this  assumption, the  sign  of  each  diagonal  element  in    $  \mathbf H:=\mathbf H (\mathbf u)$  can   be  determined,
which  can be  intuitively observed  from  Fig. \ref{even}  and  thus follows:
\begin{align}\label{eigenclss}
	\mathbf H _{ii}
	 & =
	[ (m-1) \operatorname{diag} (\mathbf u ^{\circledast^{m-2}})
	-\alpha \mathbf  I_{n}  ] _{ii}
	\nonumber
	\\
	 & =
	\begin{cases}
		(m-1) c^{m-2}-\alpha = \sigma_{c} >0  ,  \quad  \quad    u_{i}=c \\
		(m-1) d^{m-2}-\alpha = \sigma_{d}  <0,  \quad  \quad    u_{i}=d  \\
		(m-1) e^{m-2}-\alpha = \sigma_{e}  >0,  \quad  \quad    u_{i}=e  \\
	\end{cases}.
\end{align}
Then,    the  eigenvalues   of  $  \mathbf H $
are  sorted in   ascending  order  as   follows:
\begin{equation}\label{Heigenorder}
	\underbrace{
		\sigma_{d} = \dots =\sigma_{d} }_{q}
	<0
	<
	\underbrace{
		\sigma_{c} = \dots =\sigma_{c} }_{p}
	\le
	\underbrace{
		\sigma_{e} = \dots =\sigma_{e} }_{s} .
\end{equation}

Since $p \ge 1,q \ge 1,s \ge 1$
and $p +q +s =n$,
the  number  of    positive  eigenvalues   for  $
	\mathbf H
$
 satisfies
$ p+s \ge 2$.

\subsubsection{
		 Auxiliary lemmas}

Considering that there is no explicit form for $\mathbf u$ shown in   (\ref{u3struc}),
it is also  tough to deduce a benign block structure
for  $\mathbf M$
similar to
(\ref{eigensign}), as  analyzed in the above two cases.
To deal with this case,
 we first derive some  auxiliary lemmas
to facilitate  our discussion.

First,    $  \mathbf {M} $
is rewritten  as:
\begin{align}\label{matrix_Mfor3}
	\mathbf {M}
	 & =
	\mathbf  P_{\mathbf  A } ^{\bot}   \mathbf H (\mathbf u)  \mathbf  P_{\mathbf  A }^{\bot}
	=
	(\mathbf  I_{n}  -
	\mathbf u\mathbf u^{\mathrm T}
	- \frac  1n \mathbf 1_{n}\mathbf 1_{n}^{\mathrm T} )
	\mathbf H
	(\mathbf  I_{n}  -
	\mathbf u\mathbf u^{\mathrm T}
	- \frac  1n \mathbf 1_{n}\mathbf 1_{n}^{\mathrm T} )
		\nonumber \\
			& =	
	 		\mathbf H
		- 	(
		\mathbf u\mathbf u^{\mathrm T}
		+ \frac  1n \mathbf 1_{n}\mathbf 1_{n}^{\mathrm T} ) 	\mathbf H 
		-	\mathbf H  (
		\mathbf u\mathbf u^{\mathrm T}
		+ \frac  1n \mathbf 1_{n}\mathbf 1_{n}^{\mathrm T} ) 
		+(	\mathbf u\mathbf u^{\mathrm T}
		+ \frac  1n \mathbf 1_{n}\mathbf 1_{n}^{\mathrm T} )	\mathbf H  (
		\mathbf u\mathbf u^{\mathrm T}
		+ \frac  1n \mathbf 1_{n}\mathbf 1_{n}^{\mathrm T} ) .
\end{align}

Observe that   $\mathbf H$   is    a  diagonal  matrix,  and 
 it is easy to determine  its  eigenvalues  as  presented  in (\ref{Heigenorder}).
So,  the  following   task  is  to  determine     the   eigenvalues  of  the 
 remaining  part,
 which is denoted as   
	\begin{align}
	\label{equ.Hastdef}
	\mathbf H^{\ast} = 
(
\mathbf u\mathbf u^{\mathrm T}
+ \frac  1n \mathbf 1_{n}\mathbf 1_{n}^{\mathrm T} ) 	\mathbf H 
+	\mathbf H  (
\mathbf u\mathbf u^{\mathrm T}
+ \frac  1n \mathbf 1_{n}\mathbf 1_{n}^{\mathrm T} ) 
-  
(	\mathbf u\mathbf u^{\mathrm T}
+ \frac  1n \mathbf 1_{n}\mathbf 1_{n}^{\mathrm T} )	\mathbf H  (
\mathbf u\mathbf u^{\mathrm T}
+ \frac  1n \mathbf 1_{n}\mathbf 1_{n}^{\mathrm T} ) .
	\end{align}
Since both $\mathbf H$ and  $\mathbf H^{\ast} $ are symmetric, 
we then utilize the Weyl  theorem in Lemma \ref{weyltheo}
to determine the signs of  eigenvalues of   $ 	\mathbf M = \mathbf H - \mathbf H^{\ast} $,
which  	is then used to  check its  positive or negative  semi-definiteness.

We first conclude some  equations as follows, which will be useful to  facilitate our discussion.
 	By  (\ref{HU}), we have that  
$ \mathbf H 
\mathbf u
=
(m-2) \alpha  \mathbf u + (m-1) \beta \mathbf  1_{n}
$. 
It holds  that
\begin{equation}\label{HLN}
	\mathbf H
	\mathbf  1_{n}
	=
	(m-1)  \mathbf u ^{\circledast^{m-2}}-  \alpha  \mathbf  1_{n}.
\end{equation}
Furthermore, it  can be  derived  that
\begin{align}\label{UHU}
	\mathbf u^{\mathrm T}
	\mathbf H
	\mathbf u
	=
	(m-2) \alpha,
	\quad
	\mathbf  1_{n}^{\mathrm T}
	\mathbf H
	\mathbf  1_{n}
	=
	(m-1) \gamma
	-
	n  \alpha  ,
	\quad
	\mathbf u^{\mathrm T}
	\mathbf H
	\mathbf 1_{n}
	=
	n(m-1) \beta,
\end{align}
where
$ \gamma
	:=
	\mathbf 1_{n}^{\mathrm T}
	\mathbf u ^{\circledast^{m-2}}
	=
	\sum u_{i}^{m-2}
	>0
$
for  even  $m$  case, and in the last equation,
we use
that
$ \mathbf u^{\mathrm T}
	\mathbf u ^{\circledast^{m-2}}
	=
	n \beta
$.
 
	Then, it holds that 
	\begin{align}
\mathbf H  (
\mathbf u\mathbf u^{\mathrm T}
+ \frac  1n \mathbf 1_{n}\mathbf 1_{n}^{\mathrm T} )
=
	(m-2) \alpha \mathbf u\mathbf u^{\mathrm T} +(m-1) \beta \mathbf  1_{n} \mathbf u^{\mathrm T} 
	+ 
\frac  {m-1}  {n}  \mathbf u ^{\circledast^{m-2}} \mathbf 1_{n}^{\mathrm T} -   \frac  {\alpha }  {n} \mathbf  1_{n}\mathbf 1_{n}^{\mathrm T}.
	\end{align}
	Similarly, we have that 
		\begin{align}
 (
	\mathbf u\mathbf u^{\mathrm T}
	+ \frac  1n \mathbf 1_{n}\mathbf 1_{n}^{\mathrm T} ) 	\mathbf H  
	=
	(m-2) \alpha \mathbf u\mathbf u^{\mathrm T} +(m-1) \beta  \mathbf u  \mathbf  1_{n}^{\mathrm T} 
	+ 
\frac  {m-1}  {n} \mathbf 1_{n} (\mathbf u ^{\circledast^{m-2}})^{\mathrm T} -   \frac  {\alpha }  {n} \mathbf  1_{n}\mathbf 1_{n}^{\mathrm T}.
	\end{align}
	Using 
	$
	\mathbf u^{\mathrm T}  \mathbf u = 1,  \mathbf u^{\mathrm T}  \mathbf 1_n =  \mathbf 1_n^{\mathrm T}  \mathbf u =  0, \mathbf 1_n^{\mathrm T} \mathbf 1_n = n,
	$
	it holds that 
	\begin{align}
(\mathbf u \mathbf u^{\mathrm T}  + \frac{1}{n} \mathbf 1_n \mathbf 1_n^{\mathrm T} ) \mathbf H(\mathbf u \mathbf u^{\mathrm T}  + \frac{1}{n} \mathbf 1_n \mathbf 1_n^{\mathrm T} )
& = 
(m-2)\alpha \mathbf u \mathbf u^{\mathrm T} 
+(m-1) \beta   \mathbf 1_{n}  \mathbf u^{\mathrm T} 
\nonumber \\
& + (m-1)\beta \mathbf u \mathbf 1_n^{\mathrm T} 
+ \frac{ (m-1)\gamma }{n^2} \mathbf 1_n   \mathbf 1_n^{\mathrm T} 
- \frac{\alpha}{n} \mathbf 1_n \mathbf 1_n^{\mathrm T}  .
	\end{align}
	Combining those equations yields  that 
		\begin{align}
		\mathbf H^{\ast} 
		& =
	(m-2)\alpha \mathbf u \mathbf u^{\mathrm T} 
	-
	( \frac{ (m-1)\gamma }{n^2} + \frac{\alpha}{n} )\mathbf 1_n \mathbf 1_n^{\mathrm T} +
	 \frac  {m-1}  {n} ( \mathbf 1_{n} (\mathbf u ^{\circledast^{m-2}})^{\mathrm T}
	 +
	 (\mathbf u ^{\circledast^{m-2}})  \mathbf 1_{n}^{\mathrm T}
	 ).		
		\end{align}

	We have the following conclusion for eigenvalues of 	$	\mathbf H^{\ast} $:
\begin{wllemma}\label{Theorem_Gstructure}
For  the matrix
$	\mathbf H^{\ast} $ defined in (\ref{equ.Hastdef}),
where $\mathbf u$ is given by (\ref{u_classifyeven2}) 	   satisfying   $ \mathbf u^{\mathrm T}\mathbf 1_{n} =0 $ and $ \mathbf u^{\mathrm T}\mathbf u=1$,
the set of eigenvalues of $\mathbf{H}^{\ast}$ contains $n-3$ zero eigenvalues, one negative eigenvalue, and two positive nonzero  eigenvalues.
	\end{wllemma}
	\begin{proof}
 
			Let $\mathbf{Q} = [\mathbf{q}_1, \mathbf{q}_2, \ldots, \mathbf{q}_n] \in \mathbb{R}^{n \times n}$ be an orthogonal matrix satisfying
			\[
		 	\mathbf{q}_1 = \frac{\mathbf{1}_n}{\sqrt{n}},  \mathbf{q}_2 = \mathbf{u}, 
			\]
			and the remaining vectors $\mathbf{q}_3, \ldots, \mathbf{q}_n$ are arbitrary orthonormal vectors orthogonal to both $\mathbf{q}_1$ and $\mathbf{q}_2$.		
			Define
			\begin{align}
			\label{equ.rj}
				r_j := \mathbf{q}_j^T \mathbf{u}^{\circledast^{m-2}}, \qquad j = 3, 4, \ldots, n.
			\end{align}	
			Then,
			based on (\ref{HLN})  and (\ref{UHU}), the matrix $\mathbf{Q}^{\mathrm T} \mathbf{H}^{\ast} \mathbf{Q}$ has the following explicit form:
			\[
				\mathbf  G :=	\mathbf{Q}^{\mathrm T} \mathbf{H}^{\ast} \mathbf{Q}
				=
				\begin{bmatrix}
				g_{11} & g_{12} & g_{13}      & \dots  & g_{1n}      \\
				g_{21} & g_{22} & 0      & \dots  & 0      \\
				g_{31} & 0      & 0      & \dots  & 0      \\
				\vdots & \vdots & \vdots & \ddots & \vdots \\
				g_{n1} & 0      & 0      & \dots  & 0
				\end{bmatrix}
				=
				\begin{bmatrix}
				\dfrac{(m-1)\gamma}{n} - \alpha 
				& (m-1)\sqrt{n}\beta
				& \dfrac{m-1}{\sqrt{n}} r_3
				& \cdots
				& \dfrac{m-1}{\sqrt{n}} r_n \\[8pt]
				(m-1)\sqrt{n}\,\beta
				& (m-2)\alpha 
				& 0
				& \cdots
				& 0 \\[6pt]
				\dfrac{m-1}{\sqrt{n}} r_3
				& 0
				& 0
				& \cdots
				& 0 \\
				\vdots
				& \vdots
				& \vdots
				& \ddots
				& \vdots \\
				\dfrac{m-1}{\sqrt{n}} r_n
				& 0
				& 0
				& \cdots
				& 0
				\end{bmatrix},
			\]	
			where for simplicity, we represent its elements as: 
			$ g_{11} : =\dfrac{(m-1)\gamma}{n} - \alpha, 
		 g_{22} := (m-2)\alpha > 0,
			g_{12} = g_{21} : = \sqrt {n}(m-1)\beta  \ge 0 $  based on the discussion in Lemma  \ref{rem.sign.alpha}  for  even $m$  cases,
			$g_{1j} = g_{j1} := \frac{  (m-1) \mathbf q_{j}^{\mathrm T} \mathbf u ^{\circledast^{m-2}} }{ \sqrt{n}}
			=\frac{  (m-1) r_j }{ \sqrt{n}}  , j=3,4,\dots, n$.	
				With  this  form,  we then   analyze  the  signs  of   the  eigenvalues  of $\mathbf G$ (and equivalently those of $\mathbf H^{\ast}$
				up to an orthogonal transformation  $\mathbf Q$),
			which can  be  uniformly  transformed  into  finding     roots of  the  characteristic polynomial  as follows
			\begin{align}\label{charpoly}
			f(\varepsilon)
			&
			:=
			\det(\mathbf {G} -  \varepsilon \mathbf {I} )
			\nonumber \\
			& =
			(g_{22} - \varepsilon)
			(- \varepsilon)^{n-2}
			det(
			g_{11} - \varepsilon
			-
			\begin{bmatrix}
			g_{21} \\
			g_{31} \\
			\vdots \\
			g_{n1}
			\end{bmatrix}
			^{\mathrm T}
			\begin{bmatrix}
			g_{22} - \varepsilon & 0             & 0 & \dots & 0             \\
			0                    & - \varepsilon & 0 & \dots & 0             \\
			\vdots                                                           \\
			g_{n1}               & 0             & 0 & \dots & - \varepsilon
			\end{bmatrix}^{-1}
			\begin{bmatrix}
			g_{12} \\
				g_{13}      \\
			\vdots \\
				g_{1n}
			\end{bmatrix}
			)
			\nonumber \\
			& =
			(g_{22} - \varepsilon)
			(- \varepsilon)^{n-2}
			(
			g_{11} - \varepsilon
			-
			\frac{g_{12}^{2}}{g_{22} - \varepsilon}
			+
			\frac{g_{13}^{2}+\dots+ g_{1n}^{2}}{ \varepsilon}
			)
			\nonumber \\
			& =
		- 	(- \varepsilon)^{n-3}
			(
			\varepsilon^{3}
			-
			(g_{11} +g_{22}) \varepsilon^2
			+
			(g_{11} g_{22}
			-
			g_{12}^{2}
			-
			g_{13}^{2}-\dots- g_{1n}^{2}
			)  \varepsilon
			+
			g_{22} (g_{13}^{2}+\dots+ g_{1n}^{2})
			),
			\end{align}
			where
			in  the second equation,  we
			split  the matrix  into  two  $1 \times 1 $ and  $(n-1) \times (n-1)  $  blocks
			and  then utilize  (\ref{detblock}).

By $ f(\varepsilon) =0$, it necessarily holds that
		there are  $n-3$ zero   roots.
				For simplicity, denote $ g := 	g_{13}^{2}+g_{14}^{2}+\dots+ g_{1n}^{2}$, 
				and based on the conclusion stated in  Lemma \ref{lemma.g.sign}, it holds that $ g > 0$. 
			Then,
		the signs of the remaining three nonzero roots are  determined by
			\begin{align}\label{equ.pvarepsilon}
			p (\varepsilon)
			=
				\varepsilon^{3}
			-
			(g_{11} +g_{22}) \varepsilon^2
			+
			(g_{11} g_{22}
			-
			g_{12}^{2}
			-
		g
			)  \varepsilon
			+
			g_{22} g
			=0.
			\end{align}
Since $\mathbf  G$  is symmetric,
	its three nonzero eigenvalues are real; for simplicity, they are denoted by    $\varepsilon_1,\varepsilon_2, \varepsilon_3$.
			In the following, we tend to analyze the signs of the three nonzero roots.

		Since
			\begin{align}\label{gammaalpha}
			\gamma
			& =
			\sum\limits_{i=1}^{n}  u_{i}^{m-2}
			=
			(  \sum\limits_{i=1}^{n}  u_{i}^{m-2})
			(  \sum\limits_{i=1}^{n}  u_{i}^{2})
			=
			\sum\limits_{i=1}^{n}  u_{i}^{m}
			+
			\sum\limits_{i,j=1}^{n}  u_{i}^{m-2} u_{j}^{2}
			=
			\alpha + \theta,
			\end{align}
			where
			$  \theta >0$  for  even  $m$  cases,
			 since each term in $\theta$ is non-negative, and they cannot all be zero simultaneously due to the constraints on $\mathbf{u}$.
It necessarily implies that
			\begin{align}\label{trg}
			g_{11} +g_{22}
			=
			\dfrac{(m-1)\gamma}{n}  -\alpha +  (m-2)\alpha
			=
			\left( \frac{m-1}{n} + m - 3 \right)  \alpha +
			\frac {m-1}{n}  \theta
			> 0.
			\end{align}
			Similarly, it holds that 
		$		g_{22} g > 0$.
		By factorizing  $  p (\varepsilon) = (\varepsilon - \varepsilon_1)(\varepsilon - \varepsilon_2)(\varepsilon - \varepsilon_3)$
		and comparing with  (\ref{equ.pvarepsilon}),
		it holds that 
		\begin{align}\label{equ.rootsum}
	&	\varepsilon_1+\varepsilon_2+\varepsilon_3  = g_{11} +g_{22} > 0,
	\nonumber 	\\
& \varepsilon_1\varepsilon_2 +\varepsilon_2\varepsilon_3 +\varepsilon_1 \varepsilon_3 = g_{11} g_{22}
-
g_{12}^{2}
-
g,
\nonumber \\
& \varepsilon_1\varepsilon_2\varepsilon_3 = 	- g_{22}g < 0.
		\end{align}
The condition   $\varepsilon_1\varepsilon_2\varepsilon_3 < 0 $  implies 
	either one negative and two positive eigenvalues or three negative
eigenvalues, while 	$ \varepsilon_1+\varepsilon_2+\varepsilon_3   > 0 $  
rules out the latter possibility. 
Hence there must be exactly
one negative and two positive nonzero eigenvalues
for 	$p (\varepsilon) =0  $.

The proof is complete.  $\blacksquare$
\end{proof}

We then provide   an auxiliary lemma used in the above proof concerning the sign of $ g := 	g_{13}^{2}+g_{14}^{2}\dots+ g_{1n}^{2}$, which is concluded as  follows:
\begin{wllemma}\label{lemma.g.sign}
Let $n\ge 3, m\ge 3$ and $(m,n) \neq (4,3)$.	For  the term 
$ g := 	g_{13}^{2}+g_{14}^{2} + \dots+ g_{1n}^{2}$,
where 	$g_{1j}  := \frac{  (m-1) \mathbf q_{j}^{\mathrm T} \mathbf u ^{\circledast^{m-2}} }{ \sqrt{n}}$, 
$\mathbf q_{j}^{\mathrm T} \mathbf u =0$ for $j=3,4,\dots, n$,
and 
 $\mathbf u$ is given by (\ref{u_classifyeven2}) 
 satisfying   $ \mathbf u^{\mathrm T}\mathbf 1_{n} =0 $ and $ \mathbf u^{\mathrm T}\mathbf u=1$,
it holds that $ g >0$. 
\end{wllemma}

\begin{proof}
 
We prove this claim by contradiction. 
Observe that $g := g_{13}^2 + g_{14}^2 + \dots + g_{1n}^2 \ge 0$ by its definition. 
We tend to show that the case $g = 0$ is impossible in the following.
Assume that $g = 0$ holds. Then we have $r_j = 0$ for all $j = 3, 4, \ldots, n$, where $ r_j := \mathbf{q}_j^{\mathrm{T}}
 \mathbf{u}^{\circledast^{m-2}}. $
Since the vectors \( \mathbf{q}_3, \ldots, \mathbf{q}_n \) together with \( \mathbf{q}_1 = \frac{\mathbf{1}_n}{\sqrt{n}} \) and \( \mathbf{q}_2 = \mathbf{u} \) form an orthonormal basis of \( \mathbb{R}^n \), the condition \( r_j = 0 \) for all \( j \ge 3 \) implies that \( \mathbf{u}^{\circledast^{m-2}} \) is orthogonal to \( \mathbf{q}_3, \ldots, \mathbf{q}_n \). Therefore, \( \mathbf{u}^{\circledast^{m-2}} \) lies in the subspace spanned by \( \mathbf{q}_1 \) and \( \mathbf{q}_2 \), i.e.,
\[
\mathbf{u}^{\circledast^{m-2}} \in \operatorname{span}\{\mathbf{1}_n, \mathbf{u}\}.
\]
Equivalently, there exist constants \( c_1, c_2 \in \mathbb{R} \) such that
\[
\mathbf{u}^{\circledast^{m-2}} = c_1 \mathbf{1}_n + c_2 \mathbf{u},
\]
which is equivalent to 
\[
{u}_i^{ m-2} = c_1  + c_2 u_i, i =1,2\dots, n.
\]
Based on the KKT condition  in  (\ref{KKTgradient}), it holds that 
$ {u}_i^{ m-1} = \alpha  u_i +\beta, i =1,2\dots, n$. 
This implies that 
$ \alpha  u_i +\beta =  {u}_i^{ m-1} = {u}_i^{ m-2}u_i  = c_1u_i  + c_2 u_i^2$. 
Therefore, 
 ${u}_i$ is a  solution of $ c_2 u_i^2 + ( c_1- \alpha )u_i  -\beta  = 0$.
Clearly, there are at most two distinct real roots, which contradicts the fact that $\mathbf{u}$ has three distinct 
values given by (\ref{u_classifyeven2}).
 Thus, the assumption does not hold, and we prove that  $ g:= g_{13}^2 + \cdots + g_{1n}^2 >  0$. 
 $\blacksquare$
\end{proof}

\subsubsection{Proof}

 So far, we have clearly known the eigenvalues  distribution concerning the two divided parts of $\mathbf{M}$, i.e., $\mathbf{H}$ (as listed in (\ref{Heigenorder})) and $\mathbf{H}^{\ast}$ (as concluded in Lemma \ref{Theorem_Gstructure}).
Recalling that $p + s \ge 2$, Lemma \ref{Theorem_structureoflocaleven2} can be proved by considering the following two cases:

\begin{itemize}
	\item  
		When $p + s \ge 3$, for any even $m$, there are at least three positive eigenvalues of $\mathbf{H}$. 
		Assume that  the $n$   eigenvalues are  sorted  
		in    ascending  order,  we have that  $\lambda_{n-2} (\mathbf{H}) >0$. 
		For $-\mathbf{H}^{\ast}$, based on the conclusion in Lemma \ref{Theorem_Gstructure} with an extra sign-flipping operation, there are two negative eigenvalues, and   $\lambda_{3} (-\mathbf{H}^{\ast}) =0$. 
		Based on Weyl's theorem in Lemma \ref{weyltheo}, 
		it yields that $ \lambda_{n} (\mathbf{M}) \ge  \lambda_{n-2} (\mathbf{H}) + \lambda_{3} (-\mathbf{H}^{\ast}) > 0$, 
	and 	there is at least one positive eigenvalue of $\mathbf{M} = \mathbf{H} - \mathbf{H}^{\ast}$, which implies that the corresponding solution cannot be locally maximized.

	\item   
		When $p + s = 2$ ($p = 1$, $s = 1$, $q = n-2$), there are exactly two positive eigenvalues of $\mathbf{H}$ and two negative eigenvalues of $-\mathbf{H}^{\ast}$.
		Therefore, by only using the conclusion in Lemma \ref{Theorem_Gstructure}, we cannot conclude that there are definitely positive eigenvalues for  $\mathbf{M}$ based on Weyl's theorem in Lemma \ref{weyltheo}.
		We adopt a new strategy to deal with this case.
		Based on  the conclusion in Lemma  \ref{Theorem_structureofall},
		we can derive a explicit form for $\mathbf M$ in this case.
		First, for  $m=4$, 
		a simple calculation by
	$	(u_{i}-c)
		(u_{i}-d)
		(u_{i}-e)
		=
		u_{i}^{3} -\alpha u_{i} - \beta
	$ 
		deduces that
		$ c+d+e=0$.
		Then, when  $p = 1$, $s = 1$, and $q = n-2$, 
	combining this with 	$	pc^{2}+q d^{2} +s e^{2} =1 $ and $pc+q d +s e =0 $    
	yields that $d = 0$, and $c = -e = \frac{\sqrt{2}}{2}$. 
	We then consider the case with even $m \ge 6$ in general. 
Based on the analysis in the  later  Lemma  \ref{lem.ps1},
	when $p=s=1$,  the same conclusion can be derived. 
		Consequently, 
		$\sigma_d = -\alpha$, $\alpha = c^m + (-e)^m = 2c^m$, and $\sigma_c = \sigma_e = (m-1)c^{m-2} - \alpha = 2(m-1)c^m - \alpha = (m-2)\alpha$.
		Without loss of generality and for a convenient discussion, 
		we set  $\mathbf u$ as follows:
		\begin{equation}\label{u3strucnew}
		\mathbf{u} = \left[ \frac{\sqrt{2}}{2},\ -\frac{\sqrt{2}}{2},\ \mathbf{0}_{n-2}^{\mathrm{T}} \right]^{\mathrm{T}},
		\end{equation}
		which is a   permutation-equivalence  result of   (\ref{u_classifyeven2}) for $p = 1$, $s = 1$, and $q = n-2$.  
		$\mathbf{u}\mathbf{u}^{\mathrm{T}}$ can be expressed in the following block form:
		\begin{align}\label{publock_new}
		\mathbf{u}\mathbf{u}^{\mathrm{T}}
		=
		\begin{bmatrix}
		\frac{1}{2}    &   - \frac{1}{2}      & \mathbf{O}_{1 \times (n-2)} \\
		-\frac{1}{2}    &    \frac{1}{2}      & \mathbf{O}_{1 \times (n-2)} \\
		\mathbf{O}_{(n-2) \times 1} & \mathbf{O}_{(n-2) \times 1}  &  \mathbf{O}_{n-2}
		\end{bmatrix}.
		\end{align}
		With some calculations, it holds that 
		\begin{align}\label{pup1nNEW}
		\mathbf{P}_{\mathbf{A}}^{\perp}
		& =
		\begin{bmatrix}
		\left(\frac{1}{2} - \frac{1}{n}\right) \mathbf{J}_2 & \quad - \frac{1}{n} \mathbf{J}_{2 \times (n-2)} \\
		- \frac{1}{n} \mathbf{J}_{(n-2) \times 2}           &\quad  \mathbf{I}_{n-2} - \frac{1}{n} \mathbf{J}_{n-2}
		\end{bmatrix},
		\end{align}
		and $\mathbf{H}$ has the form 
		\begin{align}\label{Hblocknew}
		\mathbf{H}
		=
		\begin{bmatrix}
		(m-2)\alpha  	\mathbf{I}_2    & \quad  \mathbf{O}_{2 \times (n-2)}   \\
		\mathbf{O}_{(n-2) \times 2} & \quad - \alpha \mathbf{I}_{n-2}
		\end{bmatrix}.
		\end{align}
		Combining those results leads to 
		\begin{align}
		\mathbf{M} =
		\alpha
		\begin{bmatrix}
		\left[ 2(m-2)\left(\frac{1}{2} - \frac{1}{n}\right)^2 - \frac{n-2}{n^2} \right]\mathbf{J}_2
		&
		-2\left[ \frac{m-2}{n}\left(\frac{1}{2} - \frac{1}{n}\right) - \frac{1}{n^2} \right]\mathbf{J}_{2\times(n-2)} \\[1.2em]
		-2\left[ \frac{m-2}{n}\left(\frac{1}{2} - \frac{1}{n}\right) - \frac{1}{n^2} \right]\mathbf{J}_{(n-2)\times 2}
		&
		-\mathbf{I}_{n-2} + \frac{2m+n-2}{n^2}\mathbf{J}_{n-2}
		\end{bmatrix}.
		\end{align}
		Furthermore, we decompose $\frac{1}{\alpha}\mathbf{M}$ into two parts $\mathbf{A} + \mathbf{B}$, where 
		\begin{align}
		\mathbf{A}=
		\begin{bmatrix}
		\left[ 2(m-2)\left(\frac{1}{2} - \frac{1}{n}\right)^2 - \frac{n-2}{n^2} \right]\mathbf{J}_2
		& \mathbf{O}_{2 \times (n-2)}
		\\[1.2em]
		\mathbf{O}_{(n-2) \times 2} &
		-\mathbf{I}_{n-2} + \frac{2m+n-2}{n^2}\mathbf{J}_{n-2}
		\end{bmatrix},
		\end{align}
		and 
		\begin{align}
		\mathbf{B}=
		\begin{bmatrix}
		\mathbf{O}_2
		&
		-2\left[ \frac{m-2}{n}\left(\frac{1}{2} - \frac{1}{n}\right) - \frac{1}{n^2} \right]\mathbf{J}_{2\times(n-2)} \\[1.2em]
		-2\left[ \frac{m-2}{n}\left(\frac{1}{2} - \frac{1}{n}\right) - \frac{1}{n^2} \right]\mathbf{J}_{(n-2)\times 2}
		&
		\mathbf{O}_{n-2}
		\end{bmatrix}.
		\end{align}
		Clearly, $\mathbf{A}$ is a block diagonal matrix, and $\mathbf{B}$ also has a special structure. 
		For the rank-one matrix $\mathbf{J}_{p_1 \times p_2} = \mathbf{1}_{p_1} \mathbf{1}_{p_2}^{\mathrm{T}}$, it is easy to verify that its nonzero singular value is equal to $\sqrt{p_1 p_2}$.
		When $p_1 = p_2 = p$, one of its eigenvalues is $p$, and the remaining eigenvalues are $0$.
		For a matrix of the form $-\mathbf{I}_p + c_1\mathbf{J}_p$, where $c_1$ is a constant, it can be verified that $\mathbf{1}_p$ is an eigenvector with eigenvalue $c_1p - 1$, and the remaining $p-1$ eigenvectors are orthogonal to $\mathbf{1}_p$ with eigenvalues all equal to $-1$. 
		Using the conclusions in Lemma \ref{direct_eig} and Lemma \ref{lemma.svdblock}, the eigenvalues of $\mathbf{A}$ and $\mathbf{B}$ are:
		\begin{align}
		\underbrace{-1}_{n-3},\ 0,\ \frac{2(m-2)(n-2)-4}{n^2},\ \frac{(n-2)[(m-2)(n-2) - 2]}{n^2},
		\end{align}
		and 
		\begin{align}
		-2\sqrt{2(n-2)}\,\omega,\ \underbrace{0}_{n-2},\ 2\sqrt{2(n-2)}\,\omega, 
		\end{align}
		respectively, 
		where $\omega := \left| \frac{m-2}{n}\left(\frac{1}{2} - \frac{1}{n}\right) - \frac{1}{n^2} \right|$.
		Since $m \ge 4$, $n \ge 3$, and $(m,n) \neq (4,3)$, it holds that
		$2(m-2)(n-2) - 4 > 2 \times 2 - 4 = 0$. 
		Similarly, $(n-2)[(m-2)(n-2) - 2] > 0$.  
		It can be seen that there are two positive eigenvalues for  $\mathbf{A}$ and one negative eigenvalue for $\mathbf{B}$.
		Based on Weyl's theorem stated in Lemma \ref{weyltheo}, there is at least one positive eigenvalue for  $\mathbf{M} = \alpha(\mathbf{A} + \mathbf{B})$, since $\alpha > 0$ for even $m$.
		This implies that the corresponding solution cannot be locally maximized. 
\end{itemize} 
Combining the two cases completes the proof of Lemma \ref{Theorem_structureoflocaleven2}. 
$\blacksquare$

	  We then provide   an auxiliary lemma used in the above proof concerning the case  of $ p=s=1,q=n-2 	$ with even $m \ge 6$, which is concluded as  follows:
\begin{wllemma}\label{lem.ps1}
	Concerning  the solutions with the form of   (\ref{u_classifyeven2}) as stated in  Lemma \ref{Theorem_structureofall}, 
	for even $m \ge 6$, 
	if \(p=s=1\), it necessarily implies that 
	\[
	d=0,\qquad \beta=0,\qquad c=-e=\frac{\sqrt{2}}{2}.
	\]
\end{wllemma}
\begin{proof}
We prove this claim by contradiction.  
Observe that   $e < d	  \le   0 <c$  as stated in  Lemma \ref{Theorem_structureofall}. 
We tend to show that the case $d <  0$ is impossible in the following 
	for   \(p=s=1\) and \(q=n-2\ge 1\) where  $n \ge 3$ and even $m \ge 6$.  
	In this case, 
	(\ref{u_classifyeven2})   reads
	\[
	\mathbf u=[c,d,\ldots,d,e]^{\mathrm T},
	\]
	where \(d\) is repeated \(q\) times, and the two basic  constraints are
	$
	c+qd+e=0,  c^2+qd^2+e^2=1.
	$
	Moreover, \(c,d,e\) are the three real roots of
	$
	x^{m-1}-\alpha x-\beta=0,
	$
	with  even \(m\ge 6\).	
	Suppose, for contradiction, that \(d<0\). Since \(e<d  <  0\), denote 
	\[
	d=-t_1,\qquad e=-t_2,\qquad 0<t_1<t_2.
	\]
	The constraint \(c+qd+e=0\) gives
	$
	c=qt_1+t_2.
$
	Substituting \(d=-t_1\) and \(e=-t_2\) into the root equation, and using that \(m-1\) is odd, we get
	$
	(-t_1)^{m-1}=-t_1^{m-1},  (-t_2)^{m-1}=-t_2^{m-1}.
	$
	Hence
	\begin{align}
	\alpha t_1-\beta=t_1^{m-1},\qquad \alpha t_2-\beta=t_2^{m-1}.
		\end{align}
	Subtracting the two equations yields
	$
	\alpha(t_2-t_1)=t_2^{m-1}-t_1^{m-1}.
	$
Then, it holds that 
	\begin{align}\label{equ.alphares}
	\alpha=\frac{t_2^{m-1}-t_1^{m-1}}{t_2-t_1}
	=
	t_1^{m-2}  \sum_{i=0}^{m-2}(\frac{t_2}{t_1})^i.
		\end{align}
In addition, 
	$
	\beta=\alpha t_1-t_1^{m-1}>0.
	$
	Now use the root equation for \(c=qt_1+t_2\),
	it holds that 
$
	c^{m-1}=\alpha c+\beta=\alpha c+\alpha t_1-t_1^{m-1}=\alpha(c+t_1)-t_1^{m-1},
	$
which leads to 
	$
	c^{m-1}+t_1^{m-1}=\alpha(c+t_1).
	$
	Substituting \(c=qt_1+t_2\), we obtain
	\begin{align}\label{equ.t2qt1}
	(t_2+qt_1)^{m-1}+t_1^{m-1}
	=\alpha\bigl(t_2+(q+1)t_1\bigr).
	\end{align}
	Let \(t_2=rt_1\) with \(r>1\) based on the setting $0<t_1<t_2$. Then, 
	(\ref{equ.t2qt1}) can be rewritten as 
	\begin{align}\label{equ.maincompare}
	(r+q)^{m-1}+1
	=(r+q+1)\sum_{i=0}^{m-2}r^i,
	\end{align}
	where (\ref{equ.alphares}) is utilized. 
	\\
We now regard both sides of \eqref{equ.maincompare} as polynomials in the real variable $r$.
Define the polynomial $L(r) :=(r+q)^{m-1}+1$ and $R(r) :=(r+q+1)\sum_{i=0}^{m-2}r^i$.
Expanding them, the left-hand side reads
\[
L(r)
=r^{m-1}+(m-1)q\,r^{m-2}+\binom{m-1}{2}q^2r^{m-3}
+\cdots + (m-1)q^{m-2}r+q^{m-1}+1,
\]
and the right-hand side becomes
\[
R(r)
=r^{m-1}+(q+2)r^{m-2}+(q+2)r^{m-3}
+\cdots+(q+2)r+(q+1).
\]
For the coefficient of $r^{m-2}$, $L(r)$ has coefficient $(m-1)q$, while $R(r)$ has coefficient $q+2$. Since $m\ge 6$ and $q\ge 1$,
\[
(m-1)q-(q+2)=(m-2)q-2\ge 4q-2>0.
\]
For every other monomial $r^k$ with $1\le k\le m-3$, the coefficient of $r^k$ in $L(r)$ is $\binom{m-1}{k}q^{m-1-k}$, whereas the corresponding coefficient in $R(r)$ is always $q+2$.
For $ q \ge 1$ and $ m \ge 6$, we have
\[
\binom{m-1}{k}q^{m-1-k}
\ge  (m-1)q^{m-1-k}
\ge  (m-1)q^2
> q+2.
\]
The constant term of $L(r)$ is $q^{m-1}+1$, which is greater than or equal to the constant term $q+1$ of $R(r)$ for $ q\ge 1$ and even $m \ge 6$. 
All coefficients of $L(r)-R(r)$ are non-negative, and at least one coefficient is strictly positive. Consequently, $L(r)>R(r)$ for all $r>0$.
This means there is no real number $r>1$ satisfying equality \eqref{equ.maincompare}, yielding a contradiction.
 Therefore the assumption \(d<0\) is impossible, and we must have \(d=0\).
	\\
	With \(d=0\)  and $p=s=1$, the root equation evaluated at \(x=0\) gives
	\[
	0^{m-1}-\alpha\cdot 0-\beta=0\quad\Longrightarrow\quad \beta=0.
	\]
	The constraint \(c+qd+e=0\) then yields \(c+e=0\), i.e., \(c=-e\). Combining this with the equation  
	$
	c^2+e^2=1
$
	gives
	\[
	c=\frac{\sqrt{2}}{2},\qquad e=-\frac{\sqrt{2}}{2}.
	\]
	We also note that, with $d=0$ and $\beta=0$, the polynomial becomes $x^{m-1}-\alpha x = x(x^{m-2}-\alpha)$, whose roots are exactly $0,\pm \alpha^{1/(m-2)}$, consistent with the three real roots $c,d,e$.
	The proof is complete.  $\blacksquare$
\end{proof}

\section{Proof of   Theorem  \ref{conjecturesimplex}}\label{sec.Conjecture}
We provide a detailed proof for Theorem  \ref{conjecturesimplex}, 
which, for convenience,  is restated as follows:
\begin{theorem} 
	Let $\mathcal{S}_{\mathbf{W}}\in T^{m}(\mathbb{R}^{n-1})$ be a regular simplex tensor with $n\geq 3$, $m\geq 3$. Then,  the following statements hold:	
	\\
	(1):  If 	 $(m,n) = (3,3)$, $(3,4)$, or $(4,3)$, there are  no robust eigenpairs, either among 
	 the vectors in the regular simplex  frame or elsewhere.
	\\
	(2):  For all other $(m,n)$ combinations, the only robust eigenvectors, considered up to sign  equivalence, 
	are   	 the  vectors $\mathbf{w}_1,\ldots,\mathbf{w}_n$ in the  regular simplex  frame $\mathbf W$.
\end{theorem}
 
 \begin{proof}
  	Clearly, 			when   $(m,n) = (4,3)$ where the 
 	objective  value 
 	$
 	\mathcal S_{\mathbf W } \mathbf v^{m}
 	$   is a constant as  commented in Remark  \ref{RemarkScope},
 	then  	any unit-length vector is 
 	locally maximized and minimized.
 	In this case, 
 	   no robust eigenpairs exist.
    This case has also been analyzed and proved in  Theorem 16 of   \cite{teneigenstructure}.
    One can refer to it  for more details.
 
  For all other $(m,n)$ combinations with  $(m,n) \neq (4,3)$,
to identify all robust eigenpairs of   the    regular simplex tensor, we proceed as follows.
First, 
the sign  relation of equivalent eigenpairs as discussed in 
Lemma  \ref{lemma.sign} and the exclusion of   $\lambda = 0$ for   the   robustness issue
  guarantee
  that it suffices to focus only on eigenpairs with $\lambda > 0$.
Next, Lemma~\ref{RobustLocal} implies that among all   eigenpairs with $\lambda > 0$, only locally maximized eigenpairs can potentially be robust.
To identify the   locally maximized eigenpairs, we provide a detailed analysis in Section~\ref{sec.proof.opt}. This leads to Theorem~\ref{Theorem.local_eigenpair}, which states that,
 except for the special combination  $(m,n) = (4,3)$, the only locally maximized eigenpairs of   the
 regular simplex tensor are precisely the vectors in the  regular simplex  frame.
Consequently, the remaining task is merely to verify the robustness of
 the frame vectors 	$\mathbf{w}_{1}, \ldots, \mathbf{w}_{n}$ for $(m,n) \neq (4,3)$.

Concerning  the   vectors   in the regular simplex  frame
$
	\mathbf W =
	[
	\mathbf{w}_{1}, \ldots, \mathbf{w}_{n}
	] \in
	\mathbb{R}^{(n-1) \times  n}$,
it holds that
\begin{align}
	\mathbf{w}_{i}^{\mathrm T} \mathbf{w}_{i} =1,
	\mathbf{w}_{i}^{\mathrm T} \mathbf{w}_{j} =
	-\theta = -\frac{1}{n-1}
	( i \neq j),  \quad
	\mathbf{W} \mathbf{W}^{\mathrm T}=
	\sum_{i=1}^{n}
	\mathbf{w}_{i}
	\mathbf{w}_{i}^{\mathrm T}
	=
	\frac{n}{n-1} \mathbf{I}_{n-1}.
\end{align}
It can be calculated that  when
$ \mathbf v = \mathbf w_{j}$,
it holds that

\begin{align}\label{swj}
	\mathcal{S}_{\mathbf W}  \mathbf v^{m-2}
	 & =
	\mathcal{S}_{\mathbf W}   \mathbf w_{j}^{m-2}
	=
	(
	\sum_{i=1}^{n} \mathbf{w}_{i}^{\circ m}
	)  \mathbf w_{j}^{m-2}
	\nonumber \\
	 & =
	\sum_{i=1}^{n} (  \mathbf w_{j}^{\mathrm T} \mathbf{w}_{i})^{ m-2}
	\mathbf{w}_{i}
	\mathbf{w}_{i}^{\mathrm T}
	=
	\mathbf{w}_{j}
	\mathbf{w}_{j}^{\mathrm T}
	+
	\frac
	{  1}
	{ (1-n)^{ m-2}  }
	\sum_{i=1, i \neq j}^{n}
	\mathbf{w}_{i}
	\mathbf{w}_{i}^{\mathrm T}
	\nonumber \\
	 & =
	(
	1
	-
	\frac
	{  1}
	{ (1-n)^{ m-2}  }
	)
	\mathbf{w}_{j}
	\mathbf{w}_{j}^{\mathrm T}
	+
	\frac
	{ n}
	{ (1-n)^{ m-2} (n-1) }   \mathbf{I}_{n-1}
	.
\end{align}

In  addition,
the  corresponding  eigenvalue can be calculated by
\begin{equation}\label{lmdwj}
	\lambda
	=
	\mathcal{S}_{\mathbf W}   \mathbf  {w}_{j}^{m}
	=
	\sum_{i=1}^{n} ( \mathbf  {w}_{j}^{\mathrm T} \mathbf{w}_{i})^{ m}
	=
	1+
	\frac
	{  n-1}
	{ (1-n)^{ m}  }
	,
\end{equation}
	 which is always   greater    than 0  for $ n \ge 3, m\ge 3$, and $ (m,n) \neq  (4,3)$, satisfying 
 the condition that $ \lambda > 0$ for robustness checking.
Then,  according to  (\ref{Jacobianmatirx}),
the Jacobian matrix
at the eigenvector  $ \mathbf  {w}_{j} $
will satisfy
\begin{equation}
	\mathbf  J (\mathbf  {w}_{j},\lambda)
	\mathbf  {w}_{j}
	=
	\frac{m-1}{\lambda}
	(
	\mathcal{S}  \mathbf  {w}_{j}^{m-2}- \lambda  \mathbf  {w}_{j} \mathbf  {w}_{j}^{ \mathrm T }
	)
	\mathbf  {w}_{j}
	=
	0
	\cdot
	\mathbf  {w}_{j},
\end{equation}
which means
that
$  {\mathbf  {w}_{j} }$ itself
is an eigenvector of
$  \mathbf  J (\mathbf  {w}_{j},\lambda )  $ with eigenvalue 0,
and naturally, the other eigenvectors  lie in the
null  space of 	$  	\mathbf {w}_j$, denoted by
$\operatorname{null}(\mathbf  {w}_{j}  )$.
By further
utilizing  (\ref{swj})  and (\ref{lmdwj}),
it can be easily checked that
the  remaining  eigenvalues  of  $  \mathbf  J (\mathbf  {w}_{j},\lambda )  $
are the same and
there are only two different eigenvalues concerning
$   \mathbf  J (\mathbf  {w}_{j},\lambda )$,
which are
\begin{equation}
	0,
	\underbrace{
		\frac
		{ n  (m-1) }
		{
			1+
			(1-n)^{ m-2} (n-1) }  ,
		\dots,
		\frac
		{ n  (m-1) }
		{
			1+
			(1-n)^{ m-2} (n-1) }
	}_{n-2}
	.
\end{equation}

By considering
the odd and even cases for $m$ separately,
it can be concluded that
the  spectral radius of
$   \mathbf  J (\mathbf  {w}_{j},\lambda )$
is given by
\begin{align}\label{radiuscase}
	\rho (\mathbf J (\mathbf  {w}_{j},\lambda))
	=
	\begin{cases}
		\frac
		{ n  (m-1) }
		{ (n-1)^{ m-1}-1 }    ,
		\quad   m = 2k+1, k \in \mathbb N
		\\
		\frac
		{  n (m-1)}
		{(n-1)^{ m-1} +1 }    ,
		 \quad  m = 2k, k \in \mathbb N
	\end{cases},
\end{align}
where
for the odd $m$ cases,
since $
	1+
	(1-n)^{ m-2} (n-1)
	=	1-
	(n-1)^{ m-1}<0$
for $n \ge 3$,  $m \ge 3$,
the numerator
should be    replaced by  its  absolute value.

Note that
the   right-hand   of
(\ref{radiuscase})
has also been  derived
in Theorem 4.6
of   \cite{RobustEigen}.
See the first formula in its  proof for details.
Therefore,
the sequential
analysis
is the same as that for  Theorem 4.6.
 
	Two cases are discussed as follows:
	\begin{itemize}
\item  It can be checked that
for $n \ge 3$,  $m \ge 3$,
when
$n+ m \ge 8$,
it holds that
$ \rho (\mathbf J (\mathbf  {w}_{j},\lambda))  < 1 $,
and the
vectors in the frame of the  corresponding
regular  simplex  tensor  are  robust, demonstrating their uniqueness for those  combinations.
\item 	  If $(m,n) = (3,3)$  or  $(3,4)$,  $ \rho (\mathbf J (\mathbf  {w}_{j},\lambda))  \ge 1 $  and no robust eigenpairs exist.
\end{itemize}
In addition, for the special case $(m,n) = (4,3)$,
we can also obtain that $ \rho (\mathbf J (\mathbf  {w}_{j},\lambda)) = 1 $, 
which implies that the vectors in the frame are   non-robust, which 
is also consistent with the claim at   the beginning  of this proof. 
 Then, combining those results, 
we complete the proof for the  theorem, and state the theorem in the form presented above. 
$\blacksquare$
 \end{proof}

\begin{wlremark}[Comparison with Theorem 4.6
		in \cite{RobustEigen}]\label{compare46}
	Although
	the analysis
	of   Theorem 4.6
	in \cite{RobustEigen}
	is  also based on
	(\ref{radiuscase}),
	the  difference lies 
	 in that the result
	presented in Theorem 4.6
	is only an  	 upper   bound    result for  $\rho (\mathbf J (\mathbf  {w}_{j},\lambda)) $.
	See  Equations   (4.4) and (4.8) in \cite{RobustEigen} for details.
	 Here, however,  for
	 the vectors in the regular simplex frame, we
	accurately
	determine  the spectral radius
	by analyzing its eigenvalues distribution.
	In other words,
	we further show that
	the spectral radius  for 
	 the vectors in the  regular simplex frame  can 
	 reach   the bound  in Theorem 4.6,
	which thus   yields   a strict equality in  (\ref{radiuscase}).
\end{wlremark}

The above proof also confirms that
the  robust eigenpairs of   the   regular simplex tensor  do not exist
for the case of $(m,n) = (3,3)$, $(3,4)$, and $(4,3)$. This is also consistent with the experimental result
presented in Table 1 
 of  
\cite{RobustEigen}.
For    illustration,
Table  \ref{radius}   shows the spectral radius
of
$ \mathbf J (\mathbf  {w}_{j},\lambda)$
for  different  combinations   of  $(m,n)$.
Table  1
 of  
\cite{RobustEigen}
 records   the experimental results  on
 whether   the
tensor power  method successfully    converges to the vectors in the  regular simplex  frame.
Our direct  results  for  the  spectral radius  are also consistent  with those   presented in Table  1
 of
\cite{RobustEigen}. Note that
the case of  $m=2$
which 
 corresponds    to the matrix case  is not included.

\begin{table}[!ht]
	\normalsize
	\centering
	\caption{\\
		The  spectral  radius  calculation results
		of the  Jacobian matrix  at  
		 the vectors  in the  regular simplex  frame
		$ \mathbf w_{j}$  for the  regular simplex tensor
		$ \mathcal{S}_{\mathbf W}:=\sum_{i=1}^{n} \mathbf{w}_{i}^{\circ m}$
		with different  combinations of $(m,n)  $   based on  (\ref{radiuscase}),
		where $m,n-1$  are  the order and dimension of  the  tensor, respectively.
	}
	\setlength{\tabcolsep}{8pt}
	\renewcommand\arraystretch{1.5}
	\begin{tabular}{ | c|  c  c  c  c  c  c c c|  }
		\hline
		\diagbox{$n $ }{$\rho (\mathbf J (\mathbf  {w}_{j},\lambda)) $ }{$m $ }
		 & 3
		 & 4
		 & 5
		 & 6
		 & 7
		 & 8
		 & 9
		 & 10
		\\
		\hline
		3
		 & \bf  2
		 & \bf   1
		 & $\frac{4 }{ 5}$
		 & $\frac{ 5}{ 11} $
		 & $\frac{ 2}{7}$
		 & $\frac{7 }{43 }$
		 & $\frac{ 8}{ 85}$
		 & $\frac{1 }{19} $
		\\
		\hline
		4
		 & \bf  1
		 & $\frac{3 }{7}$
		 & $\frac{1 }{ 5} $
		 & $\frac{ 5}{ 61}  $
		 & $\frac{ 3}{ 91}$
		 & $\frac{ 7}{547}$
		 & $\frac{ 1}{205 }$
		 & $\frac{9 }{ 4921}$
		\\
		\hline
		5
		 & $\frac{2 }{3}$
		 & $\frac{3 }{13 }$
		 & $\frac{4 }{ 51}$
		 & $\frac{ 1}{41 }$
		 & $\frac{2 }{273 }$
		 & $\frac{7 }{ 3277}$
		 & $\frac{8  }{ 13017}$
		 & $\frac{9 }{ 52429}$
		\\
		\hline
		6
		 & $\frac{ 1}{2 }$
		 & $\frac{ 1}{7 }$
		 & $\frac{1 }{26}$
		 & $\frac{ 5}{ 521}$
		 & $\frac{1 }{434 }$
		 & $\frac{ 7}{ 13021}$
		 & $\frac{1 }{8138 }$
		 & $\frac{ 1}{36169 }$
		\\
		\hline
		7
		 & $\frac{2 }{ 5}$
		 & $\frac{3 }{31 }$
		 & $\frac{ 4}{ 185}$
		 & $\frac{5 }{ 1111 }$
		 & $\frac{ 6}{ 6665}$
		 & $\frac{ 1}{ 5713}$
		 & $\frac{8 }{ 239945}$
		 & $\frac{ 2}{ 319927}$
		\\
		\hline
		8
		 & $\frac{1 }{ 3}$
		 & $\frac{ 3}{43 }$
		 & $\frac{1 }{75 }$
		 & $\frac{5 }{2101 }$
		 & $\frac{1 }{2451 }$
		 & $\frac{7 }{ 102943}$
		 & $\frac{ 1}{ 90075}$
		 & $\frac{ 1}{560467 }$
		\\
		\hline
	\end{tabular}
	\label{radius}
\end{table}

\section{ Conclusion and future  work}\label{futurework}

In this paper,
we mainly focus on a conjecture concerning  robust eigenpairs  of   the  regular simplex tensor
 of order $m$ and dimension $n-1$ with  $n\ge 3$ and $m\ge 3$,
and  	 provide a proof concerning its existence and uniqueness.
	In detail, this work shows that, up to sign equivalence,  
 no robust eigenpairs exist 	for  
 the  cases $(m,n) = (3,3), (3,4)$, or $(4,3)$.
 For the remaining combinations of $(m,n)$, 
	the only robust eigenvectors of a regular simplex tensor are the  vectors in the frame, 
	thereby confirming  the conjecture stated in 
	\cite{RobustEigen}. 
 Following  the simplification and convention  of sign equivalence, 
 our proof   first investigates   the relationship between
robust
and  locally maximized eigenpairs,
providing an auxiliary  criterion for
narrowing the scope of   robust  eigenpair.
Then,
 the  core of this work is to provide the optimization landscape of the related model.
Such a landscape analysis
will  greatly simplify
the proof   process of the focused conjecture 
 by   checking
a smaller subset  that   only  involves  locally maximized  eigenpairs.
Benefiting from such a  procedure,
 the   conjecture on  robust eigenpairs
is finally   confirmed.

Limitations  and potential implications   of this work are   discussed here for future study.
On the one hand,
 the  regular simplex tensor is a highly structured type, with
a finite number of equiangular vectors  as  its  generating  elements.
Such a special structure 
 can hardly  be satisfied in most application scenarios, thus
severely limiting   its practical utility.
On the other hand,
we also realize  that  	 the  simplex
serves as a fundamental  geometrical structure, and
is widely used in many fields.
For example,
 remote sensing image processing   always   assumes
that all pixels  lie  in the interior of a simplex \cite{psca}.
 A  similar assumption
also  holds   for  topic modeling \cite{LDA}, to  name   a few.

In this sense,
under  the premise   of not
destroying   the  inherent   simplex  structure,
how to generalize the above conclusion
 to  more general  data
 structures
 is  an interesting
issue.
We naturally
 consider   the Dirichlet distribution  whose samples
are distributed in the interior of the simplex,
indicating a strong and essential
connection with the  underlying   simplex \cite{2006Pattern}.
Actually,
in the remote sensing and topic modeling  fields
mentioned above,
 the   Dirichlet distribution
is widely  used  as   a prior to model the sampling   process.
Our preliminary  simulation experiments
have confirmed
that the eigen-structure of
 the  tensor deduced by   the   Dirichlet prior
is consistent with that of  regular simplex tensor.
We  will focus on this
bridge to provide a  rigorous    theoretical guarantee,
which is beyond the scope and length of the current 
 work, and  will be  studied  in our
future  work.

In addition, the proof strategy developed in the paper is not limited to 
 the  regular simplex tensor   only  
and can also generalize  to
other types of tensors, such as the equiangular tight frame tensor mentioned in \cite{RobustEigen},
which could be a more generalized concept than  the  regular simplex   tensor.
This may also prompt some potential investigations for future generalization  to   other tensors.
Meanwhile, we only focus on the Z-eigenpairs  over    the real field.
Whether this strategy can be further generalized  to
   E-eigenpairs  over   the complex field remains an open issue.
The other conjectures   stated   in the original reference \cite{RobustEigen}
are also worth investigating   in the future.
One can refer to Conjectures  4.9 and 5.2, Problem 5.1 for details.

\section*{Declarations}

\textbf{Acknowledgements}
{This work was supported by National Natural Science Foundation of China (Grant No. 62401088).}

\textbf{Conflict of interest}
{The authors declare that they have no conflict of interest.}

\section*{The datasets generated during and/or analyzed during the current study are available from the corresponding author on reasonable request.}

\bibliographystyle{unsrt}
\bibliography{simplexref}

\begin{appendix}

	\section{Appendix}
	In this  appendix,
	we 	 provide some important theorems   and lemmas
	used in our proof.

	\begin{wldefinition}[\textbf{direct  sum}]
		The  direct  sum  of   an   $ n \times n$  matrix  $\mathbf  {A}  $ and  an   $m \times m$ matrix   $\mathbf  {B} $  is a  matrix  
			 of size    $(n+m) \times(n+m)$,  denoted    $\mathbf  A \oplus \mathbf  B$ 	 as  follows:
		\begin{equation}
			\mathbf  {A} \oplus \mathbf  {B}=
			\left[\begin{array}{cc}
					\mathbf {A}              & \mathbf {O}_{n \times m} \\
					\mathbf {O}_{m \times n} & \mathbf {B}
				\end{array}\right] .
		\end{equation}
	\end{wldefinition}

	\begin{wllemma}(Ref \cite{zhang2017matrix})\label{direct_eig}
		Suppose  that
		$ (\lambda_{i},  \mathbf u_{i})$
		($i=1,2,\dots,n$),
		$  (\sigma_{j},  \mathbf v_{j})$
		($j=1,2,\dots,m$)
		are  eigenpairs of
		$ \mathbf A$ and $ \mathbf B$, respectively.
		The ($n+m$)  eigenpairs   of
		$ \mathbf  {A} \oplus \mathbf  {B}$
			 are
		$ (\lambda_{i},
			\begin{bmatrix}
				\mathbf u_{i} \\
				\mathbf 0_{m}
			\end{bmatrix}	)$,
		$ (\sigma_{j},
			\begin{bmatrix}
				\mathbf 0_{n} \\
				\mathbf v_{j}
			\end{bmatrix}	)$ 	($i=1,2,\dots,n$, $j=1,2,\dots,m$).
	\end{wllemma}

	\begin{wllemma}(Ref \cite{zhang2017matrix})\label{detblocklemma}
		Assume that   a  matrix  	 of size    $(n+m) \times(n+m)$,
		can be   partitioned into two blocks of size $n \times n$ and $m \times m$, respectively.
	Then, 	its  determinant  can be calculated by
		\begin{equation}\label{detblock}
			\det \left[\begin{array}{ll}
					\mathbf {A} & \mathbf {B} \\
					\mathbf {C} & \mathbf {D}
				\end{array}\right]
			=det  (\mathbf {D})
			det\left(\mathbf {A}-\mathbf {B} \mathbf {D}^{-1} \mathbf {C}\right)
		\end{equation}
		when  the  matrix  $\mathbf D$ is
		assumed to be  invertible.
	\end{wllemma}

	\begin{wllemma}(Ref \cite{zhang2017matrix})\label{AB_BA_eig}
		Given
		the    $ m \times n$    matrix
		$\mathbf A$
		and   the   $n \times m$   matrix
		$\mathbf B$ ( $m <  n $),
		assume  that
		the $m$ eigenvalues of
		$ 	\mathbf A
			\mathbf B $
		are  denoted  by
		$
			\lambda_{1}, \lambda_{2},  \dots, \lambda_{m}$,
		then  the $n$  eigenvalues  of
		$ 	\mathbf  B
			\mathbf A $
		are  given  by
 	$\lambda_1, \lambda_2, \dots, \lambda_m, 0, \dots, 0$, where 
	the multiplicity  of zero eigenvalue is $n-m$.
	\end{wllemma}

	\begin{wllemma}[\textbf{Weyl Theorem}]\label{weyltheo}
		Assume  that
		$\mathbf A$, $\mathbf B \in
			\mathbb C^{n \times n}$ are Hermitian   matrices,  and  the  eigenvalues are  sorted  
			 in    ascending  order.
		\begin{equation}
			\begin{aligned}
				\lambda_{1}(\mathbf{A})            & \le \lambda_{2}(\mathbf{A}) \le \cdots \le \lambda_{n}(\mathbf{A})                       \\
				\lambda_{1}(\mathbf{B})            & \le \lambda_{2}(\mathbf{B}) \le \cdots \le \lambda_{n}(\mathbf{B})                       \\
				\lambda_{1}(\mathbf{A}+\mathbf{B}) & \le \lambda_{2}(\mathbf{A}+\mathbf{B}) \le \cdots \le \lambda_{n}(\mathbf{A}+\mathbf{B})
			\end{aligned}
		\end{equation}

		Then,  the  following  inequalities   hold:
		\begin{equation}
			\lambda_{i}
			(\mathbf{A}+\mathbf{B})
			\ge
			\begin{cases}
				\lambda_{i}  (\mathbf{A})  +\lambda_{1}(\mathbf{B}) \\
				\lambda_{i-1}(\mathbf{A})+\lambda_{2}(\mathbf{B})   \\
				\vdots                                              \\
				\lambda_{1}(\mathbf{A})+\lambda_{i}(\mathbf{B})
			\end{cases}
			\quad
			\lambda_{i}(\mathbf A+\mathbf B)
			\le
			\begin{cases}
				\lambda_{i}(\mathbf{A})+\lambda_{n}(\mathbf B)      \\
				\lambda_{i+1}(\mathbf{A})+\lambda_{n-1}(\mathbf{B}) \\
				\vdots                                              \\
				\lambda_{n}(\mathbf{A})+\lambda_{i}(\mathbf{B})
			\end{cases}
		\end{equation}
		for  $i=1,2,\dots, n$.
	\end{wllemma}

\begin{wllemma}[Jordan--Wielandt lemma]\label{lemma.svdblock}
	Let $\mathbf  A \in \mathbb{R}^{m \times n}$ and let
$
	\sigma_1 \ge \sigma_2 \ge \cdots \ge \sigma_p \ge 0
	$
	be its singular values, where $p = \min\{m, n\}$. Consider the augmented matrix
	$
	\mathbf B = \begin{pmatrix}
	\mathbf O & \mathbf A \\
	\mathbf A^{\mathrm T} &\mathbf  O
	\end{pmatrix}.
	$
	Then the eigenvalues of $\mathbf B$ are
	\[
	-\sigma_1, \; -\sigma_2, \; \ldots, \; -\sigma_p,
	\; \underbrace{0, \ldots, 0}_{|m-n| },
	\; \sigma_p, \; \ldots, \; \sigma_2, \; \sigma_1.
	\]
\end{wllemma}

	The  following  theorem is  well  established     for   the   constrained  optimization  problem   to  identify  the  locally  optimal   solutions
	(Page 332  of     \cite{Numerical}):

	\begin{theorem}[Second-order  sufficient   condition]\label{second_order_necessary}\cite{Numerical}
For model    (\ref{opti_ori}), define  the set
		$\mathbb V $
	as
		\begin{equation}\label{vdefine}
			\mathbb V=\{
			\mathbf w \in \mathbb R^{n}
			\vert   (\triangledown  g)^{\mathrm T} \mathbf w   =0
			\}=
			\operatorname{null} [(\triangledown  g)^{\mathrm T} ]
			=
			\operatorname{null} [\mathbf A ],
		\end{equation}
		where   $  \operatorname{null}( \mathbf A) $
		denotes the null space  of $\mathbf A$  and $  \triangledown  g =\mathbf v $  denotes the gradient of the constraint: $ g (\mathbf v) =
			\mathbf v^{\mathrm T}
			\mathbf v
			-1
			=0 $.
			For the Hessian matrix defined in (\ref{hessian_matrix}),	suppose that
			for any  nonzero  vector $ \mathbf w \in \mathbb V $,
			if
			\begin{equation}\label{second_order}
			\mathbf w^{\mathrm T}
			\mathbf H (\mathbf v)
			\mathbf w
			< 0
			\end{equation}
			holds, then
			$\mathbf v $
			is a strict  local maximum solution of (\ref{opti_ori}).
	\end{theorem}

For the above conclusion, 
\cite{Numerical}  further deduced  the concept of   the  projected  Hessian  matrix (denoted 
$\mathbf P 
\in \mathbb R^{(n-1) \times  (n-1)} 
$), which   is  calculated  by  
\begin{equation}\label{Pmatrix}
\mathbf P = 
\mathbf Q_{2}^{\mathrm T}
\mathbf H(\mathbf v) 
\mathbf Q_{2},  
\end{equation}
where 
$  \mathbf Q_{2} $ is  obtained by
QR  factorization of   $     \triangledown  g $:
  \begin{equation}\label{QR_factor}
\begin{split}
\triangledown  g
&=\mathbf Q
\begin{bmatrix}
a   \\
\mathbf 0
\end{bmatrix} =
\begin{bmatrix}
\mathbf Q_{1}  &   \mathbf Q_{2}
\end{bmatrix}
\begin{bmatrix}
a  \\
\mathbf 0
\end{bmatrix}
=\mathbf Q_{1} a
\end{split}.
\end{equation} 
See Page 337 of \cite{Numerical}    for details. 
By checking   the   negative  definiteness of  $\mathbf P$, 
 i.e.,  $\mathbf P \prec 0 $, one can identify the   locally   maximized solutions. 
However,  QR  factorization  can only be  numerically  computed but   
is  not  suitable   for theoretical analysis. 
To deal with this issue, in our previous work \cite{wang2024locally}, we deduce an equivalent formula, which is stated as follows. 
	Based on  (\ref{vdefine}), we    conclude  that  $   \mathbf w$
	also  lies  in  the  orthogonal complement space of $ \mathbf  v$.
	Therefore,  $   \mathbf w$ can be linearly expressed  by the column vectors of  $\mathbf  P_{\mathbf  v } ^{\bot}$,
	checking  (\ref{second_order})
	is equivalent to  checking   the  negative semi-definiteness    of  the  matrix, denoted
	$\mathbf   K$, with the  form of
	\begin{equation}
		\notag
			\mathbf{K }  : = \mathbf {K}(\mathbf v)=
		(\mathbf  P_{\mathbf  v } ^{\bot})^{\mathrm T}    \mathbf H (\mathbf v)  \mathbf  P_{\mathbf  v }^{\bot}
		=
		(\mathbf  P_{\mathbf  v } ^{\bot})    \mathbf H (\mathbf v)  \mathbf  P_{\mathbf  v }^{\bot}
		,
	\end{equation}
	which
	corresponds to
	(\ref{Mhess}).
		 
	Due to the introduction of $ \mathbf  P_{\mathbf  v }^{\bot}$
	which  has   rank $n-1$,  
	$\mathbf{K } \in \mathbb R^{n \times n} $
			must   have at least one zero eigenvalue.  
			The conclusion in Lemma  3 of    \cite{wang2024locally}
			tells us that  
			if   $\mathbf  P \prec 0$, 
	it holds that 		$\mathbf  K  \preceq 0$ and   has   only one zero eigenvalue. 		
	This explains the statement in Lemma \ref{lem.LOCALMAX}. 
Similarly, 
			if   $\mathbf  P \succ 0$, 
		it holds that 		$\mathbf  K  \succeq 0$ and   has   only one zero eigenvalue. 

	Note that
	there  is
	only one constraint for model (\ref{opti_ori}).
	When more constraints are included, it can be similarly analyzed.
	For example,
	for    model  (\ref{optmodel}) with two constraints,  $\mathbf A$  should be  defined  as
	$
		\mathbf A =  [\triangledown  g_{1}(\mathbf u) , \triangledown  g_{2}(\mathbf u) ] =
		[\mathbf u, \mathbf 1_{n}]  \in   \mathbb R^{n  \times 2 }.
	$	
 The
	locally    optimal   solutions of   (\ref{optmodel})
	can  be  identified  by
	checking the  negative
	semi-definiteness of the   matrix   defined  in (\ref{Mhess2}) 
		 with a form of 
		$\mathbf {M} =
	\mathbf  P_{\mathbf  A } ^{\bot}   \mathbf H (\mathbf u)  \mathbf  P_{\mathbf  A }^{\bot}
	$. 
		Since $\mathbf u^{\mathrm T}\mathbf 1_{n} =0$,
	$	\mathbf  P_{\mathbf  A }^{\bot}$
 has   rank $n-2$. 
	If   $\mathbf  P \prec 0$, 
	it holds that 		$\mathbf  M  \preceq 0$ and    has    two zero eigenvalues. 		
	This is checked and discussed in Lemma \ref{Theorem_structureoflocal}
	and \ref{Theorem_structureoflocaleven1}
 to guarantee locally maximized solutions.
The  locally minimized   solutions
 	can be checked and discussed in a similar way.

\end{appendix}

\end{document}